\documentclass{article}

\usepackage{microtype}
\usepackage[nottoc]{tocbibind}
\usepackage[utf8]{inputenc}
\usepackage{graphicx}
\usepackage{float}
\usepackage{bm}
\usepackage{amsthm}
\usepackage{amsmath}
\usepackage{amssymb}
\usepackage{mathrsfs}
\usepackage{bbm}
\usepackage[all]{xy}
\usepackage{tikz-cd}
\usepackage{stmaryrd}
\usepackage{enumitem}
\usepackage{tabularx}
\usepackage{placeins}
\usepackage{pdflscape}
\usepackage{hyperref}

\DeclareMathOperator{\Ver}{\mathrm{Ver}}
\DeclareMathOperator{\Tilt}{\mathbf{Tilt}}
\DeclareMathOperator{\FPdim}{\mathrm{FPdim}}

\theoremstyle{definition}
\newtheorem{Definition}{Definition}[section]

\theoremstyle{plain}
\newtheorem{Theorem}[Definition]{Theorem}

\theoremstyle{plain}

\newcounter{mainthm}

\theoremstyle{plain}
\newtheorem{Proposition}[Definition]{Proposition}

\theoremstyle{plain}
\newtheorem{Lemma}[Definition]{Lemma}

\theoremstyle{plain}
\newtheorem{Corollary}[Definition]{Corollary}

\theoremstyle{plain}

\theoremstyle{definition}
\newtheorem{Question}[Definition]{Question}

\theoremstyle{plain}

\theoremstyle{definition}

\theoremstyle{definition}
\newtheorem{Example}[Definition]{Example}

\theoremstyle{definition}

\theoremstyle{remark}
\newtheorem{Remark}[Definition]{Remark}

\title{Higher Verlinde Categories:\ The Mixed Case}
\author{Thibault D. Décoppet}
\date{July 2024}

\begin{document}

\bibliographystyle{alpha}

\maketitle
    \hspace{1cm}
    \begin{abstract}
    Over a field of characteristic $p>0$, the higher Verlinde categories $\Ver^{\sigma}_{p^n}$ are obtained by taking the abelian envelope of quotients of the category of tilting modules for the algebraic group $\mathrm{SL}_2$. These symmetric tensor categories have been introduced in \cite{BEO, C:monoidal}, and their properties have been extensively studied in the former reference. In \cite{STWZ}, the above construction for $\mathrm{SL}_2$ has been generalized to Lusztig's quantum group for $\mathfrak{sl}_2$ and root of unity $\zeta$, which produces the mixed higher Verlinde categories $\Ver^{\zeta}_{p^{(n)}}$. Inspired by the results of \cite{BEO}, we study the properties of these braided tensor categories in detail. In particular, we establish a Steinberg tensor product formula for the simple objects of $\Ver^{\zeta}_{p^{(n)}}$, construct a braided embedding $\Ver^{\sigma}_{p^n}\hookrightarrow \Ver^{\zeta}_{p^{(n+1)}}$, identify the symmetric center of $\Ver^{\zeta}_{p^{(n)}}$ with $\Ver^{\sigma}_{p^n}$ or $\Ver^{\sigma,+}_{p^n}$ depending on the order of $\zeta$, and determine the Grothendieck ring of $\Ver^{\zeta}_{p^{(n)}}$.
\end{abstract}

\tableofcontents

\section*{Introduction}

Recently, there has been a lot of progress in our understanding of symmetric tensor categories over a field of positive characteristic. The main direction of study consists in developing a positive characteristic analogue of Deligne's theorem:\ In characteristic zero, a symmetric tensor category admits a fibre functor to $\mathrm{sVec}$, the symmetric tensor category of super vector spaces, if and only if it has moderate growth \cite{De}. In positive characteristic, a characterization of tensor categories admitting a fibre functor to $\mathrm{sVec}$ was obtained in \cite{C:fibre}. In contrast to the characteristic zero case, there exists finite semisimple symmetric tensor categories that do not admit such a fibre functor \cite{GK, GM, O4}. Over a field of characteristic $p\geq 5$, the most fundamental example is the Verlinde categories $\Ver_p$. In \cite{CEO1}, it was established that a tensor category admits a fibre functor to $\Ver_p$ if and only if it has moderate growth and is Frobenius exact (see also the related earlier work \cite{EO}). However, there is currently no result covering all tensor categories of moderate growth in positive characteristic (see, for instance, \cite{CEO2, C:polynomial, CF} for work in this direction). Namely, nested sequences of finite symmetric tensor categories $\Ver_{p^n}$ were constructed in \cite{BE} for $p=2$ and in \cite{BEO} for an arbitrary prime $p$ (see also \cite{C:monoidal}). These symmetric tensor categories are incompressible, in the sense that any braided tensor functor out of them is an embedding, so that $\Ver_{p^n}$ does not admit a fibre functor to $\Ver_{p^m}$ when $m<n$. As a consequence, the symmetric tensor categories $\Ver_{p^n}$ occupy a fundamental role in the theory of tensor categories in positive characteristic analogous to that played by $\mathrm{Vec}$ and $\mathrm{sVec}$ in characteristic zero. It is therefore particularly compelling to investigate algebraic structures defined in the categories $\Ver_{p^n}$. For instance, commutative algebras, and schemes in $\Ver_{p^n}$ were studied in \cite{C:commutative}, and \cite{C:groupscheme} respectively.

The symmetric (higher) Verlinde categories $\Ver_{p^n}$ are obtained by taking the abelian envelope of quotients of the category of tilting modules for the algebraic group $\mathrm{SL}_2$. In particular, a number of their properties are analogous to that of the category of tilting modules for the algebraic group $\mathrm{SL}_2$ as was shown in \cite{BEO}. For instance, the Frobenius functor $\mathbb{F}$, which generalizes the classical Frobenius twist for representations of an algebraic group, induces an embedding $\Ver_{p^n}\hookrightarrow \Ver_{p^{n+1}}$. Furthermore, the simple objects of $\Ver_{p^n}$ satisfy a version of Steinberg's tensor product theorem. In a slightly different direction, the symmetric tensor categories $\Ver_{p^n}$ admit lifts to characteristic zero given by the semisimple braided Verlinde categories constructed from the categories of tilting modules for quantum $\mathfrak{sl}_2$ at appropriate roots of unity. Now, over a field of characteristic $p>0$, the braided (and in fact ribbon) category of tilting modules for quantum $\mathfrak{sl}_2$ at an arbitrary root of unity $\zeta$ was extensively studied in \cite{STWZ}. Generalizing the above construction of the symmetric Verlinde categories, the authors introduced the finite ribbon tensor categories $\Ver^{\zeta}_{p^{(n)}}$, which we refer to as the mixed (higher) Verlinde categories. Our main contribution, which is presented in more detail below, is an analysis of the features of the mixed Verlinde categories. As expected the properties of these categories are reminiscent of both those of the symmetric Verlinde categories, but also of the categories of tilting modules for quantum $\mathfrak{sl}_2$ as studied for instance in \cite{AK}. For instance, assuming for simplicity that $\zeta$ has odd order, we construct an analogue of the quantum Frobenius-Lusztig twist in the form of a braided tensor functor $\Ver_{p^n}\hookrightarrow \Ver^{\zeta}_{p^{(n+1)}}$. Further, we obtain a counterpart to Steinberg's tensor product theorem. Finally, we also identify the symmetric center of the braided tensor category $\Ver^{\zeta}_{p^{(n+1)}}$ with $\Ver_{p^n}$. This last result ought to be compared to the similar characteristic zero statements obtained in \cite{Neg2} for the categories of representations of an arbitrary quantum group at a root of unity.

In characteristic zero, the semisimple Verlinde categories have found many applications to low-dimensional topology. More precisely, at even roots of unity, these categories are so-called semisimple modular tensor categories, and may therefore be used as input for the Reshetikhin-Turaev and Crane-Yetter constructions, thereby producing $3$- and $4$-dimensional TQFTs. As already hinted at in \cite{STWZ}, we suspect that the mixed higher Verlinde categories will find similar applications to low dimensional topology. More precisely, using the language of \cite{CGHPM}, we expect that the finite ribbon tensor categories $\Ver_{p^{(n)}}^{\zeta}$ are chromatic non-degenerate and therefore yield new non-compact $4$-dimensional TQFTs. Along different lines, which are related through the cobordism hypothesis, it would be interesting to study the higher algebraic properties of braided tensor categories with symmetric center $\Ver_{p^n}$. In characteristic zero, braided fusion categories were extensively studied in \cite{DGNO}. Thanks to Deligne's theorem, the general theory can in large part be reduced to studying braided fusion categories whose symmetric center is either $\mathrm{Vec}$ or $\mathrm{sVec}$. One of the main objects of interest in this direction, which was introduced in \cite{DMNO}, is the quantum Witt group $\mathcal{W}itt(\mathrm{Vec})$ associated to braided fusion categories with symmetric center $\mathrm{Vec}$. This construction can be generalized to any symmetric fusion category \cite{DNO}, and, in particular, there is a super-variant $\mathcal{W}itt(\mathrm{sVec})$, which is also of special interest. In fact, the semisimple Verlinde categories yield many interesting classes in the groups $\mathcal{W}itt(\mathrm{Vec})$ and $\mathcal{W}itt(\mathrm{sVec})$. In characteristic $p>0$, this suggests that the finite braided tensor categories whose symmetric center is $\Ver_{p^n}$ as well as the corresponding groups $\mathcal{W}itt(\Ver_{p^n})$ deserve particular attention. The mixed Verlinde categories $\Ver_{p^{(n+1)}}^{\zeta}$ are examples of such finite braided tensor categories thanks to our results, and we will argue that they give non-trivial classes in the appropriate quantum Witt groups. We will return to this general line of investigation in the future.

\subsection*{Results}

Let $\mathbbm{k}$ be an algebraically closed field of characteristic $p$, and let $\zeta$ be a root of unity in $\mathbbm{k}$. Let $\Tilt^{\zeta}$ be the category of tilting modules for the algebraic group $\mathrm{SL}_2$ if $\zeta = \pm 1$ or for Lusztig's divided power quantum group for $\mathfrak{sl}_2$ if $\zeta\neq 1$. The category $\Tilt^{\zeta}$ is monoidal, and admits a canonical spherical structure. This last fact also manifests itself in that $\Tilt^{\zeta}$ is the Cauchy completion of the Temperley-Lieb category with circle evaluating to $-(\zeta+\zeta^{-1})$. We will exploit this relation in the form of the universal property of the Temperley-Lieb category. It is well-known that a choice of square root $\zeta^{1/2}$ for $\zeta$ in $\mathbbm{k}$ allows us to endow $\Tilt^{\zeta}$ with a braiding. We denote the corresponding braided monoidal category by $\Tilt^{\zeta^{1/2}}$, which is in fact a ribbon monoidal category. Now, as was shown in \cite{STWZ}, the monoidal category $\Tilt^{\zeta}$ contains a sequence $\mathbf{J}_{p^{(n)}}$ of thick tensor ideals. Generalizing \cite{BEO}, the finite tensor category $\Ver_{p^{(n)}}^{\zeta}$ is the abelianization of the quotient $\Tilt^{\zeta}/\mathbf{J}_{p^{(n)}}$. In particular, it inherits a spherical structure. Further, upon choosing $\zeta^{1/2}$, it inherits a compatible braiding, and we write $\Ver_{p^{(n)}}^{\zeta^{1/2}}$ for the corresponding braided (in fact ribbon) tensor category. With $\zeta^{1/2}=-1$, this is exactly the finite symmetric tensor category $\Ver_{p^{n}}$ considered in \cite{BEO} equipped, in addition, with a canonical ribbon structure. Slightly more generally, it is worth pointing out that, with $\zeta = \pm 1$, $\Ver_{p^n}$ and $\Ver_{p^{(n)}}^{\zeta}$ are equivalent as plain tensor categories. We will also consider the tensor subcategory $\Ver_{p^{(n)}}^{\zeta,+}\subset\Ver_{p^{(n)}}^{\zeta}$ generated by the tilting modules with even highest weight.

Our first main result is that the canonical faithful functor $F:\Tilt^{\zeta}/\mathbf{J}_{p^{(n)}}\hookrightarrow \Ver_{p^{(n)}}^{\zeta}$ is full. Thanks to the general theory of abelian envelopes (see, for instance, \cite{BEO}), it then follows that $\Ver_{p^{(n)}}^{\zeta}$ is in fact the abelian envelope of $\Tilt^{\zeta}/\mathbf{J}_{p^{(n)}}$, that is, the functor $F$ is universal amongst faithful monoidal functors from $\Tilt^{\zeta}/\mathbf{J}_{p^{(n)}}$ to a tensor category. The usefulness of this result stems for the fact that it allows us to reduce many questions about the tensor category $\Ver_{p^{(n)}}^{\zeta}$ to generally much easier questions about $\Tilt^{\zeta}$. We wish to point out that the proof proceeds by showing that indecomposable tilting modules for quantum $\mathfrak{sl}_2$ have simple socle.

With the above result at our disposal, we go on to study the finite tensor categories $\Ver_{p^{(n)}}^{\zeta}$. Their main properties are summarized in the following omnibus theorem, which should be compared with \cite[Theorem 1.3]{BEO}. Throughout, we use the following notations, $N$ denotes the order of $\zeta$ in $\mathbbm{k}^{\times}$, $\ell=N$ if $N>2$ is odd, $\ell=N/2$ if $N > 2$ is even, and $\ell = p$ if $N\leq 2$, and $p^{(n)}=\ell p^{n-1}$ for $n\geq 1$ and $p^{(0)}=1$. Further, we define roots of unity $\sigma = (-1)^{\ell + N}$ and $\sigma^{1/2} = \zeta^{(\ell-2)\ell/2}$ in $\mathbbm{k}$.

\begin{Theorem}
The finite tensor categories $\Ver_{p^{(n)}}^{\zeta,+}\subset\Ver_{p^{(n)}}^{\zeta}$ satisfy the following properties:
\begin{enumerate}
    \item The category $\Ver_{p^{(1)}}^{\zeta}$ is semisimple. For $n\geq 2$, $\Ver_{p^{(n)}}^{\zeta,+}$ and $\Ver_{p^{(n)}}^{\zeta}$ are not semisimple.
    \item If $p>2$, the tensor category $\Ver_{p^{(n)}}^{\zeta}$ contains a unique non-trivial invertible object. Furthermore, if, in addition, $\ell$ is odd, then the (spherical) tensor category $\Ver_{p^{(n)}}^{\zeta}$ decomposes as $\Ver_{p^{(n)}}^{\zeta,+}\boxtimes\, \mathrm{Vec}^{\sigma}(\mathbb{Z}/2)$.
    \item Write $Q_k(x)$ for the $(k-1)$-th Chebyshev polynomial of the second kind, that is the integer polynomial with roots $2\cos(j\pi/k)$, for $j=1,...,k-1$. Then, the Grothendieck ring of $\Ver_{p^{(n)}}^{\zeta}$ is isomorphic to $\mathbb{Z}[x]/\big(Q_{p^{(n)}}/Q_{p^{(n-1)}}(x)\big)$. The Frobenius-Perron dimension provides a surjective homomorphism from the Grothendieck ring of $\Ver_{p^{(n)}}^{\zeta}$ onto $\mathbb{Z}[2\cos(\pi/p^{(n)})]$.
    \item The Frobenius-Perron dimension of $\Ver_{p^{(n)}}^{\zeta}$ is $p^{(n)}/(2\sin(\pi/p^{(n)}))$.
    \item The tensor categories $\Ver_{p^{(n)}}^{\zeta,+}$ with $p^{(n)}$ odd are incompressible. With $p>2$, the braided tensor categories $\Ver_{p^{(n)}}^{\zeta^{1/2}}$ for any $\ell$ and $\Ver_{p^{(n)}}^{\zeta^{1/2},+}$ for $\ell\not\equiv 2\mod 4$ are incompressible. With $p=2$, the braided tensor categories $\Ver_{p^{(n)}}^{\zeta^{1/2}}$ and $\Ver_{p^{(n)}}^{\zeta^{1/2}, +}$ are incompressible.
    \item The category $\Ver_{p^{(n)}}^{\zeta}$ admits a semisimple ribbon lift to characteristic zero given by the semisimplification of the category of tilting modules for quantum $\mathfrak{sl}_2$ at a root of unity of order $Np^{n-1}$.
    \item (Quantum Frobenius-Lusztig) There is a ribbon embedding $\mathbbm{q}\mathbb{FL}:\Ver^{\sigma^{1/2}}_{p^{n-1}}\hookrightarrow\Ver^{\zeta^{1/2}}_{p^{(n)}}$ from the symmetric Verlinde category $\Ver^{\sigma^{1/2}}_{p^{n-1}}$ into the mixed Verlinde category $\Ver^{\zeta^{1/2}}_{p^{(n)}}$.
    \item (Steinberg tensor product theorem) There exist simple objects $\mathbb{T}_{\zeta}(j)$ of $\Ver^{\zeta}_{p^{(n)}}$ with $j=0,...,\ell-1$, and $j\neq \ell-1$ if $n=1$, such that every simple object of $\Ver^{\zeta}_{p^{(n)}}$ can be uniquely written as $\mathrm{L}_{\zeta}(a)=\mathbb{T}_{\zeta}(a_0)\otimes\mathbbm{q}\mathbb{FL}\big(\mathrm{L}(b)\big)$, where $0\leq a\leq p^{(n)}-p^{(n-1)}-1$ is decomposed as $a=a_0+b\ell$ with $0\leq a_0\leq \ell-1$, $0\leq b\leq p^{n-1}-p^{n-2}-1$, and $\mathrm{L}(b)$ is the simple object of $\Ver_{p^{n-1}}$ from \cite[Theorem 1.3(8)]{BEO}.
    \item The category $\Ver_{p^{(n)}}^{\zeta}$ has $(p-1)$ blocks of sizes $1,\, (p-1),\, p(p-1),\, ...,$ $p^{n-3}(p-1)$, and $(\ell-1)$ blocks of size $p^{n-2}(p-1)$, so a total of $(n-1)(p-1)+(\ell-1)$ blocks. Moreover, if $p>2$, all blocks of the same size are equivalent, even for different $n$ and $\zeta$. If $p=2$, this holds for all blocks of size strictly greater than $1$.
    \item With $\zeta\neq \pm 1$, $p > 2$ and $n\geq 2$, the Grothendieck ring of the stable category of $\Ver_{p^{(n)}}^{\zeta}$ is isomorphic to $$GrStab(\Ver_{p^{(n)}}^{\zeta})\cong\mathbb{F}_p[z]/Q_{\ell}(z)^{p^{n-2}} \oplus\mathbb{F}_p[z,g]/(z^{\frac{p^{n-2}-1}{2}},g^2-1).$$ With $\zeta\neq \pm 1$, $p = 2$ and $n\geq 2$, the Grothendieck ring of the stable category of $\Ver_{p^{(n)}}^{\zeta}$ is isomorphic to $$GrStab(\Ver_{2^{(n)}}^{\zeta})\cong\mathbb{F}_2[z]/Q_{\ell}(z)^{2^{n-2}} \oplus\mathbb{F}_2[z]/z^{2^{n-2}-1}.$$
    \item For $\ell > 2$, $n\geq 2$ if $p>2$ and $n > 2$ if $p=2$, the tensor category $\Ver^{\sigma}_{p^{n-1}}$ is a Serre subcategory of $\Ver_{p^{(n)}}^{\zeta}$.
    \item Let $\zeta\neq\pm 1$ and $n\geq 2$. If $\zeta$ has odd order, then $\Ver_{p^{n-1}}^{\sigma^{1/2}}$ is the symmetric center of $\Ver_{p^{(n)}}^{\zeta^{1/2}}$. If $\zeta$ has even order, then $\Ver_{p^{n-1}}^{\sigma^{1/2},+}$ is the symmetric center of $\Ver_{p^{(n)}}^{\zeta^{1/2}}$.
\end{enumerate}
\end{Theorem}

\noindent In \cite[Theorem 5.8]{STWZ}, it was shown that, so long as $\zeta\neq \pm 1$, the symmetric center of $\Tilt^{\zeta^{1/2}}$ is trivial. In particular, the symmetric center of $\Tilt^{\zeta^{1/2}}/\mathbf{J}_{p^(n)}$ is trivial for $n\geq 1$. By contrast, we find that the symmetric center of $\Ver_{p^{(n)}}^{\zeta^{1/2}}$ is quite large. This apparent discrepancy is explained by the fact that all of the non-trivial objects in the symmetric center of $\Ver_{p^{(n)}}^{\zeta^{1/2}}$ are created by the process of abelianization.

\section{Preliminaries}

\subsection{Tensor Categories}\label{sub:preliminaries}

Throughout, we will work over an algebraically closed field $\mathbbm{k}$. We will use the basic theory of tensor categories as presented in \cite{EGNO}, and follow their conventions.

For instance, we say that a $\mathbbm{k}$-linear category is locally finite, also called artinian, if it is abelian, has finite dimensional $Hom$-spaces, and every object has finite length. A finite category is a locally finite category that has finitely many simple objects and enough projectives. Then, a multitensor category is a locally finite monoidal category $\mathcal{C}$ that is rigid in the sense that every object $C$ of $\mathcal{C}$ has a left dual $C^*$ and a right dual $^*C$. A tensor category is a multitensor category for which the monoidal unit $\mathbbm{1}$ is a simple object. For instance, the category of finite dimensional vector spaces, denoted by $\mathrm{Vec}_{\mathbbm{k}}$ or $\mathrm{Vec}$, is a finite tensor category. Given an object $C$ of a finite tensor category, we can consider its Frobenius-Perron dimension $\FPdim(C)\in\mathbb{C}$. This assignment produces an algebraic integer and is compatible with the tensor product \cite[Sections 3.3 \& 4.5]{EGNO}. More generally, we will also encounter Karoubian monoidal categories, that is additive monoidal categories that have splittings for idempotents. In particular, all the functors that we consider will be additive. More specifically, a tensor functor between two (multi)tensor categories is an exact $\mathbbm{k}$-linear monoidal functor. A tensor functor $F:\mathcal{C}\rightarrow\mathcal{D}$ is injective if it is fully faithful, and surjective if every simple object of $\mathcal{D}$ is a subquotient of $F(C)$ for some object $C$ in $\mathcal{C}$.

We will consider monoidal categories equipped with additional structures. Firstly, a braiding on a monoidal category $\mathcal{C}$ is a coherent family of isomorphisms $\beta_{C,D}:C\otimes D\cong D\otimes C$ for every objects $C$, $D$ of $\mathcal{C}$. A braiding is symmetric if the double braiding is trivial, that is $\beta_{D,C}\circ \beta_{C,D} = Id_{C\otimes D}$. The symmetric center $\mathcal{Z}_{(2)}(\mathcal{C})$ of a braided monoidal category $\mathcal{C}$ is the full symmetric monoidal subcategory on those objects $C$ such that $\beta_{D,C}\circ \beta_{C,D} = Id_{C\otimes D}$ for every object $D$ in $\mathcal{C}$. Secondly, given a monoidal category $\mathcal{C}$, recall that a pivotal structure consists in the data of isomorphisms $\lambda_C:C\rightarrow C^{**}$ for every object $C$ of $\mathcal{C}$ that are compatible with the monoidal structure of $\mathcal{C}$. In any rigid monoidal category, one can define the left trace $\mathrm{Tr}^L$ of any endomorphism $C\rightarrow C^{**}$ as in \cite[Section 4.7]{EGNO}. In a pivotal monoidal category, the categorical (or quantum) dimension of an object $C$ is $\mathrm{dim}(C):=\mathrm{Tr}^L(\lambda_C)$. A spherical monoidal category is a pivotal monoidal category $\mathcal{C}$ for which $\mathrm{dim}(C)=\mathrm{dim}(C^*)$ for every object $C$. Finally, a ribbon category is a braided monoidal category $\mathcal{C}$ equipped with a twist, that is isomorphisms $\theta_C:C\cong C$ such that $\theta_{C\otimes D} =(\theta_C\otimes\theta_D)\circ \beta_{D,C}\circ \beta_{C,D}$ and $\theta_{C^*}=(\theta_C)^*$ for every objects $C$ and $D$ of $\mathcal{C}$. It is well-known that, equivalently, a ribbon category is a braided monoidal category equipped with a pivotal structure such that $\lambda_{C^*} = (\lambda_C^*)^{-1}$. In particular, every ribbon category is spherical.


Additionally, we review two key concepts from \cite{BEO}. Firstly, we recall a key result from the theory of abelian envelopes (see also \cite{C:monoidal} for a different approach and \cite{CEOP} for further investigation). Given a Karoubian rigid monoidal category $\mathcal{T}$, an object $Q$ is called splitting if tensoring with $Q$ splits every morphism in $\mathcal{T}$. Splitting objects form a thick ideal $\mathcal{S}$ in $\mathcal{T}$, and we say that $\mathcal{T}$ is separated if a morphism in $\mathcal{T}$ that is annihilated by tensoring with every splitting object is necessarily zero. It was shown in \cite[Section 2.10]{BEO} that if, in addition, $\mathcal{S}$ is finitely generated as a left ideal, then there exists a multitensor category $\mathcal{C}(\mathcal{T})$ together with a faithful monoidal functor $F:\mathcal{T}\hookrightarrow\mathcal{C}(\mathcal{T})$. Furthermore, they proved that if $F$ is full, then $\mathcal{C}(\mathcal{T})$ is the abelian envelope of $\mathcal{T}$ in the sense of \cite{EAHS}, that is, $F$ is universal amongst faithful monoidal functors from $\mathcal{T}$ to a multitensor category. In particular, in this last case, if $\mathcal{T}$ is endowed with additional structure such as a braiding or a twist, then so is $\mathcal{C}(\mathcal{T})$.

Secondly, we recall the definition of a flat lift to characteristic zero from \cite[Section 2.11]{BEO}. Let $R$ be a complete discrete valuation ring with residue field $\mathbbm{k}$, and let $\mathbb{K}$ be its field of fractions. We consider $\mathcal{C}_R$, a monoidal abelian $R$-linear category, whose underlying category is equivalent to the category of representation of an $R$-algebra $A$ that is free of finite rank over $R$. Let us in addition assume that every flat object of $\mathcal{C}$, i.e.\ every object that is free as an $R$-module, admits a left and a right dual. Then, $\mathcal{C}_{\mathbbm{k}}$, the Cauchy completion of $\mathcal{C}_R\otimes_R\mathbbm{k}$ is a finite tensor category over $\mathbbm{k}$. Likewise, $\mathcal{C}_{\mathbb{K}}$, the Cauchy completion of $\mathcal{C}_R\otimes_R\mathbb{K}$ is a finite tensor category over $\mathbb{K}$. Given an arbitrary finite $\mathbbm{k}$-linear tensor category $\mathcal{C}$, a flat deformation of $\mathcal{C}$ to $R$ is a monoidal abelian category $\mathcal{C}_R$ as above and such that $\mathcal{C}_{\mathbbm{k}}\simeq\mathcal{C}$ as $\mathbbm{k}$-linear tensor categories. The generic fiber of this deformation is the finite tensor category $\mathcal{C}_{\mathbb{K}}$.

\subsection{Tilting Modules for \texorpdfstring{$\mathrm{SL}_2$}{SL2} in the Mixed Case}\label{sub:tilting}

As above $\mathbbm{k}$ is an algebraically closed field, and we will now assume that is has characteristic $p>0$. Further, we now fix $\zeta$ a primitive $N$-th root of unity in $\mathbbm{k}$. For $\zeta\neq \pm 1$, we write $U_{\zeta}=U_{\zeta}(\mathrm{SL}_2)$ for Lusztig's divided power quantum group for $\mathfrak{sl}_2$ and $\zeta$ (using the conventions of \cite{AK}). We also write $u_{\zeta}=u_{\zeta}(\mathrm{SL}_2)\subset U_{\zeta}$ for the corresponding small quantum group. For $\zeta = \pm 1$, we write $\overline{U}=hy(\mathrm{SL}_2)$ for the hyperalgebra of the simple algebraic group $\mathrm{SL}_2$, also known as its algebra of distributions. In \cite{Lus}, a surjective homomorphism of Hopf algebras $F_{\zeta}:U_{\zeta}\rightarrow \overline{U}$ was constructed. Moreover, this homomorphism identifies $\overline{U}$ with the quotient $U_{\zeta}//u_{\zeta}$.

We write $\mathbf{Mod}^{\zeta}$ for the tensor category of finite dimensional modules over $\mathbbm{k}$ of type $1$ for $\overline{U}$ if $\zeta=\pm 1$, and for $U_{\zeta}$ if $\zeta\neq \pm 1$. The simple modules $L_{\zeta}(v)$ are indexed by non-negative integers $v\in\mathbb{N}_{\geq 0}$, corresponding to their highest weight. We are most interested in another family of highest weight modules:\ The indecomposable tilting modules $\mathrm{T}_{\zeta}(v)$. We have $\mathrm{T}_{\zeta}(v)=L_{\zeta}(v)$ for $0\leq v\leq \ell-2$, where $\ell = N$ if $N > 2$ is odd, $\ell = N/2$ if $N > 2$ is even, and $\ell = p$ if $N\leq 2$.\footnote{The integer $\ell$ is the smallest positive integer $a$ for which the quantum integer $[a]_{\zeta}=\zeta^{-(a-1)}+\zeta^{-(a-3)}+...+\zeta^{a-3}+\zeta^{a-1}$ is zero in $\mathbbm{k}$.} In general, a $U_{\zeta}$-module is called tilting provided that it admits both a filtration by Weyl modules and a filtration by dual Weyl modules. We write $\Tilt^{\zeta}$ for the full subcategory of $\mathbf{Mod}^{\zeta}$ on those objects that are tilting modules, i.e.\ objects are finite direct sums of $\mathrm{T}_{\zeta}(v)$ for $v\in\mathbb{N}_{\geq 0}$. The $\mathbbm{k}$-linear category $\Tilt^{\zeta}$ is Karoubian and Krull-Schmidt. Moreover, the tensor structure $\otimes$ on $\mathbf{Mod}^{\zeta}$ restricts to a monoidal structure on $\Tilt^{\zeta}$ with monoidal unit $\mathbf{1}=\mathrm{T}_{\zeta}(0)=L_{\zeta}(0)$, and for which objects have duals. We also note that $\Tilt^{\zeta}=\Tilt^{-\zeta}$ because the Hopf algebras $U_{\zeta}$ and $U_{-\zeta}$ are isomorphic. 
Similarly, we also have $\Tilt^{\zeta}=\Tilt^{\zeta^{-1}}$. 
For more details on tilting modules in the mixed case, we refer the reader to \cite{AK} and \cite{And} (see also \cite{STWZ} for a summary).

When $\sigma = \pm 1$, we will simply write $\mathrm{T}(v)$ and $L(v)$ for the indecomposable tilting modules and the simple modules of $\mathbf{Mod}^{\sigma}$. We now set $\sigma=(-1)^{N+\ell}$, then the surjective homomorphism of Hopf algebras $F_{\zeta}:U_{\zeta}\rightarrow \overline{U}$ induces a fully faithful tensor functor $(-)^{[q]}: \mathbf{Mod}^{\sigma}\rightarrow \mathbf{Mod}^{\zeta}$, which we refer to as the quantum Frobenius-Lusztig twist. The following result is well-known.

\begin{Proposition}[\cite{And}]\label{prop:mixedDonkin}
For $\ell-1\leq a\leq 2\ell -2$, and any non-negative integer $b$, we have $$\mathrm{T}_{\zeta}(a + \ell b) = \mathrm{T}_{\zeta}(a)\otimes \mathrm{T}(b)^{[q]}$$ in $\Tilt^{\zeta}$.
\end{Proposition}

\noindent We will also make use of various other formulas from \cite{STWZ} for the tensor products of tilting modules. These are most conveniently expressed using the notion of $p\ell$-adic expansion. We set $p^{(0)}=1$, and for any integer $n\geq 1$, we define $p^{(n)}= p^{n-1}\ell$. Any non-negative integer $a$ admits a unique expansion $[a_k,...,a_0]_{p\ell} : =\sum_{i=0}^ka_ip^{(i)}  = a$ with $a_k\neq 0$, $0\leq a_i \leq p-1$ for $i>0$, and $0\leq a_0\leq \ell-1$. Conversely, any tuple $(b_k,...,b_0)$ of potentially negative integers defines an integer $[b_k,...,b_0]_{p\ell} : =\sum_{i=0}^kb_ip^{(i)}$.

Following \cite[Section 2B]{STWZ}, the monoidal category $\mathbf{Mod}^{\zeta}$ and therefore also $\Tilt^{\zeta}$ can be equipped with both a spherical structure and a braiding, which are compatible and therefore yield a ribbon structure. Firstly, following \cite{BK}, we can use $\zeta$ to define an pivot in $U_{\zeta}$, or $\overline{U}$, providing us with a pivotal structure on $\mathbf{Mod}^{\zeta}$. It is clear that this pivotal structure is in fact spherical. Secondly, upon choosing a square root $\zeta^{1/2}$ for $\zeta$ in $\mathbbm{k}$, we can consider the $\mathrm{R}$-matrix from \cite[Section IX.7]{Kas}, and thereby endow the category $\mathbf{Mod}^{\zeta}$ with a braiding. These two structures are compatible and therefore yield a ribbon structure. In particular, in what follows, we write $\Tilt^{\zeta,\zeta^{1/2}}$, or more succinctly $\Tilt^{\zeta^{1/2}}$ when no confusion can arise, for the ribbon monoidal category of tilting modules for $U_{\zeta}$ with braiding corresponding to $\zeta^{1/2}$.

\begin{Remark}\label{rem:standardtilting}
Using our conventions, the standard symmetric structure on the monoidal category of tilting modules for $\mathrm{SL}_2$ is recovered by taking $\zeta = +1$ and $\zeta^{1/2} = -1$. If $p\neq 2$, taking $\zeta^{1/2} = +1$ yields a distinct symmetric structure. On the other hand, still working under the assumption that $p\neq 2$ and taking $\zeta = -1$, the two choices of $\zeta^{1/2}$ yield braidings on $\Tilt^{-1}$ that are not symmetric. It is straightforward to check that the symmetric center of $\Tilt^{-1,(-1)^{1/2}}$ contains precisely those indecomposable tilting modules that have even highest weight.
\end{Remark}

The ribbon category $\mathbf{Tilt}^{\zeta}$ of tilting modules for $U_{\zeta}$ is related to the so-called Temperley-Lieb category. More precisely, let us write $\mathbf{TL}^{\zeta}$ for the free spherical monoidal category on one self-dual object $X$ of quantum dimension $-[2]_{\zeta} = -(\zeta^{-1}+\zeta)$. By definition, $\mathbf{TL}^{\zeta}$ admits a particularly convenient graphical calculus reviewed, for instance, in \cite[Section 2.B]{STWZ}. By virtue of its universal property, there is a pivotal functor $E:\mathbf{TL}^{\zeta}\rightarrow \mathbf{Tilt}^{\zeta}$ sending $X$ to $\mathrm{T}_{\zeta}(1)$. The functor $E$ becomes an equivalence upon Cauchy completion (see \cite[Prop. 2.20]{STWZ} and references therein). Moreover, upon choosing a square root $\zeta^{1/2}$ for $\zeta$, one can use Kauffman's skein relation to endow the Temperley-Lieb category with a compatible braiding, turning it into a ribbon category. The functor $E$ intertwines this ribbon structure with the ribbon structure on $\Tilt^{\zeta^{1/2}}$ (see \cite[Prop. 2.21]{STWZ}).

For later use, we record the following result, which is presumably well-known, but, given that we have not been able to locate a proof in the literature, we include one for completeness.

\begin{Lemma}\label{lem:braidingsTilting}
Let $\zeta$ be arbitrary in $\mathbbm{k}^{\times}$. The equivalence classes of ribbon structures on $\Tilt^{\zeta}$ that extend the canonical spherical structure are given by the square roots of $\zeta^{\pm 1}$ in $\mathbbm{k}$. The corresponding braidings are specified by Kauffman's skein relation.
\end{Lemma}
\begin{proof}
Let $\beta$ be a braiding on $\Tilt^{\zeta}$ that is compatible with the spherical structure. We begin by observing that the braiding $\beta$ is completely determined by its value $\beta_{\mathrm{T}_{\zeta}(1),\mathrm{T}_{\zeta}(1)}$. Namely, $\Tilt^{\zeta}$ is the Cauchy completion of the Temperley-Lieb category, which is generated by the tensor powers of $\mathrm{T}_{\zeta}(1)$.

Let us write $\mathrm{coev}:\mathrm{T}_{\zeta}(1)\otimes\mathrm{T}_{\zeta}(1)\rightarrow \mathbf{1}$ and $\mathrm{ev}:\mathbf{1}\rightarrow\mathrm{T}_{\zeta}(1)\otimes\mathrm{T}_{\zeta}(1)$ for the canonical evaluation and coevaluation morphisms, i.e.\ $\mathrm{ev}\circ\mathrm{coev} = - [2]_{\zeta}$. Then, we have $\beta_{\mathrm{T}_{\zeta}(1),\mathrm{T}_{\zeta}(1)} = \lambda \cdot Id + \mu \cdot (\mathrm{coev}\circ\mathrm{ev}):\mathrm{T}_{\zeta}(1)\otimes\mathrm{T}_{\zeta}(1)\rightarrow \mathrm{T}_{\zeta}(1)\otimes\mathrm{T}_{\zeta}(1)$ for some scalars $\lambda,\mu\in\mathbbm{k}$.

Now, the braiding $\beta$ must be natural and must satisfy the hexagon equations. In particular, we must have the following equality $$(\mathrm{ev}\otimes \mathrm{T}_{\zeta}(1))\circ (\mathrm{T}_{\zeta}(1)\otimes\beta_{\mathrm{T}_{\zeta}(1),\mathrm{T}_{\zeta}(1)})\circ (\beta_{\mathrm{T}_{\zeta}(1),\mathrm{T}_{\zeta}(1)}\otimes\mathrm{T}_{\zeta}(1)) = \mathrm{T}_{\zeta}(1)\otimes\mathrm{ev}.$$ Expanding the left hand-side, we find that we must have $$\lambda\mu = 1\ \mathrm{and}\ \lambda^2 - \lambda\mu[2]_{\zeta} + \mu^2 = 0$$ in $\mathbbm{k}$. Combining these equations together, we find that $\lambda = \mu^{-1}$ and $\lambda^2 = \zeta^{\pm 1}$. This finishes the proof.
\end{proof}

\section{Properties of Tilting Modules for \texorpdfstring{$\mathrm{SL}_2$}{SL2} in the Mixed Case}

We work over the algebraically closed field $\mathbbm{k}$ of characteristic $p>0$, and fix $\zeta$ a primitive $N$-th root of unity in $\mathbbm{k}$. We will freely use the notations of the previous section.

\subsection{Socles of Tilting Modules for \texorpdfstring{$\mathrm{SL}_2$}{SL2}}

We prove a technical result about indecomposable tilting modules considered as objects of $\mathbf{Mod}^{\zeta}$. It is a generalization to the mixed case of \cite[Lemma 3.6]{BEO}.

\begin{Lemma}\label{lem:3.6}
Assume that $Hom_{U_{\zeta}}(\mathbf{1},\mathrm{T}_{\zeta}(s))$ is non-zero. Then, $s=2p^{(k)}-2$ for some non-negative integer $k$, and the socle of $\mathrm{T}_{\zeta}(s)$ is exactly $\mathbf{1}=\mathrm{T}_{\zeta}(0)$.
\end{Lemma}

\begin{proof}
If there is a non-zero morphism $\mathrm{T}_{\zeta}(0)\rightarrow\mathrm{T}_{\zeta}(s)$, then $\mathrm{T}_{\zeta}(0)=\mathbf{1}$ is a Weyl factor of $\mathrm{T}_{\zeta}(s)$. It then follows from \cite[Proposition 3.3]{STWZ} that $s=2p^{(k)}-2$ for some non-negative integer $k$. This establishes the first part of the claim. The statement about the socle follows from the more general result below.
\end{proof}

The proof of the next result is inspired by an argument of Donkin covering the case $\zeta=\pm 1$, which is recorded in \cite[Proposition 6.2.4]{Mar}.

\begin{Proposition}\label{prop:simplesocle}
Every indecomposable tilting module for $U_{\zeta}$ has simple socle.
\end{Proposition}

\begin{proof}
It is enough to consider $\zeta\neq\pm 1$. For $0\leq r\leq \ell-1$, the tilting modules $\mathrm{T}_{\zeta}(r)$ are simple.

For $\ell \leq r\leq 2\ell -2$. We claim that, as a $U_{\zeta}$-module, $\mathrm{T}_{\zeta}(r)$ has composition series given by $[\mathrm{T}_{\zeta}(2\ell-2-r),L_{\zeta}(r),\mathrm{T}_{\zeta}(2\ell-2-r)]$. This follows from the fact that $\mathrm{T}_{\zeta}(r)$ is indecomposable and self-dual together with the knowledge of its simple factors, which can be derived from \cite[Proposition 3.3]{STWZ}.

Let us now assume that $2\ell-2\leq r$, and write $r=s+\ell t$ for some integer $\ell-1\leq s\leq 2\ell-2$. Thanks to proposition \ref{prop:mixedDonkin}, we find that $\mathrm{T}_{\zeta}(r)=\mathrm{T}_{\zeta}(s)\otimes\mathrm{T}(t)^{[q]}$. But, upon restriction along the inclusion $u_{\zeta}\subset U_{\zeta}$, the simple $U_{\zeta}$-module $L_{\zeta}(2\ell-2-s)$ remains simple by \cite[Theorem 1.9]{AK}. In particular, the socle of $\mathrm{T}_{\zeta}(r)$ as a $u_{\zeta}$-module is given by $L_{\zeta}(2\ell-2-s)\otimes \mathrm{T}(t)^{[q]}$.

Let $v$ be any integer, and write $v=w +\ell x$ with $0\leq w\leq \ell-1$. It follows from \cite[Theorem 1.10]{AK} that $L_{\zeta}(v)=L_{\zeta}(w)\otimes L(x)^{[q]}$ as $U_{\zeta}$-modules. Then, if the simple $U_{\zeta}$-module $L_{\zeta}(w)$ is in $\mathrm{Soc}_{U_{\zeta}}(\mathrm{T}_{\zeta}(r))$, we must have $w=2\ell-2-s$ by considering the restrictions to $u_{\zeta}$-modules. As a consequence, we find $$Hom_{U_{\zeta}}(L_{\zeta}(v),\mathrm{T}_{\zeta}(r))=Hom_{U_{\zeta}}(L_{\zeta}(2\ell-2-s)^*\otimes L_{\zeta}(2\ell-2-s), \big(L(x)^*\otimes \mathrm{T}(t)\big)^{[q]}).$$ But, by the formulas in \cite[Lemma 4.1]{STWZ}, we have that $L_{\zeta}(2\ell-2-s)^*\otimes L_{\zeta}(2\ell-2-s)$ is a finite direct sum of tilting modules for $U_{\zeta}$ of highest weight between $0$ and $2\ell-2$. More precisely, $\mathbf{1}=L_{\zeta}(0)$ occurs as a direct summand exactly once if $s\neq \ell-1$, and does not occur otherwise. On the other hand, if $s= \ell-1$, $\mathrm{T}_{\zeta}(2\ell-2)$ occurs as a direct summand exactly once. We therefore find that, for any $s$, there is precisely one direct summand of $L(2\ell-2-s)^*\otimes L(2\ell-2-s)$ that has a composition factor on which $u_{\zeta}$ acts trivially. Furthermore, the head of the corresponding direct summand is $\mathbf{1}$, which implies that \begin{equation}\label{eqn:quantumclassicalfactorization}Hom_{U_{\zeta}}(L_{\zeta}(v),\mathrm{T}_{\zeta}(r))=Hom_{U_{\zeta}}(L(x)^{[q]},\mathrm{T}(t)^{[q]})=Hom_{\overline{U}}(L(x),\mathrm{T}(t)),\end{equation} where the second equality uses that the quantum Frobenius-Lusztig $(-)^{[q]}$ functor is fully faithful as it arises from restricting along the surjective algebra homomorphism $F_{\zeta}:U_{\zeta}\rightarrow\overline{U}$. The above argument shows that $$\mathrm{Soc}_{U_{\zeta}}(\mathrm{T}_{\zeta}(r)) = L_{\zeta}(2\ell-2-s)\otimes\big(\mathrm{Soc}_{\overline{U}}(\mathrm{T}(t))\big)^{[q]}.$$ Finally, by \cite[Proposition 6.2.4]{Mar}, the socle of $\mathrm{T}(t)$ is simple, so that it follows from \cite[Proposition 3.4(b)]{STWZ} that $\mathrm{Soc}_{U_{\zeta}}(\mathrm{T}_{\zeta}(r))$ is simple as a $U_{\zeta}$-module.
\end{proof}

\subsection{Tensor Ideals and Separatedness}

We now analyze the structure of the tensor category $\Tilt^{\zeta}$, its tensor ideals, and their quotients. We begin by recalling two results established in \cite[Sections 5]{STWZ}.

\begin{Theorem}[\cite{STWZ}]\label{thm:tensorideals}
Every (not necessarily thick) tensor ideal of $\Tilt^{\zeta}$ is of the form $\mathbf{J}_{p^{(n)}}:=\{\mathrm{T}_{\zeta}(v)| v\geq p^{(n)}-1\}^{\oplus}$ for some non-negative integer $n$.
\end{Theorem}

We then define the Karoubian monoidal category $$\Tilt^{\zeta}_{p^{(n)}}:=\Tilt^{\zeta}/\mathbf{J}_{p^{(n)}}.$$

\begin{Lemma}[\cite{STWZ}]\label{lem:separated}
The Karoubian monoidal category $\Tilt^{\zeta}_{p^{(n)}}$ is separated. Moreover, its ideal of splitting objects is $\mathbf{J}_{p^{(n-1)}}/\mathbf{J}_{p^{(n)}}$.
\end{Lemma}

Let us write $\Tilt^{\zeta}_{<p^{(n)}-1}$ for the full additive subcategory of $\Tilt^{\zeta}_{p^{(n)}}$ generated by the tilting modules $\mathrm{T}_{\zeta}(i)$ with $0\leq i<p^{(n)}-1$.

\begin{Proposition}\label{prop:Tiltingsubcategoryidentification}
The composite $\Tilt^{\zeta}_{<p^{(n)}-1}\hookrightarrow\Tilt^{\zeta}\rightarrow \Tilt^{\zeta}_{p^{(n)}}$ is an equivalence of additive categories.
\end{Proposition}
\begin{proof}
The proof follows \cite[Proposition 3.5]{BEO} using lemma \ref{lem:3.6} above. Given that this is an important result, we provide details for the reader's convenience.

It follows from the definition that the composite functor is full and essentially surjective, it therefore suffices to argue that it is faithful. More precisely, we have to show that no nonzero morphism in $\Tilt^{\zeta}_{<p^{(n)}-1}$ factors through $\mathbf{J}_{p^{(n)}}$. Towards a contradiction, assume that $g:\mathrm{T}_{\zeta}(v)\rightarrow \mathrm{T}_{\zeta}(w)$ with $0\leq v,w<p^{(n)}-1$ factors through $\mathbf{J}_{p^{(n)}}$. Then, the corresponding morphism $f:\mathbf{1}\rightarrow \mathrm{T}_{\zeta}(v)^*\otimes \mathrm{T}_{\zeta}(w)$ under the isomorphism $$Hom_{U_{\zeta}}(\mathrm{T}_{\zeta}(v),\mathrm{T}_{\zeta}(w))\cong Hom_{U_{\zeta}}(\mathbf{1},\mathrm{T}_{\zeta}(v)^*\otimes\mathrm{T}_{\zeta}(w))$$ also factors through $\mathbf{J}_{p^{(n)}}$. In particular, this means that there exists $s\geq p^{(n)}-1$ and morphisms $\mathbf{1}\rightarrow \mathrm{T}_{\zeta}(s)$ and $\mathrm{T}_{\zeta}(s)\rightarrow \mathrm{T}_{\zeta}(v)^*\otimes\mathrm{T}_{\zeta}(w)$ whose composite is non-zero. Thanks to lemma \ref{lem:3.6}, we find that not only we must have $s=2p^{(k)}-2$ for some $k\geq n$, but also that the morphism $\mathrm{T}_{\zeta}(s)\rightarrow \mathrm{T}_{\zeta}(v)^*\otimes\mathrm{T}_{\zeta}(w)$ is injective. This is not possible given that the highest weight of $\mathrm{T}_{\zeta}(v)^*\otimes\mathrm{T}_{\zeta}(w)$ is $v+w < 2p^{(n)}-2\leq s$.
\end{proof}

\subsection{Additional Properties}

The following technical lemmas will also be used below.

\begin{Lemma}\label{lem:tiltingmodulesembedding}
Let $p^{(k-1)}-1\leq i\leq p^{(k)}-2$, then there exists $p^{(k)}-1\leq j\leq p^{(k+1)}-2$ such that $\mathrm{T}_{\zeta}(i)$ embeds into $\mathrm{T}_{\zeta}(j)$.
\end{Lemma}
\begin{proof}
The case $k=1$ follows from the proof of proposition \ref{prop:simplesocle}. Let us therefore assume that $1<k<n$. We can therefore write $i=a+\ell b$ with $\ell-1 \leq a\leq 2\ell -2$ and a non-negative integer $b$. Observe that $p^{k-2}-1\leq b\leq p^{k-1}-2$. Appealing to proposition \ref{prop:mixedDonkin}, we have that $\mathrm{T}_{\zeta}(i)=\mathrm{T}_{\zeta}(a)\otimes \mathrm{T}(b)^{[q]}$. By \cite[Lemma 4.4]{BEO}, which is the positive characteristic version of the present lemma, there exists an embedding $\mathrm{T}(b)\hookrightarrow\mathrm{T}(c)$ for some $p^{k-1}-1\leq c\leq p^{k}-2$. We therefore have an embedding $$\mathrm{T}_{\zeta}(i)=\mathrm{T}_{\zeta}(a)\otimes \mathrm{T}(b)^{[q]}\hookrightarrow\mathrm{T}_{\zeta}(a)\otimes \mathrm{T}(c)^{[q]}=\mathrm{T}_{\zeta}(a+\ell c)$$ with $p^{(k)}-1\leq a+\ell c\leq p^{(k+1)}-2$ as desired.
\end{proof}

\begin{Lemma}\label{lem:precisesummand}
Let $0\leq i\leq \ell-1$ and $m=\ell r+\ell-2$ for some non-negative integer $r$. Then, $\mathrm{T}_{\zeta}(i)\otimes\mathrm{T}_{\zeta}(k)$ contains $\mathrm{T}_{\zeta}(m)$ as a direct summand if and only if $k=m-i$, in which case it appears exactly once.
\end{Lemma}
\begin{proof}
By considering the highest weight of $\mathrm{T}_{\zeta}(i)\otimes\mathrm{T}_{\zeta}(m-i)$, we see that $\mathrm{T}_{\zeta}(m)$ manifestly appears as a direct summand with multiplicity precisely one. It therefore only remains to prove the converse. We will do so by using induction on $i$. The case $i=0$ is obvious. Let us therefore take $i=1$. In this case, it follows from \cite[Proposition 4.7]{STWZ} that $\mathrm{T}_{\zeta}(m)$ may appear as a direct summand of $\mathrm{T}_{\zeta}(1)\otimes\mathrm{T}_{\zeta}(k)$ either if $k+1=m$ or if both $k-1=m$ and $k+1\not\equiv 0\ (\textrm{mod}\ \ell)$. But, $m\equiv \ell-2\ (\textrm{mod}\ \ell)$, so the later case does not occur.

Suppose that the result is known for $i-1\geq 1$. Let us assume that $\mathrm{T}_{\zeta}(m)$ is a direct summand of $\mathrm{T}_{\zeta}(i)\otimes\mathrm{T}_{\zeta}(k)$. It follows that $\mathrm{T}_{\zeta}(m)$ is a direct summand of $\mathrm{T}_{\zeta}(1)\otimes\mathrm{T}_{\zeta}(i-1)\otimes\mathrm{T}_{\zeta}(k)$. Namely, we have that $\mathrm{T}_{\zeta}(1)\otimes\mathrm{T}_{\zeta}(i-1)=\mathrm{T}_{\zeta}(i)\oplus\mathrm{T}_{\zeta}(i-2)$ given that $2\leq i\leq \ell-1$ by \cite[Lemma 4.1]{STWZ}. Then, by comparing highest weights, we find that $\mathrm{T}_{\zeta}(m-i+1)$ must be a direct summand of $\mathrm{T}_{\zeta}(1)\otimes\mathrm{T}_{\zeta}(k)$. But, by \cite[Proposition 4.7]{STWZ}, this is only possible if either $k=m-i$ or $k=m-i+2$. Furthermore, in either case, $\mathrm{T}_{\zeta}(m-i+1)$ appears with multiplicity one because we have $2\leq i\leq \ell-1$, so that we cannot have both $k+1\equiv\ell-1\ (\textrm{mod}\ \ell)$ and $k+1\equiv m-i+2\equiv -i\ (\textrm{mod}\ \ell)$. It therefore only remains to exclude the case when $k=m-i+2$. In order to see this, note that the inductive hypothesis implies that $\mathrm{T}_{\zeta}(i-2)\otimes\mathrm{T}_{\zeta}(m-i+2)$ contains $\mathrm{T}_{\zeta}(m)$ as a direct summand with multiplicity precisely one. But, it follows from our discussion above that $\mathrm{T}_{\zeta}(1)\otimes\mathrm{T}_{\zeta}(i-1)\otimes\mathrm{T}_{\zeta}(k)=(\mathrm{T}_{\zeta}(i)\oplus\mathrm{T}_{\zeta}(i-2))\otimes\mathrm{T}_{\zeta}(k)$ contains $\mathrm{T}_{\zeta}(m)$ with multiplicity exactly one. Therefore, $\mathrm{T}_{\zeta}(m)$ is not a direct summand of $\mathrm{T}_{\zeta}(i)\otimes\mathrm{T}_{\zeta}(m-i+2)$ as desired.
\end{proof}

\subsection{Temperley-Lieb Objects}

Let $\mathcal{C}$ be an arbitrary monoidal $\mathbbm{k}$-linear category.

\begin{Definition}
A $\zeta$-Temperley-Lieb object in $\mathcal{C}$ is an object $X$ of $\mathcal{C}$ admitting a left dual $X^*$ together with an isomorphism $\phi:X\xrightarrow{\sim}X^*$ satisfying $$\mathrm{Tr}^L((\phi^{-1})^*\circ\phi)=-[2]_{\zeta} = -(\zeta^{-1}+\zeta).$$
\end{Definition}

Associated to any $\zeta$-Temperley-Lieb object $X$ in $\mathcal{C}$, there are maps $\mathrm{coev}:\mathbbm{1}\rightarrow X\otimes X$ and $\mathrm{ev}:X\otimes X\rightarrow \mathbbm{1}$ defined using $\phi$ and $\phi^{-1}$ that satisfy the snake equations and such that $\mathrm{ev}\circ\mathrm{coev}=-[2]_{\zeta}$.

Let now $\mathcal{B}$ be a braided monoidal $\mathbbm{k}$-linear category, and let us also fix a square root $\zeta^{1/2}$ of $\zeta$.

\begin{Definition}
A $\zeta$-Temperley-Lieb object in $\mathcal{B}$ is braided if $\beta_{X,X}$, the self-braiding on $X$, satisfies the equation $$\beta_{X,X} = \zeta^{1/2}\cdot Id_X + \zeta^{-1/2}\cdot (\mathrm{coev}\circ \mathrm{ev}).$$
\end{Definition}

It follows from the universal property of the Temperley-Lieb category $\mathbf{TL}^{\zeta}$ that its generator $X$ with $\phi = Id_X$ is the universal example of a $\zeta$-Temperley-Lieb object. Further, if we consider the braided monoidal category $\mathbf{TL}^{\zeta^{1/2}}$ with braiding given by the Kauffman bracket, the $\zeta$-Temperley-Lieb object $X$ is braided. In fact, the tilting module $\mathrm{T}_{\zeta}(1)$ in $\mathbf{Tilt}^{\zeta}$ is also an example of a braided $\zeta$-Temperley-Lieb object. In particular, the next result, which is a reformulation of \cite[Theorem 2.4]{O3}, is a consequence of the fact that $\mathbf{Tilt}^{\zeta}$ is the Cauchy completion of $\mathbf{TL}^{\zeta}$.

\begin{Proposition}[\cite{O3}]\label{prop:braidedTLobjects}
For any Karoubian monoidal $\mathbbm{k}$-linear category $\mathcal{C}$, the category of monoidal functors $F:\Tilt^{\zeta}\rightarrow \mathcal{C}$ is equivalent to the category of $\zeta$-Temperley-Lieb objects in $\mathcal{C}$. For any Karoubian braided monoidal $\mathbbm{k}$-linear category $\mathcal{B}$, the category of braided monoidal functors $F:\Tilt^{\zeta^{1/2}}\rightarrow \mathcal{B}$ is equivalent to the category of braided $\zeta$-Temperley-Lieb objects in $\mathcal{B}$.
\end{Proposition}

\begin{Remark}\label{rem:TLobjectspherical}
Let us momentarily assume that the Karoubian monoidal $\mathbbm{k}$-linear category $\mathcal{C}$ is equipped with a spherical structure $\lambda$. Further, let $X$ be a $\zeta$-Temperley-Lieb object in $\mathcal{C}$ with isomorphism $\phi:X\rightarrow X^*$. If these two structures are compatible in the sense that $\lambda_X=(\phi^{-1})^*\circ\phi:X\rightarrow X^{**}$, then the monoidal functor $F:\Tilt^{\zeta}\rightarrow\mathcal{C}$ corresponding to $X$ via proposition \ref{prop:braidedTLobjects} is in fact spherical. A similar observation holds for braided $\zeta$-Temperley-lieb object in $\mathcal{B}$, when this braided monoidal category is additionally equipped with a ribbon structure.
\end{Remark}

For any integer $k$, let $Q_{k}$ be the $(k-1)$-th Chebyshev polynomial of the second kind defined by $$Q_{k}(2\cos(x))=\frac{\sin(kx)}{\sin(x)}.$$ In particular, the roots of $Q_k(x)$ are given by $2\cos(j\pi/k)$ for $j=1,...,k-1$. Alternatively, these polynomials may also be defined recursively as follows: \begin{equation}\label{eq:recursionChebychev}
Q_1(x)=1,\ Q_2(x)=x,\ Q_{k+1}(x)=xQ_{k}(x) - Q_{k-1}(x).
\end{equation} The polynomials $Q_{k}$ have integer coefficients, and we can therefore write $Q_{k}=Q_{k}^+-Q_{k}^-$ for two unique integer polynomials $Q_{k}^+$ and $Q_{k}^-$ with non-negative coefficients.

\begin{Lemma}
In the category $\Tilt^{\zeta}$, there is an isomorphism $$\mathrm{T}_{\zeta}(p^{(n)}-1)\oplus Q_{p^{(n)}}^-(\mathrm{T}_{\zeta}(1))\cong Q_{p^{(n)}}^+(\mathrm{T}_{\zeta}(1)).$$
\end{Lemma}
\begin{proof}
Firstly, recall that characters of tilting modules are linearly independent. It is therefore enough to establish that the above equality holds at the level of characters. By \cite[Proposition 3.3]{STWZ}, the character of $\mathrm{T}_{\zeta}(p^{(n)}-1)$ is $[p^{(n)}]_{t}$, with $t$ a formal variable. But, the character of $\mathrm{T}_{\zeta}(1)$ is $[2]_{t}=t^{-1}+t$, so, using equation \eqref{eq:recursionChebychev}, it is straightforward to show that for every integer $k$, we have $$Q_{k}(t^{-1}+t)=[k]_{t}=t^{-k+1}+...+t^{k-1}.$$ Taking $k=p^{(n)}$ concludes the proof.
\end{proof}

In particular, there is a split injection $\iota_n:Q_{p^{(n)}}^-(\mathrm{T}_{\zeta}(1))\hookrightarrow Q_{p^{(n)}}^+(\mathrm{T}_{\zeta}(1)),$ in $\Tilt^{\zeta}$, which becomes an isomorphism in $\Tilt^{\zeta}_{p^{(n)}}=\Tilt^{\zeta}/\mathbf{J}_{p^{(n)}}$.

\begin{Definition}
A braided $\zeta$-Temperley-Lieb object $X$ in $\mathcal{B}$ has degree $n\in\mathbb{N}$ if $n$ is the smallest integer for which $\iota_n$ is an isomorphism in $\mathcal{B}$. Otherwise, $X$ has degree $\infty$.
\end{Definition}

\begin{Proposition}\label{prop:nTLobjects}
Let $\mathcal{B}$ be a Karoubian braided monoidal $\mathbbm{k}$-linear category, then, evaluation on $\mathrm{T}_{\zeta}(1)$ yields an equivalence between the category of braided monoidal functors $F:\Tilt^{\zeta}_{p^{(n)}}\rightarrow \mathcal{B}$ and that of braided $\zeta$-Temperley-Lieb objects of degree less than $n$ in $\mathcal{B}$. Moreover, $F$ is faithful if and only if the corresponding $\zeta$-Temperley-Lieb object has degree precisely $n$.
\end{Proposition}
\begin{proof}
The proof follows from the combination of proposition \ref{prop:braidedTLobjects} with theorem \ref{thm:tensorideals}.
\end{proof}

\section{Mixed Higher Verlinde Categories}

We work over an algebraically closed field $\mathbbm{k}$ of characteristic $p>0$. We also fix a primitive $N$-th root of unity $\zeta$ in $\mathbbm{k}$, as well as a positive integer $n$. We write $\Ver^{\zeta}_{p^{(n)}}$ for the finite tensor category that is the abelianization of $\Tilt^{\zeta}_{p^{(n)}}=\Tilt^{\zeta}/\mathbf{J}_{p^{(n)}}$, which exists thanks to the results of \cite{BEO,C:monoidal} recalled at the end of section \ref{sub:preliminaries} together with lemma \ref{lem:separated}. These tensor categories were first introduced in \cite[Section 5.A]{STWZ}. We use $F:\Tilt^{\zeta}_{p^{(n)}}\rightarrow \Ver^{\zeta}_{p^{(n)}}$ to denote the canonical faithful monoidal functor. 

We can consider the Karoubian monoidal subcategories $\Tilt^{\zeta, +}_{p^{(n)}}\subset \Tilt^{\zeta}_{p^{(n)}}$ generated by $\mathrm{T}_{\zeta}(i)$ with $i$ even, and their abelianizations $\Ver^{\zeta, +}_{p^{(n)}}\subseteq \Ver^{\zeta}_{p^{(n)}}$. Provided that $p^{(n)}>2$, the subcategory $\Ver^{\zeta, +}_{p^{(n)}}$ is non-trivial.

\subsection{Completeness}\label{sub:completeness}

Our proof of the following theorem follows closely \cite[Subsection 4.2.2]{BEO}. Given the central role that this result plays, we include an extensive outline of the proof, and we refer the reader to \cite{BEO} for the remaining details.

\begin{Theorem}\label{thm:completeness}
The category $\Tilt^{\zeta}_{p^{(n)}}$ is complete, i.e.\ the canonical faithful functor $F:\Tilt^{\zeta}_{p^{(n)}}\rightarrow \Ver^{\zeta}_{p^{(n)}}$ is full. In particular, $\Ver^{\zeta}_{p^{(n)}}$ is the abelian envelope of $\Tilt^{\zeta}_{p^{(n)}}$.
\end{Theorem}
\begin{proof}
By definition, we have to show that for any objects $X,Y$ in $\Tilt^{\zeta}_{p^{(n)}}$, the injective map $$Hom_{\Tilt^{\zeta}_{p^{(n)}}}(X,Y)\rightarrow Hom_{\Ver^{\zeta}_{p^{(n)}}}(F(X),F(Y))$$ is an isomorphism. By rigidity, it is in fact enough to establish that the injective map $$Hom_{\Tilt^{\zeta}_{p^{(n)}}}(\mathbf{1},\mathrm{T}_{\zeta}(i))\rightarrow Hom_{\Ver^{\zeta}_{p^{(n)}}}(\mathbbm{1},F(\mathrm{T}_{\zeta}(i)))$$ is an isomorphism for all $0\leq i\leq p^{(n)}-2$. Further, we already know that this property holds for all $p^{(n-1)}-1\leq i$ by the definition of $\Ver^{\zeta}_{p^{(n)}}$ as the abelianization of $\Tilt^{\zeta}_{p^{(n)}}$ (see \cite[Theorem 2.41]{BEO}).

Let us set $\mathrm{S}:=\mathrm{T}_{\zeta}(p^{(n)}-1)$, and write $\epsilon:\mathrm{S}\otimes\mathrm{S}\rightarrow \mathbbm{1}$ for the coevaluation morphism in $\Tilt^{\zeta}$. We define a morphism $\tau:\mathrm{S}^{\otimes 4}\rightarrow\mathrm{S}^{\otimes 2}$ by $\tau:=\epsilon\otimes \mathrm{S}^{\otimes 2}-\mathrm{S}^{\otimes 2}\otimes\epsilon$. By construction (see \cite[Subsection 2.7]{BEO}), the sequence $$F(\mathrm{S}^{\otimes 4})\xrightarrow{\tau}F(\mathrm{S}^{\otimes 2})\xrightarrow{\epsilon}\mathbbm{1}\rightarrow 0$$ is exact in $\Ver^{\zeta}_{p^{(n)}}$. Consequently, it is enough to show that the sequence $$0\rightarrow Hom_{\Tilt^{\zeta}_{p^{(n)}}}(\mathbf{1},\mathrm{T}_{\zeta}(i))\xrightarrow{\epsilon} Hom_{\Tilt^{\zeta}_{p^{(n)}}}(\mathrm{S}^{\otimes 2},\mathrm{T}_{\zeta}(i))\xrightarrow{\tau} Hom_{\Tilt^{\zeta}_{p^{(n)}}}(\mathrm{S}^{\otimes 4},\mathrm{T}_{\zeta}(i))$$ is exact for all $0\leq i\leq p^{(n)}-2$. Moreover, appealing to proposition \ref{prop:Tiltingsubcategoryidentification} above, we find that this is equivalent to proving that the corresponding sequence in $\Tilt^{\zeta}_{<p^{(n)}-1}$ is exact.

But, $\Tilt^{\zeta}$ is a full monoidal subcategory of the abelian category $\mathbf{Mod}^{\zeta}$ of all finite dimensional modules of type $1$ for $U_{\zeta}$. In particular, we may view the tilting modules $\mathrm{S}$ and $\mathrm{T}_{\zeta}(i)$ as objects of $\mathbf{Mod}^{\zeta}$. Further, as $\mathbf{1}$ is simple and $\mathrm{S}^{\otimes 2}\rightarrow \mathbf{1}$ is non-zero, the map $$Hom_{U_{\zeta}}(\mathbf{1},\mathrm{T}_{\zeta}(i))\rightarrow Hom_{U_{\zeta}}(\mathrm{S}^{\otimes 2},\mathrm{T}_{\zeta}(i))$$ is injective for every $0\leq i\leq p^{(n)}-2$. It will therefore suffice to show that the sequence $$\begin{tikzcd}[sep=small]
0 \arrow[r] & {Hom_{U_{\zeta}}(\mathbf{1},\mathrm{T}_{\zeta}(i))} \arrow[r] & {Hom_{U_{\zeta}}(\mathrm{S}^{\otimes 2},\mathrm{T}_{\zeta}(i))} \arrow[r] & {Hom_{U_{\zeta}}(\mathrm{S}^{\otimes 4},\mathrm{T}_{\zeta}(i))}
\end{tikzcd}$$ is exact in the middle for all $0\leq i \leq p^{(n)}-2$, and we already know that this is the case for $p^{(n-1)}-1\leq i \leq p^{(n)}-2$.

Thanks to lemma \ref{lem:tiltingmodulesembedding}, we have that every tilting module $\mathrm{T}_{\zeta}(i)$ with $0\leq i\leq p^{(n-1)}-2$ embeds into $\mathrm{T}_{\zeta}(f(i))$ with $p^{(n-1)}-1\leq f(i)\leq p^{(n)}-2$. Furthermore, lemma \ref{lem:3.6} implies that $f(i)=2p^{(n-1)}-2$ if and only if $i=2p^{(k)}-2$ for some $0\leq k\leq n-2$. We therefore have a map of sequences $$\begin{tikzcd}[sep=small]
0 \arrow[r] & {Hom_{U_{\zeta}}(\mathbf{1},\mathrm{T}_{\zeta}(i))} \arrow[r] \arrow[d, hook] & {Hom_{U_{\zeta}}(\mathrm{S}^{\otimes 2},\mathrm{T}_{\zeta}(i))} \arrow[r] \arrow[d,hook] & {Hom_{U_{\zeta}}(\mathrm{S}^{\otimes 4},\mathrm{T}_{\zeta}(i))} \arrow[d,hook] \\
0 \arrow[r] & {Hom_{U_{\zeta}}(\mathbf{1},\mathrm{T}_{\zeta}(f(i)))} \arrow[r]        & {Hom_{U_{\zeta}}(\mathrm{S}^{\otimes 2},\mathrm{T}_{\zeta}(f(i)))} \arrow[r]        & {Hom_{U_{\zeta}}(\mathrm{S}^{\otimes 4},\mathrm{T}_{\zeta}(f(i))).}       
\end{tikzcd}$$ But, as was noted above, the bottomn sequence is exact, and the left vertical map is an isomorphism by lemma \ref{lem:3.6}, this implies that the top sequence is exact as desired.

The second part follows readily from \cite[Theorem 2.42]{BEO}, which is recalled at the end of section \ref{sub:preliminaries}.
\end{proof}

It follows from the universal property of the abelian envelope that the finite tensor category $\Ver^{\zeta}_{p^{(n)}}$ inherits a canonical pivotal structure from $\Tilt^{\zeta}_{p^{(n)}}$, which is in fact spherical, and is compatible with the functor $F:\Tilt^{\zeta}_{p^{(n)}}\rightarrow \Ver^{\zeta}_{p^{(n)}}$. We note that $\Ver^{\zeta}_{p^{(n)}}$ and $\Ver^{-\zeta}_{p^{(n)}}$ are equivalent as tensor categories, but that their spherical structures are distinct. On the other hand, $\Ver^{\zeta}_{p^{(n)}}$ and $\Ver^{\zeta^{-1}}_{p^{(n)}}$ are equivalent as spherical tensor categories. A choice of square root $\zeta^{1/2}$ for $\zeta$ yields a ribbon structure on the tensor category $\Tilt^{\zeta}_{p^{(n)}}$, which, thanks to the universal property of the abelian envelope, has the following consequence.

\begin{Corollary}
Let $\zeta^{1/2}$ be a square root for $\zeta$. There is a ribbon structure on $\Ver^{\zeta}_{p^{(n)}}$, and we denote by $\Ver^{\zeta,\zeta^{1/2}}_{p^{(n)}}$, or more succinctly $\Ver^{\zeta^{1/2}}_{p^{(n)}}$, the corresponding finite ribbon tensor category, such that the canonical monoidal functor $F:\Tilt^{\zeta^{1/2}}_{p^{(n)}}\rightarrow \Ver^{\zeta^{1/2}}_{p^{(n)}}$ is compatible with the ribbon structures.
\end{Corollary}

\begin{Remark}\label{rem:standardVerlinde}
In particular, if $\sigma = \pm 1$, the underlying tensor category of 
$\Ver^{\sigma}_{p^{(n)}}$ is the finite tensor category $\Ver_{p^{n}}$ introduced in \cite{BEO} (see also \cite{BE} for the case $p=2$). Furthermore, these references considered the tensor category $\Ver^{\sigma}_{p^{n}}$ equipped with its standard symmetric structure, which we do recover by considering $\Ver^{\sigma,\sigma^{1/2}}_{p^{(n)}}$ with $\sigma = +1$ and $\sigma^{1/2} = -1$. The symmetric structure given by taking $\sigma^{1/2} = +1$ will also appear, and so will the finite ribbon tensor category given by $\sigma = -1$ and $\sigma^{1/2}$ a primitive fourth root of unity provided that $p>2$. In this last case, it follows from remark \ref{rem:standardtilting} that the braiding is not symmetric, and that the symmetric center of $\Ver_{p^{n}}^{\sigma^{1/2}}$ is $\Ver_{p^{n}}^{\sigma^{1/2},+}$ whenever $\sigma=-1$ so that $\sigma^{1/2}$ is a primitive fourth root of unity.
\end{Remark}

\begin{Corollary}\label{cor:braidingsVerlinde}
The equivalence classes of ribbon structures on $\Ver_{p^{(n)}}^{\zeta}$ that extend the canonical spherical structure are given by the square roots of $\zeta^{\pm 1}$ in $\mathbbm{k}$.
\end{Corollary}
\begin{proof}
We have seen that the canonical pivotal monoidal functor $F:\Tilt^{\zeta}\rightarrow \Ver_{p^{(n)}}^{\zeta}$ restricts to a fully faithful functor $\Tilt^{\zeta}_{< p^{(n)}-1}\rightarrow \Ver_{p^{(n)}}^{\zeta}$. In particular, it follows from the proof of lemma \ref{lem:braidingsTilting} that a ribbon structure on $\Ver_{p^{(n)}}^{\zeta}$ induces a compatible ribbon structure on $\Tilt^{\zeta}$ provided that $F$ is fully faithful on $\mathrm{T}_{\zeta}(1)$, $\mathrm{T}_{\zeta}(1)^{\otimes 2}$, and $\mathrm{T}_{\zeta}(1)^{\otimes 3}$. This holds if $F$ is fully faithful on $\mathrm{T}_{\zeta}(1)$, $\mathrm{T}_{\zeta}(2)$, and $\mathrm{T}_{\zeta}(3)$, that is, if $3 < p^{(n)}-1$, or, equivalently, $p^{(n)}\geq 5$. In fact, it suffices that $F$ be fully faithful on $\mathrm{T}_{\zeta}(1)$ and $\mathrm{T}_{\zeta}(1)^{\otimes 2}$, and induces a monomorphism $$Hom_{U_{\zeta}}(\mathrm{T}_{\zeta}(1)^{\otimes 3},\mathrm{T}_{\zeta}(1))\hookrightarrow Hom_{\Ver_{p^{(n)}}^{\zeta}}(F(\mathrm{T}_{\zeta}(1)^{\otimes 3}),F(\mathrm{T}_{\zeta}(1))).$$ One checks directly that this weaker criterion holds when $p^{(n)}=4$. Finally, when $p^{(n)}=3$, we have $\Ver_{p^{(n)}}^{\zeta} = \mathrm{Vec}(\mathbb{Z}/2)$ as tensor categories, in which case the result is clear.
\end{proof}

\subsection{Basic Properties}

\begin{Theorem}\label{thm:basic}
\begin{enumerate}
    \item There is a fully faithful monoidal functor $F:\Tilt^{\zeta}_{p^{(n)}}\rightarrow \Ver^{\zeta}_{p^{(n)}}$. In particular, $\Ver^{\zeta}_{p^{(n)}}$ has $p^{(n)}-p^{(n-1)}$ simple objects whose projective covers are $F(\mathrm{T}_{\zeta}(v))$ with $p^{(n)}-1\leq v\leq p^{(n-1)}-2$, and all simple objects of $\Ver^{\zeta}_{p^{(n)}}$ are self-dual.
    \item For $n\geq 1$, let $\mathcal{O}_{p^{(n)}}:=\mathbb{Z}[2\cos{(\pi/p^{(n)})}]$ and $\mathcal{O}^+_{p^{(n)}}:=\mathbb{Z}[2\cos{(2\pi/p^{(n)})}]$ be subrings of $\mathbb{C}$. Then, the Frobenius-Perron dimension defines surjective ring homomorphisms $Gr(\Ver^{\zeta}_{p^{(n)}})\twoheadrightarrow\mathcal{O}_{p^{(n)}}$ and $Gr(\Ver^{\zeta, +}_{p^{(n)}})\twoheadrightarrow\mathcal{O}^+_{p^{(n)}}$.
    \item There are no tensor functors $\Ver^{\zeta, +}_{p^{(n)}}\rightarrow \Ver^{\zeta, +}_{p^{(m)}}$ or $\Ver^{\zeta}_{p^{(n)}}\rightarrow \Ver^{\zeta}_{p^{(m)}}$ for $m<n$.
\end{enumerate}
\end{Theorem}
\begin{proof}
The first part follows from theorem \ref{thm:completeness} together with \cite[Theorems 2.41 and 2.42]{BEO}. This also uses the fact that every object of $\Tilt^{\zeta}$ is self-dual.

Let us now prove the second part. The split Grothendieck ring of $\Tilt^{\zeta}_{p^{(n)}}$ is a $\mathbb{Z}_+$-ring of finite rank \cite[Chapter 3]{EGNO}, so that it is sensible to talk about the Frobenius-Perron dimension $\FPdim(T)$ of any $T\in \Tilt^{\zeta}_{p^{(n)}}$. By proposition \ref{prop:nTLobjects}, we have $$Q_{p^{(n)}}(\FPdim(\mathrm{T}_{\zeta}(1)))=0.$$ In fact, $Q_{p^{(n)}}$ is the minimal polynomial of the matrix of multiplication by $\mathrm{T}_{\zeta}(1)$ in the Grothendieck ring of $\Tilt^{\zeta}_{p^{(n)}}$ and $2\cos{(\pi/p^{(n)})}$ is its largest real root, so that $$\FPdim(\mathrm{T}_{\zeta}(1))=\FPdim(F(\mathrm{T}_{\zeta}(1)))=2\cos{(\pi/p^{(n)})}$$ by \cite[Prop. 3.3.4 and 3.3.13]{EGNO}. Thus, we find that $\FPdim(\mathrm{T}_{\zeta}(2))= 2\cos{(2\pi/p^{(n)})}+1$ for $\ell>2$ and $\FPdim(\mathrm{T}_{\zeta}(2))= 2\cos{(2\pi/p^{(n)})}+2$ for $\ell=2$. This shows that the image of the ring homomorphism $\FPdim:Gr(\Ver^{\zeta}_{p^{(n)}})\rightarrow \mathbb{C}$ contains $\mathcal{O}_{p^{(n)}}$. But, $\mathrm{T}_{\zeta}(1)$ generates $\Ver^{\zeta}_{p^{(n)}}$, and therefore $\mathrm{T}_{\zeta}(2)$ generates $\Ver^{\zeta, +}_{p^{(n)}}$, which implies that $\FPdim:Gr(\Ver^{\zeta}_{p^{(n)}})\rightarrow \mathcal{O}_{p^{(n)}}$ and $\FPdim:Gr(\Ver^{\zeta, +}_{p^{(n)}})\rightarrow \mathcal{O}^+_{p^{(n)}}$ are surjective.

The last part follows from the second and functoriality of the Frobenius-Perron dimension.
\end{proof}

\begin{Remark}\label{rem:dimensionkernelFPdim}
For $p^{(n)}\geq 2$, we have that $\mathcal{O}_{p^{(n)}}$ is the ring of integers in the field $\mathbb{Q}(2\cos{(\pi/p^{(n)})})$ by \cite[Prop. 2.16]{Wash}. Likewise, $\mathcal{O}^+_{p^{(n)}}$ is the ring of integers in the field $\mathbb{Q}(2\cos{(2\pi/p^{(n)})})$. Observe that, if $p^{(n)}$ is odd, then $\mathcal{O}^+_{p^{(n)}}=\mathcal{O}_{p^{(n)}}$. On the one hand, the rank of $\Ver^{\zeta,+}_{p^{(n)}}$ is $(p^{(n)}-p^{(n-1)})/2$. On the other hand, the degree of the field extension $\mathbb{Q}\subseteq \mathbb{Q}(2\cos{(2\pi/p^{(n)})})$ over $\mathbb{Q}$ is $\phi(p^{(n)})/2$, where $\phi$ denotes Euler's totient function. As a consequence, unlike in the symmetric case considered in \cite[Theorme 4.5]{BEO}, the ring homomorphism $\FPdim:Gr(\Ver^{\zeta, +}_{p^{(n)}})\rightarrow \mathcal{O}_{p^{(n)}}$ is not injective if $\ell\neq p$, i.e.\ $\zeta\neq \pm 1$.
\end{Remark}

We have the following universal property for the braided tensor category $\Ver^{\zeta}_{p^{(n)}}$.

\begin{Corollary}\label{cor:mixedVerlindeuniversalproperty}
For any braided tensor category $\mathcal{B}$ over $\mathbbm{k}$, the category of braided tensor functors $F:\Ver^{\zeta}_{p^{(n)}}\rightarrow\mathcal{B}$ is equivalent to the category of braided $\zeta$-Temperley-Lieb objects in $\mathcal{B}$ of degree exactly $n$.
\end{Corollary}

\subsection{Lift to Characteristic Zero}\label{sec:lifttocharzero}

We begin by considering the case $\zeta\neq\pm 1$, that is $\ell\neq p$. We use $W(\mathbbm{k})$ to denote the ring of Witt vectors for $\mathbbm{k}$. We write the fraction field of $W(\mathbbm{k})$ as $\mathrm{Frac}(W(\mathbbm{k}))$, and its algebraic closure as $\overline{\mathrm{Frac}(W(\mathbbm{k}))}$. The primitive $N$-th root of unity $\zeta\in\mathbbm{k}$ can be canonically lifted to a root of unity in $W(\mathbbm{k})$, which we will also write as $\zeta$. In fact, this also holds for the chosen square root $\zeta^{1/2}$. Then, we fix $\xi$ a $p^{n-1}$-th root of $\zeta$ in $\overline{\mathrm{Frac}(W(\mathbbm{k}))}$ together with a compatible square root $\xi^{1/2}$. In particular, $\xi$ is a primitive $Np^{n-1}$-th root of unity. We consider the discrete valuation ring $R:=W(\mathbbm{k})[\xi]$ with uniformizing parameter $\xi^N-1$, and its fraction field $\mathbb{K}:=\mathrm{Frac}(R)$. In the case $\zeta = \pm 1$, which has already been treated in \cite[Section 4.3]{BEO}, we instead take $\xi$ to be a primitive $p^n$-th root of unity in $\mathbb{K}$, and uniformizing parameter $\xi-1$.

\begin{Proposition}\label{prop:lifting}
The finite ribbon tensor category $\Ver^{\zeta^{1/2}}_{p^{(n)}}(\mathbbm{k})$ admits a flat deformation to a ribbon category over $R$ whose generic fiber is the semisimple Verlinde category $\Ver_{p^{(n)}}^{\xi^{1/2}}(\mathbb{K})$.
\end{Proposition}
\begin{proof}
We will repeatedly use that the ribbon tensor category $\Tilt^{\zeta^{1/2}}(\mathbbm{k})$ is equivalent to the Cauchy completion of the Temperley-Lieb category $\mathbf{TL}^{\zeta^{1/2}}(\mathbbm{k})$. In particular, given that the ribbon tensor category $\Ver^{\zeta^{1/2}}_{p^{(n)}}(\mathbbm{k})$ is the abelianization of $\Tilt^{\zeta^{1/2}}(\mathbbm{k})/\mathbf{J}_{p^{(n)}}(\mathbbm{k})$, it is enough to show that $\Tilt^{\zeta^{1/2}}(\mathbbm{k})/\mathbf{J}_{p^{(n)}}(\mathbbm{k})$ can be lifted to a flat Karoubian ribbon category over $R$ that has finitely many indecomposable objects. We consider the ribbon category $\mathbf{TL}^{\xi^{1/2}}(R)$, the Temperley-Lieb category over $R$ with parameter $\xi$, as well as its Cauchy completion $\Tilt^{\xi^{1/2}}(R)$. The ribbon category $\Tilt^{\xi^{1/2}}(R)$ is flat by construction. In addition, its indecomposable objects are given by $\widetilde{\mathrm{T}}_{\xi}(i)$, $i\geq 0$, and these map bijectively to the corresponding indecomposable objects $\mathrm{T}_{\zeta}(i)$ in $\Tilt^{\zeta^{1/2}}(\mathbbm{k})$. This follows from the fact that the indecomposable objects of $\Tilt^{\zeta^{1/2}}(\mathbbm{k})$ are indexed by the integers and that the ring $R$ is complete so that idempotents in $\Tilt^{\zeta^{1/2}}(\mathbbm{k})$ can be lifted to $\Tilt^{\xi^{1/2}}(R)$ by \cite[Theorem 21.31]{Lam}.
Then, the full subcategory $$\mathbf{J}_{p^{(n)}}(R):=\{\widetilde{\mathrm{T}}_{\xi}(i)|\,i\geq p^{(n)}-1\}^{\oplus}$$ is a tensor ideal of $\Tilt^{\xi^{1/2}}(R)$, and is generated by $\widetilde{\mathrm{T}}_{\xi}(p^{(n)}-1)$.

We claim that a non-zero morphism in $\Tilt^{\xi^{1/2}}(R)$ factors through $\mathbf{J}_{p^{(n)}}(R)$ if and only if it is in $\mathbf{J}_{p^{(n)}}(R)$. In order to see this, recall, for instance from \cite[Subsection 3.3]{BK}, see also \cite[Theorem 5.1]{STWZ}, that $\Tilt^{\xi^{1/2}}(\mathbb{K})$ has a unique non-trivial tensor ideal $\mathbf{J}_{p^{(n)}}(\mathbb{K})$. Then, observe that the Karoubi closure of the image of $\mathbf{J}_{p^{(n)}}(R)$ under the localization functor $\Tilt^{\xi^{1/2}}(R)\rightarrow\Tilt^{\xi^{1/2}}(\mathbb{K})$ is the tensor ideal $\mathbf{J}_{p^{(n)}}(\mathbb{K})$ of $\Tilt^{\xi^{1/2}}(\mathbb{K})$. More precisely, the image of $\widetilde{\mathrm{T}}_{\xi}(p^{(n)}-1)$ generates $\mathbf{J}_{p^{(n)}}(\mathbb{K})$.
The claim follows from the fact that the localization functor is faithful.

We now consider the ribbon Karoubian monoidal category $\Tilt^{\xi}(R)/\mathbf{J}_{p^{(n)}}(R)$ over $R$. It is flat by the discussion in the previous paragraph, and its indecomposable objects are (the images of) $\widetilde{\mathrm{T}}_{\xi}(i)$ with $0\leq i\leq p^{(n)}-2$. On one hand, by construction, the reduction of $\Tilt^{\xi}(R)/\mathbf{J}_{p^{(n)}}(R)$ to $\mathbbm{k}$ is exactly $\Tilt^{\zeta}(\mathbbm{k})/\mathbf{J}_{p^{(n)}}(\mathbbm{k})$. On the other hand, its localization $\Tilt^{\xi}(\mathbb{K})/\mathbf{J}_{p^{(n)}}(\mathbb{K})$ is the split semisimple Verlinde category $\Ver_{p^{(n)}}^{\xi}(\mathbb{K})$ of \cite[subsection 3.3]{BK} whose simple objects are the indecomposable Weyl modules $W_{\xi}(i)$, $0\leq i\leq p^{(n)}-2$. This concludes the proof.
\end{proof}

\begin{Remark}\label{rem:fermion}
For later use, we now recall some properties of the ribbon fusion category $\Ver^{\xi^{1/2}}_{p^{(n)}}(\mathbb{K})$. We do so at a broader level of generality than we have been considering above. Namely, we now let $\xi$ be a primitive root of unity of order $M$ in $\mathbb{K}$, and set $m = M/2$ if $m$ is even, and $m=M$ otherwise. We consider the split semisimple Verlinde category $\Ver_{p^{(n)}}^{\xi^{1/2}}(\mathbb{K}):=\Tilt^{\xi^{1/2}}(\mathbb{K})/\mathbf{J}_{m}(\mathbb{K})$, the quotient of $\Tilt^{\xi^{1/2}}(\mathbb{K})$ by its unique non-trivial tensor ideal. The equivalence classes of simple objects of $\Ver^{\xi^{1/2}}_{m}(\mathbb{K})$ are given by the objects $W_{\xi}(i)$ for $0\leq i\leq m-2$. The quantum dimension of $W_{\xi}(i)$ is $(-1)^i[i+1]_{\xi}$ in $\mathbb{K}$. In particular, we have \begin{equation}\label{eq:dimensionfermion}\mathrm{dim}\big(W_{\xi}(m-2)\big)=(-1)^{m + M}.\end{equation} In fact, the object $W_{\xi}(m-2)$ is invertible and has order $2$ as can be seen, for instance, from the proof of \cite[Theorem 6]{Saw}. Moreover, it follows from \cite[Example 3.3.22]{BK} that the twist $\theta$ on $W_{\xi}(m-2)$ is given by \begin{equation}\label{eq:twistfermion}\theta_{W_{\xi}(m-2)}=\xi^{(m-2)m/2},\end{equation} a fourth root of unity. We wish to emphasize that this formula depends on the square root $\xi^{1/2}$. Let us use $\mathcal{C}$ to denote the ribbon fusion subcategory of $\Ver^{\xi^{1/2}}_{m}(\mathbb{K})$ generated by $W_{\xi}(m-2)$. As fusion categories, we have $\mathcal{C}\simeq \mathrm{Vec}_{\mathbb{Z}/2}(\mathbb{K})$. Equations \eqref{eq:dimensionfermion} and \eqref{eq:twistfermion} then determine a unique ribbon structure on $\mathrm{Vec}_{\mathbb{Z}/2}(\mathbb{K})$. For example, if $M$ is odd, then $\mathcal{C}\simeq \mathrm{sVec}(\mathbb{K})$, the braided category of super vector spaces equipped with the spherical structure given by the total dimension.

When $\xi$ has even order, $\Ver^{\xi^{1/2}}_{m}(\mathbb{K})$ is a modular fusion category, and, in particular, its symmetric center is trivial. On the other hand, when $\xi$ has odd order, the braided fusion category $\Ver^{\xi^{1/2}}_{m}(\mathbb{K})$ is not modular:\ It follows from the proof of \cite[Theorem 6]{Saw} that $\mathcal{C}\simeq \mathrm{sVec}(\mathbb{K})$ is the symmetric center of $\Ver^{\xi^{1/2}}_{m}(\mathbb{K})$. As a consequence, writing $\Ver^{\xi^{1/2},+}_{m}(\mathbb{K})$ for the full tensor subcategory generated by the simple objects $W_{\xi}(i)$ with $i$ even, we find that there is an equivalence of ribbon fusion categories $\Ver^{\xi^{1/2}}_{m}(\mathbb{K})\simeq \Ver^{\xi^{1/2},+}_{m}(\mathbb{K})\boxtimes\,\mathrm{sVec}(\mathbb{K})$. It is well-known that the ribbon fusion category $\Ver^{\xi^{1/2},+}_{m}(\mathbb{K})$ is modular.
\end{Remark}

\subsection{Frobenius-Perron Dimension}

\begin{Proposition}\label{prop:FPdim}
The Frobenius-Perron dimension of the category $\Ver_{p^{(n)}}^{\zeta}$ is given by $$\FPdim(\Ver_{p^{(n)}}^{\zeta})=\frac{p^{(n)}}{2\sin^2{(\frac{\pi}{p^{(n)}})}}.$$ Further, for $0\leq i\leq \ell-1$ and with $q=e^{\pi i/p^{(n)}}\in\mathbb{C}^{\times}$, we have $$\FPdim(\mathrm{T}_{\zeta}(i)) = [i+1]_q.$$
\end{Proposition}
\begin{proof}
By \cite[Exercise 4.10.7]{EGNO}, we have that $$\FPdim(\mathrm{T}_{\xi}(i))= [i+1]_q= \frac{q^{i+1}-q^{-i-1}}{q-q^{-1}}.$$ The first part now follows from \cite[Proposition 2.60(ii)]{BEO}, as it provides us with the first equality below: $$\FPdim(\Ver_{p^{(n)}}^{\zeta}(\mathbbm{k}))=\FPdim(\Ver_{p^{(n)}}^{\xi}(\mathbb{K}))=\frac{p^{(n)}}{2\sin^2{(\frac{\pi}{p^{(n)}})}}.$$ The second part is a consequence of proposition \ref{prop:lifting} and its proof given that the Frobenius-Perron dimension is preserved in this case by \cite[Proposition 3.3.13]{EGNO}. Namely, for $0\leq i\leq \ell-1$, the tilting modules $\mathrm{T}_{\zeta}(i)$ satisfy $End_{U_{\zeta}}(\mathrm{T}_{\zeta}(i))\cong\mathbbm{k}$ given that they are simple as $U_{\zeta}$-modules. It therefore follows that the simple object $W_{\xi}(i)$ in $\Ver_{p^{(n)}}^{\xi}(\mathbb{K})$ lifts to $\widetilde{\mathrm{T}}_{\xi}(i)$ in $\Ver_{p^{(n)}}^{\xi}(R)$ for $0\leq i\leq \ell-1$, which proves the claim.
\end{proof}

\begin{Corollary}\label{cor:fermion}
Assume either that $p>2$ or that $p=2$, $n=1$ and $\ell > 2$. Then, the finite ribbon tensor category $\Ver_{p^{(n)}}^{\zeta^{1/2}}$ contains a non-trivial invertible object $\mathrm{G}$ of order two whose dimension and twist are given by $$\mathrm{dim}(\mathrm{G})=(-1)^{\ell + N} \ \textrm{and}\ \theta_{\mathrm{G}}=\zeta^{(\ell-2)\ell/2}.$$ In particular, if $p^{(n)}$ is odd, we have $\Ver_{p^{(n)}}^{\zeta^{1/2}}\simeq\Ver_{p^{(n)}}^{\zeta^{1/2},+}\boxtimes\, \mathrm{sVec}$ as finite ribbon tensor categories.
\end{Corollary}
\begin{proof}
The statement holds because it holds for $\Ver^{\xi^{1/2},+}_{p^{(n)}}(\mathbb{K})$ as is explained in remark \ref{rem:fermion} above. More precisely, as in \cite[Section 2.4]{BEO}, the non-trivial invertible object $W_{\xi}(p^{(n)}-2)$ of $\Ver^{\xi^{1/2}}_{p^{(n)}}(\mathbb{K})$ can be lifted to a non-trivial invertible (flat) object in $\Ver^{\xi^{1/2}}_{p^{(n)}}(R)$ and therefore yields an invertible object $\mathrm{G}$ in $\Ver_{p^{(n)}}^{\zeta}(\mathbbm{k})$. We claim that it is non-trivial, i.e.\ distinct from $\mathbbm{1}$. This follows from the facts that the projective cover of $\mathrm{G}$ must be $\mathrm{T}_{\zeta}(p^{(n)}-2)$ and that there is no non-zero map from $\mathrm{T}_{\zeta}(p^{(n)}-2)$ to $\mathbbm{1}$ by lemma \ref{lem:3.6} and our hypotheses.
\end{proof}

\begin{Remark}
In the cases $p>2$ or $p=2$, $n=1$ and $\ell > 2$, it will follow from subsequent results (see, for instance, proposition \ref{prop:simpletensorproduct} below) that $\mathrm{G}$ is the unique non-trivial invertible object of $\Ver_{p^{(n)}}^{\zeta}$. In the cases $p=2$ and $n\geq 2$, or $n=1$ and $\ell = 2$, the tensor category $\Ver_{p^{(n)}}^{\zeta}$ has no non-trivial invertible object.
\end{Remark}

\subsection{The Objects \texorpdfstring{$\mathbb{T}_{\zeta}(i)$}{T(i)}}

We write $\mathbb{T}_{\zeta}(i)$ for the image of $\mathrm{T}_{\zeta}(i)$ under the canonical tensor functor $\Tilt^{\zeta}/\mathbf{J}_{p^{(n)}}\rightarrow \Ver^{\zeta}_{p^{(n)}}$. It follows from \cite[Proposition 5.4]{STWZ} that the objects $\mathbb{P}_{\zeta}(i):=\mathbb{T}_{\zeta}(i)$ with $p^{(n-1)}-1\leq i\leq p^{(n)}-2$ are exactly the indecomposable projective objects of $\Ver^{\zeta}_{p^{(n)}}$.

\begin{Lemma}\label{lem:smallsimpleandcovers}
For $0\leq i\leq\ell-1$, the objects $\mathbb{T}_{\zeta}(i)$ are simple. Furthermore, the projective cover of $\mathbb{T}_{\zeta}(i)$ is $\mathbb{T}_{\zeta}(2p^{(n-1)}-2-i)$.
\end{Lemma}
\begin{proof}
Let $p^{(n-1)}-1\leq j\leq p^{(n)}-2$ be an integer. We claim that the $Hom$-space $Hom_{\Ver^{\zeta}_{p^{(n)}}}(\mathbb{P}_{\zeta}(j),\mathbb{T}_{\zeta}(i))$ in $\Ver^{\zeta}_{p^{(n)}}$ is non-zero only if $j=2p^{(n-1)}-2-i$, in which case it is one dimensional. Thanks to proposition \ref{prop:Tiltingsubcategoryidentification} and theorem \ref{thm:completeness}, it is enough to check these properties in $\Tilt^{\zeta}$. But, we have $$Hom_{\Tilt^{\zeta}}(\mathrm{T}_{\zeta}(i),\mathrm{T}_{\zeta}(j))\cong Hom_{\Tilt^{\zeta}}(\mathbf{1},\mathrm{T}_{\zeta}(i)\otimes\mathrm{T}_{\zeta}(j)).$$ Thus, by lemma \ref{lem:3.6}, it is enough to show that $\mathrm{T}_{\zeta}(i)\otimes\mathrm{T}_{\zeta}(j)$ contains $\mathrm{T}_{\zeta}(2p^{(n-1)}-2)$ as a direct summand exactly once if and only if $j=2p^{(n-1)}-2-i$. This is precisely the content of lemma \ref{lem:precisesummand}, which concludes the present proof.
\end{proof}

\begin{Lemma}\label{lem:simpleobjectU}
Assume $n\geq 2$, and $p > 2$ if $n=2$. The object $\mathbb{T}_{\zeta}(\ell)$ of $\Ver^{\zeta}_{p^{(n)}}$ has length $3$ and composition series $$[\mathbb{T}_{\zeta}(\ell-2),\mathbb{U},\mathbb{T}_{\zeta}(\ell-2)]$$ for a simple object $\mathbb{U}$ of $\Ver^{\zeta}_{p^{(n)}}$.
\end{Lemma}
\begin{proof}
We have $\mathbb{T}_{\zeta}(\ell)=\mathbb{T}_{\zeta}(1)\otimes \mathbb{T}_{\zeta}(\ell -1)$, so that the second part of proposition \ref{prop:FPdim} yields $$\FPdim(\mathbb{T}_{\zeta}(\ell))=\frac{(q^2-q^{-2})(q^{\ell}-q^{-\ell})}{(q-q^{-1})^2}$$ with $q=e^{\pi i/p^{(n)}}\in\mathbb{C}$. Now, it follows from lemma \ref{lem:smallsimpleandcovers} that there is a non-zero morphism $\mathbb{T}_{\zeta}(\ell)\rightarrow \mathbb{T}_{\zeta}(\ell-2)$. By taking duals, we also get a non-zero morphism $\mathbb{T}_{\zeta}(\ell-2)\rightarrow \mathbb{T}_{\zeta}(\ell)$. As $\mathbb{T}_{\zeta}(\ell-2)$ is simple, it follows that $\mathbb{T}_{\zeta}(\ell)$ has a composition series $[\mathbb{T}_{\zeta}(\ell-2),\mathbb{U},\mathbb{T}_{\zeta}(\ell-2)]$ for some object $\mathbb{U}$. It only remains to argue that $\mathbb{U}$ is simple. But, $$\FPdim(\mathbb{T}_{\zeta}(\ell-2))=\frac{q^{\ell-1}-q^{-\ell+1}}{q-q^{-1}},$$ so that we have $$\FPdim(\mathbb{U})=\FPdim(\mathbb{T}_{\zeta}(\ell))-2\FPdim(\mathbb{T}_{\zeta}(\ell-2))=q^{\ell}+q^{-\ell},$$
showing that $0<\FPdim(\mathbb{U})<2$, and therefore that the object $\mathbb{U}$ must be simple.
\end{proof}

\begin{Remark}
If $n=2$ and $p = 2$, then it follows from the proof above that $\mathbb{T}_{\zeta}(2)$ has length $2$ and composition series $[\mathbbm{1},\mathbbm{1}]$. In this case, we let $\mathbb{U} = 0$ for convenience.
\end{Remark}

\subsection{Temperley-Lieb Objects in Mixed Verlinde Categories}

\begin{Lemma}
For $n\geq 2$, there is a pivotal monoidal functor $$qFL:\mathbf{Tilt}^{\sigma}\rightarrow \Ver^{\zeta}_{p^{(n)}}$$ sending $\mathrm{T}(1)$ to $\mathbb{U}$, where $\sigma = (-1)^{\ell+N}$.
\end{Lemma}
\begin{proof}
Thanks to proposition \ref{prop:braidedTLobjects}, it is enough to exhibit a self-dual object $(X,\phi)$ of $\Ver^{\zeta}_{p^{(n)}}$ with $\mathrm{Tr}^L((\phi^{-1})^*\circ\phi)=-\sigma 2$. But, by \cite[Proposition 3.23]{STWZ}, we have $$\mathrm{dim}(\mathbb{U})=\mathrm{dim}(\mathbb{T}_{\zeta}(\ell)) -2\ \mathrm{dim}(\mathbb{T}_{\zeta}(\ell-2))=(-2)(-1)^{\ell}[\ell -1]_{\zeta}= -\sigma 2,$$ where the last equality uses that $[\ell -1]_{\zeta}=-1$ if $\zeta$ has odd order and $[\ell -1]_{\zeta}=+1$ otherwise, so that we can take $\mathbb{U}$ equipped with the identity isomorphism.
\end{proof}

Let us write $\mathbb{S}:=\mathbb{T}_{\zeta}(\ell-1)$. For later use, we wish to compute the tensor product of $\mathbb{S}$ with the object $\mathbb{U}$ of $\Ver^{\zeta}_{p^{(n)}}$.

\begin{Lemma}\label{lem:StensorU}
For $n\geq 2$, and with $p>2$ if $n=2$, we have $\mathbb{S}\otimes\mathbb{U}=\mathbb{T}_{\zeta}(2\ell -1)$.
\end{Lemma}
\begin{proof}
The morphisms $\mathbb{T}_{\zeta}(\ell)\rightarrow \mathbb{T}_{\zeta}(\ell-2)$ and $\mathbb{T}_{\zeta}(\ell-2)\rightarrow \mathbb{T}_{\zeta}(\ell)$ from the proof of lemma \ref{lem:simpleobjectU} split upon tensoring with $\mathbb{S}:=\mathbb{T}_{\zeta}(\ell-1)$ by the proof of \cite[Lemma 5.3]{STWZ}. Namely, by definition, $\mathbb{T}_{\zeta}(\ell-1)$ is the image in $\Ver^{\zeta}_{p^{(n)}}$ of the first Steinberg module $\mathrm{T}_{\zeta}(\ell-1)$ in $\Tilt^{\zeta}$. Thus, it is enough to compute both $\mathrm{T}_{\zeta}(\ell)\otimes \mathrm{T}_{\zeta}(\ell -1)$ and $\mathrm{T}_{\zeta}(\ell-2)^{\oplus 2}\otimes \mathrm{T}_{\zeta}(\ell -1)$ in $\Tilt^{\zeta}$, and then compare the different factors. In order to do so, recall that $\mathrm{T}_{\zeta}(\ell)=\mathrm{T}_{\zeta}(\ell-1)\otimes\mathrm{T}_{\zeta}(1)$, so that \begin{equation}\label{eq:tensortilting1}\mathrm{T}_{\zeta}(\ell)\otimes \mathrm{T}_{\zeta}(\ell -1)=\mathrm{T}_{\zeta}(\ell-1)\otimes \mathrm{T}_{\zeta}(\ell -1)\otimes\mathrm{T}_{\zeta}(1).\end{equation} Now, it follows from \cite[Lemma 4.1]{STWZ} that \begin{equation}\label{eq:tensortilting2}\mathrm{T}_{\zeta}(\ell-2)\otimes \mathrm{T}_{\zeta}(\ell -1) = \bigoplus_{i=0}^{\lfloor(\ell-2)/2\rfloor}\mathrm{T}_{\zeta}(2\ell-2i-3),\end{equation}\begin{equation}\label{eq:tensortilting3}\mathrm{T}_{\zeta}(\ell-1)\otimes \mathrm{T}_{\zeta}(\ell -1) = \bigoplus_{i=0}^{\lceil(\ell-2)/2\rceil}\mathrm{T}_{\zeta}(2\ell-2i-2).\end{equation} Let us write $$x_i=\begin{cases}0\ \textrm{if}\ \ell-2i-1=0,\\ 2\ \textrm{if}\ \ell-2i-1=1,\\ 1\ \textrm{otherwise}.\end{cases}$$ Combining equations \eqref{eq:tensortilting1} and \eqref{eq:tensortilting3} with \cite[Proposition 4.7]{STWZ}, we therefore find that \begin{align}\label{eq:tensortilting4}\mathrm{T}_{\zeta}(\ell)\otimes\mathrm{T}_{\zeta}(\ell-1) &=\bigoplus_{i=0}^{\lceil(\ell-2)/2\rceil}\Big(\mathrm{T}_{\zeta}(2\ell-2i-1)\oplus \mathrm{T}_{\zeta}(2\ell-2i-3)^{\oplus x_i}\Big)\notag\\
&=\mathrm{T}_{\zeta}(2\ell-1)\oplus \bigoplus_{i=0}^{\lfloor(\ell-2)/2\rfloor}\mathrm{T}_{\zeta}(2\ell-2i-3)^{\oplus 2}\\
&=\mathrm{T}_{\zeta}(2\ell-1)\oplus \Big(\mathrm{T}_{\zeta}(\ell-2)^{\oplus 2}\otimes \mathrm{T}_{\zeta}(\ell -1)\Big).\notag\end{align} This establishes that $\mathbb{S}\otimes\mathbb{U}=\mathbb{T}_{\zeta}(2\ell -1)$ as desired.
\end{proof}

The following result generalizes \cite[Proposition 4.33]{BEO}. In fact, our proof follows theirs closely. Both for the reader's convenience and because we need to appeal to different results, we give a complete proof.

\begin{Proposition}\label{prop:mixedVerlindeDonkin}
Let $n\geq 2$. For any $\ell -1 \leq a\leq 2\ell -2$ and non-negative integer $b$, we have $$\mathbb{T}_{\zeta}(a)\otimes qFL(\mathrm{T}(b))=\mathbb{T}_{\zeta}(a+\ell b).$$
\end{Proposition}

\begin{proof}
Let us consider the thick tensor ideal $\mathbf{J}_{p^{(1)}}/\mathbf{J}_{p^{(n)}}\subset \Tilt^{\zeta}_{p^{(n)}}$. We write $\mathbf{I}_{p^{(n)}}$ for the image of $\mathbf{J}_{p^{(1)}}/\mathbf{J}_{p^{(n)}}$ under $F:\Tilt^{\zeta}_{p^{(n)}}\rightarrow \Ver_{p^{(n)}}^{\zeta}$. We remark that, thanks to theorem \ref{thm:completeness}, $\mathbf{I}_{p^{(n)}}$ is the full subcategory of $\Ver_{p^{(n)}}^{\zeta}$ generated by $\mathbb{S}$ under taking tensor products and direct summands.

We will prove by induction that $\mathbb{S}\otimes qFL(\mathrm{T}(i))$ lies in $\mathbf{I}_{p^{(n)}}$ for every integer $i$. The case $i=0$ is clear as $\mathbb{T}_{\zeta}(0)=\mathbbm{1}$, the monoidal unit. Let us now assume that $\mathbb{S}\otimes qFL(\mathrm{T}(i))$ lies in $\mathbf{I}_{p^{(n)}}$ for some $i$. It follows from \cite[Proposition 4.7]{STWZ} that $\mathbb{T}_{\zeta}(i+1)$ is a direct summand of $\mathbb{T}_{\zeta}(i)\otimes \mathbb{T}_{\zeta}(1)$. This implies that $\mathbb{S}\otimes qFL(\mathrm{T}(i+1))$ is a direct summand of $$\mathbb{S}\otimes qFL(\mathrm{T}(i))\otimes qFL(\mathrm{T}(1))=\mathbb{S}\otimes\mathbb{U}\otimes qFL(\mathrm{T}(i))=\mathbb{T}_{\zeta}(2\ell-1)\otimes qFL(\mathbb{T}(i)),$$ where the last equality is given by lemma \ref{lem:simpleobjectU}. But, $\mathbb{T}_{\zeta}(2\ell-1)$ is a direct summand of $\mathbb{T}(\ell)\otimes\mathbb{S}$ by equation \eqref{eq:tensortilting4}, so that we therefore find that $\mathbb{S}\otimes qFL(\mathrm{T}(i+1))$ is a direct summand of $\mathbb{T}_{\zeta}(\ell)\otimes\mathbb{S}\otimes qFL(\mathrm{T}(i))$. This proves that $qFL(\mathrm{T}(i+1))$ lies in $\mathbf{I}_{p^{(n)}}$ since this subcategory of $\Ver_{p^{(n)}}^{\zeta}$ is closed under direct summands and tensor products.

Let us now observe that for $\ell -1 \leq a\leq 2\ell -2$, the object $\mathbb{T}_{\zeta}(a)$ is a direct summand of $\mathbb{T}_{\zeta}(a-(\ell-1))\otimes\mathbb{S}$ by \cite[Lemma 4.1]{STWZ}. This proves that $\mathbb{T}_{\zeta}(a)\otimes qFL(\mathrm{T}(b))$ is a direct summand of $\mathbb{T}_{\zeta}(a - (\ell -1))\otimes \mathbb{S}\otimes qFL(\mathrm{T}(b))$, and thereby $\mathbb{T}_{\zeta}(a)\otimes qFL(\mathrm{T}(b))$ is in $\mathbf{I}_{p^{(n)}}$ thanks to the above argument.

We turn our attention to proving the statement of the proposition. Our proof proceed by induction on $b$. The case $v=0$ is clear, so let us assume that $\mathbb{T}_{\zeta}(a)\otimes qFL(\mathrm{T}(b))=\mathbb{T}_{\zeta}(a+\ell b)$ holds for all $b\leq v$. Let us write $S(v)$ for the multisubset of $\mathbb{N}_{\geq 0}$ such that $$\mathrm{T}(1)\otimes \mathrm{T}(v)=\mathrm{T}(v+1)\oplus\bigoplus_{b\in S(v)}\mathrm{T}(b).$$ By \cite[Proposition 4.7]{STWZ}, for any $s\in S(v)$, we have $s<v$. Using proposition \ref{prop:mixedDonkin}, we have the following chain of equalities $$\mathbb{S}\otimes\mathbb{T}_{\zeta}(a)\otimes qFL\big(\mathrm{T}(v+1)\oplus\bigoplus_{b\in S(v)}\mathrm{T}(b)\big)= \mathbb{S}\otimes \mathbb{T}_{\zeta}(a)\otimes qFL(\mathrm{T}(1)\otimes \mathrm{T}(v))= $$ $$\mathbb{S}\otimes \mathbb{U}\otimes \mathbb{T}_{\zeta}(a)\otimes qFL(\mathrm{T}(v)) = \mathbb{T}_{\zeta}(2\ell-1)\otimes\mathbb{T}_{\zeta}(a+\ell v)=F\big(\mathrm{T}_{\zeta}(2\ell-1)\otimes\mathrm{T}_{\zeta}(a+\ell v)\big)=$$ $$F\big(\mathrm{T}_{\zeta}(\ell-1)\otimes\mathrm{T}(1)^{[q]}\otimes\mathrm{T}_{\zeta}(a)\otimes\mathrm{T}(v)^{[q]}\big)=\mathbb{S}\otimes F\big(\mathrm{T}_{\zeta}(a)\otimes(\mathrm{T}(1)\otimes\mathrm{T}(v))^{[q]}\big)=$$ $$\mathbb{S}\otimes F\Big(\mathrm{T}_{\zeta}(a)\otimes\big(\mathrm{T}(v+1)\oplus\bigoplus_{b\in S(v)}\mathrm{T}(b)\big)^{[q]}\Big)=\mathbb{S}\otimes F\big(\mathrm{T}_{\zeta}(a+\ell(v+1))\oplus\bigoplus_{b\in S(v)}\mathrm{T}_{\zeta}(a+\ell b)\big)=$$ $$= \mathbb{S}\otimes \big(\mathbb{T}_{\zeta}(a+\ell(v+1))\oplus\bigoplus_{b\in S(v)}\mathbb{T}_{\zeta}(a+\ell b)\big).$$ As a consequence of the inductive hypothesis, we find that \begin{equation} \label{eq:SteinebrgmixedVerlindeDonkin}\mathbb{S}\otimes\mathbb{T}_{\zeta}(a)\otimes qFL(\mathrm{T}(v+1))=\mathbb{S}\otimes\mathbb{T}_{\zeta}(a+\ell(v+1)).\end{equation}

Finally, note that the (complexified) split Grothendieck ring $Gr(\Tilt^{\zeta}_{p^{(n)}})\otimes_{\mathbb{Z}} \mathbb{C}$ is a finite semisimple algebra \cite[Corollary 3.7.7]{EGNO}. Furthermore, this algebra is commuative as $\Tilt^{\zeta}_{p^{(n)}}$ admits a braiding, so that it is in fact isomorphic to a finite direct sum of copies of the algebra $\mathbb{C}$. But, in such a ring, the equality $a^2b_1=a^2b_2$ implies $ab_1=ab_2$. Applying this to the equality given by equation \eqref{eq:SteinebrgmixedVerlindeDonkin}, recalling from the first part of the proof that both $\mathbb{T}_{\zeta}(a)\otimes qFL(\mathrm{T}(v))$ and $\mathbb{T}_{\zeta}(a+\ell(v+1))$ lie in $\mathbf{I}_{p^{(n)}}$, we find $\mathbb{T}_{\zeta}(a)\otimes qFL(\mathrm{T}(v+1))=\mathbb{T}_{\zeta}(a+\ell (v+1))$ as desired.
\end{proof}

\subsection{The Quantum Frobenius-Lusztig Functor}

We use the pivotal monoidal functor $qFL:\mathbf{Tilt}^{\sigma}\rightarrow \Ver^{\zeta}_{p^{(n)}}$ of the previous section to construct a ribbon tensor functor $$\mathbbm{q}\mathbb{FL}:\Ver_{p^{n-1}}^{\sigma^{1/2}}\rightarrow \Ver^{\zeta^{1/2}}_{p^{(n)}},$$ which is an analogue of the quantum Frobenius-Lusztig twist for the mixed Verlinde categories.

\begin{Theorem}\label{thm:quantumFrobeniusLusztig}
For $n\geq 2$, there is a ribbon embedding $$\mathbbm{q}\mathbb{FL}:\Ver_{p^{n-1}}^{\sigma^{1/2}}\rightarrow \Ver^{\zeta^{1/2}}_{p^{(n)}}$$ with $\sigma^{1/2} = \zeta^{(\ell-2)\ell/2}$, a fourth root of unity in $\mathbbm{k}^{\times}$.
\end{Theorem}
\begin{proof}
Firstly, we claim that $qFL:\mathbf{Tilt}^{\sigma}\rightarrow \Ver^{\zeta}_{p^{(n)}}$ factors through $\mathbf{Tilt}^{\sigma}_{p^{n-1}}$ via a faithful functor $\widetilde{qFL}:\mathbf{Tilt}^{\sigma}_{p^{n-1}}\rightarrow \Ver^{\zeta}_{p^{(n)}}$. In order to prove this, it is enough to show that $qFL(\mathrm{T}(p^{n-1}-1))=0$ and $qFL(\mathrm{T}(p^{n-2}-1))\neq 0$. But, by appealing to proposition \ref{prop:mixedVerlindeDonkin}, we have in $\Ver_{p^{(n)}}^{\zeta}$ $$\mathbb{T}_{\zeta}(\ell-1)\otimes qFL(\mathrm{T}(p^{n-1}-1))=\mathbb{T}_{\zeta}(\ell-1 + \ell(p^{n-1}-1))=\mathbb{T}_{\zeta}(p^{(n)}-1)=0,$$ $$\mathbb{T}_{\zeta}(\ell-1)\otimes qFL(\mathrm{T}(p^{n-2}-1))=\mathbb{T}_{\zeta}(\ell-1 + \ell(p^{n-2}-1))=\mathbb{T}_{\zeta}(p^{(n-1)}-1)\neq 0,$$ which proves the claim.

Secondly, it now follows from the universal property of the abelian envelope that there is a pivotal tensor functor $\mathbbm{q}\mathbb{FL}:\Ver_{p^{n-1}}^{\sigma}\rightarrow \Ver^{\zeta}_{p^{(n)}}$ sending $\mathbb{T}(1)$ to $\mathbb{U}$. Note that the functor $\mathbbm{q}\mathbb{FL}$ is automatically faithful by \cite[Remark 4.3.10]{EGNO}.

Thirdly, we will show that $\mathbbm{q}\mathbb{FL}$ is full, i.e.\ it is an embedding. Thanks to the exactness of $\mathbbm{q}\mathbb{FL}$, it suffices to check fullness on the projective objects. Further, the claim follows provided that we can show that the injective map $$Hom_{\mathbf{Tilt}^{\sigma}}(\mathbf{1},\mathrm{T}(b))\hookrightarrow Hom_{\Ver^{\zeta}_{p^{(n)}}}(qFL(\mathbf{1}),qFL(\mathrm{T}(b)))$$ is surjective for any $0\leq b\leq p^{n-1}-2$. By lemma \ref{lem:3.6}, we therefore have to show that $Hom_{\Ver^{\zeta}_{p^{(n)}}}(qFL(\mathbf{1}),qFL(\mathrm{T}(b)))$ is one dimensional for $b=2p^{k}-2$, $0\leq k\leq n-1$ and zero otherwise. As $\mathbbm{1}=\mathbb{T}(0)$ is a quotient of $\mathbb{T}(2\ell-2)$, we have $$\dim Hom_{\Ver^{\zeta}_{p^{(n)}}}(\mathbbm{1},qFL(\mathrm{T}(b)))\leq \dim Hom_{\Ver^{\zeta}_{p^{(n)}}}(\mathbb{T}_{\zeta}(2\ell-2),qFL(\mathrm{T}(b)))=$$ $$\dim Hom_{\Ver^{\zeta}_{p^{(n)}}}(\mathbbm{1},\mathbb{T}_{\zeta}(2\ell-2)\otimes qFL(\mathrm{T}(b)))=\dim Hom_{\Ver^{\zeta}_{p^{(n)}}}(\mathbbm{1},\mathbb{T}_{\zeta}(2\ell-2+\ell b)),$$ and the claim follows from lemma \ref{lem:3.6} via theorem \ref{thm:completeness}.

Finally, we claim that the pivotal tensor embedding $\mathbbm{q}\mathbb{FL}$ is braided with $\sigma^{1/2} = \zeta^{(\ell-2)\ell/2}$. Thanks to corollary \ref{cor:braidingsVerlinde}, the compatible braidings on the spherical finite tensor category $\Ver_{p^{n-1}}^{\sigma}$ are parameterised by the square roots of $\sigma^{\pm 1}$ in $\mathbbm{k}$. But, as $\sigma = \pm 1$ there are exactly two such braidings if $p>2$ and one otherwise. We will therefore assume that $p>2$. In order to finish the proof, it therefore suffices to determine which braiding commutes with $\mathbbm{q}\mathbb{FL}$. As $n\geq 2$ and $p > 2$, the finite tensor category $\Ver_{p^{n-1}}^{\sigma}$ has a unique non-trivial invertible object $\mathrm{F}$ by \cite[Corollary 4.11]{BEO}. We claim that its image under $\mathbbm{q}\mathbb{FL}$ must be the non-trivial invertible object $\mathrm{G}$ of $\Ver_{p^{(n)}}^{\zeta}$ from corollary \ref{cor:fermion}. Namely, the projective cover of $\mathrm{F}$ in $\Ver_{p^{n-1}}^{\sigma}$ is $\mathbb{T}(p^{n-1}-2)$ and $\mathrm{T}_{\zeta}(2\ell-2)$ surjects onto $\mathbbm{1}$, so that $\mathbb{T}_{\zeta}(2\ell -2)\otimes \mathbbm{q}\mathbb{FL}(\mathbb{T}(p^{n-1}-2))$ surjects onto $\mathbbm{q}\mathbb{FL}(\mathrm{F})$. Then, it follows from proposition \ref{prop:mixedVerlindeDonkin} that $\mathbb{T}_{\zeta}(2\ell -2)\otimes \mathbbm{q}\mathbb{FL}(\mathbb{T}(p^{n-1}-2))=\mathbb{T}_{\zeta}(p^{(n)}-2)$, which is projective. But, as was noted in the proof of corollary \ref{cor:fermion}, $\mathbb{T}_{\zeta}(p^{(n)}-2)$ is also the projective cover of $\mathrm{G}$, and the claim follows. Now, on the one hand, we have from the same corollary that $$\mathrm{dim}(\mathrm{G})=(-1)^{\sigma} \ \textrm{and}\ \theta_{\mathrm{G}}=\zeta^{(\ell-2)\ell/2}.$$ One the other hand, we have $$\mathrm{dim}(\mathrm{F})=(-1)^{\sigma} \ \textrm{and}\ \theta_{\mathrm{F}}=\sigma^{1/2}.$$ Thence, $\mathbbm{q}\mathbb{FL}$ is compatible with the braiding on $\Ver_{p^{(n)}}^{\zeta^{1/2}}$ if and only if $\sigma^{1/2} = \zeta^{(\ell-2)\ell/2}$. This completes the proof of the theorem.
\end{proof}

\begin{Remark}
The construction of the quantum Frobenius-Lusztig twist for the mixed Verlinde categories given above is inspired by the properties of the Frobenius functor between the symmetric Verlinde categories that were investigated in \cite[Sections 4.9 \& 4.10]{BEO}. More precisely, for any symmetric finite tensor category $\mathcal{C}$ over $\mathbbm{k}$, there is a symmetric monoidal functor $\mathbb{F}:\mathcal{C}\rightarrow \mathcal{C}^{(1)}\boxtimes\Ver_p$ called the Frobenius functor \cite{EO}. When $\mathcal{C}=\Ver_{p^n}$, the Frobenius functor factorizes through $\Ver_{p^n}^{(1)}\boxtimes\mathrm{Vec}\simeq \Ver_{p^n}$, that is, there is a canonical symmetric monoidal functor $\mathbb{F}:\Ver_{p^n}\rightarrow \Ver_{p^n}$. Moreover, the object $\mathbb{F}(\mathbb{T}(1))=\mathbb{U}$ is simple and appears in the composition series of $\mathbb{T}(p)$ which is given by $[\mathbb{T}(p-2),\mathbb{U},\mathbb{T}(p-2)]$. This fact can be used to prove that the Frobenius functor induces a symmetric embedding $\Ver_{p^{n-1}}\hookrightarrow \Ver_{p^n}$.

By contrast, our construction of the quantum Frobenius-Lusztig functor proceeds differently. We use directly the properties of the simple object $\mathbb{U}$ in $\Ver_{p^{(n)}}^{\zeta}$ together with the universal property of the spherical monoidal category $\Tilt^{\sigma}$ and its quotients. This method may be used to give an ad hoc construction of the Frobenius functor in the positive characteristic case. A slight difficulty with this approach is to show that the resulting functor is braided.
\end{Remark}

\begin{Corollary}\label{cor:coverqFLsimple}
Let $n\geq 2$, and let $\mathrm{L}$ be a simple object of $\Ver^{\sigma}_{p^{n-1}}$ with projective cover $\mathbb{T}_{\zeta}(r)$ with $p^{n-2}-1\leq r\leq p^{n-1}-2$. Then, the projective cover of $\mathbbm{q}\mathbb{FL}(\mathrm{L})$ in $\Ver_{p^{(n)}}^{\zeta}$ is $\mathbb{T}_{\zeta}(2\ell-2+\ell r)$.
\end{Corollary}
\begin{proof}
As $\mathbb{T}_{\zeta}(2\ell -2)$ surjects onto $\mathbbm{1}=\mathbb{T}_{\zeta}(0)$, $\mathbb{T}_{\zeta}(2\ell -2)\otimes \mathbbm{q}\mathbb{FL}(\mathbb{T}(r))$ surjects onto $\mathbbm{q}\mathbb{FL}(\mathrm{L})$, and $\mathbb{T}_{\zeta}(2\ell-2+\ell r)=\mathbb{T}_{\zeta}(2\ell -2)\otimes \mathbbm{q}\mathbb{FL}(\mathbb{T}(r))$ by proposition \ref{prop:mixedVerlindeDonkin}.
\end{proof}


\begin{Proposition}\label{prop:preblocks}
The functor $\Ver^{\sigma}_{p^{n-1}}\rightarrow\Ver_{p^{(n)}}^{\zeta}$ given by $X\mapsto \mathbbm{q}\mathbb{FL}(X)\otimes \mathbb{T}_{\zeta}(\ell-1)$ identifies $\Ver^{\sigma}_{p^{n-1}}$ with a direct summand of $\Ver_{p^{(n)}}^{\zeta}$ as an additive category.
\end{Proposition}
\begin{proof}
This functor is manifestly exact and faithful. It is therefore enough to show that it preserves projective objects and that it is full on projective objects. The first part of the claim follows from proposition \ref{prop:mixedVerlindeDonkin}. As for the second part, it follows from theorem \ref{thm:completeness} that it is sufficient to show that for any non-negative integers $b,d$, we have $$Hom_{U_{\zeta}}(\mathrm{T}_{\zeta}(\ell-1)\otimes \mathrm{T}(b)^{[q]},\mathrm{T}_{\zeta}(\ell-1)\otimes \mathrm{T}(d)^{[q]}) \cong Hom_{U}(\mathrm{T}(b),\mathrm{T}(d)).$$ This is a special case of equation \eqref{eqn:quantumclassicalfactorization}.
\end{proof}

\subsection{The Tensor Product Theorem}

\begin{Proposition}\label{prop:preVerlindeSteinbergtensor}
Assume $n\geq 2$. Let $\mathrm{L}$ be a simple object of $\Ver_{p^{n-1}}^{\sigma}$ with projective cover $\mathbb{T}(r)$ for some integer $p^{n-2}-1\leq r\leq p^{n-1}-2$. Then, for any $0\leq i\leq \ell-1$, the object $\mathbb{T}_{\zeta}(i)\otimes \mathbbm{q}\mathbb{FL}(\mathrm{L})$ of $\Ver_{p^{(n)}}^{\zeta}$ is simple, and its projective cover is $\mathbb{T}_{\zeta}(2\ell -2-i+\ell r)$. Furthermore, every simple object of $\Ver_{p^{(n)}}^{\zeta}$ is of this form.
\end{Proposition}
\begin{proof}
Recall that the indecomposable projective objects of $\Ver_{p^{n}}^{\zeta}$ are given by $\mathbb{T}_{\zeta}(j)$ for $p^{(n-1)}-1\leq j\leq p^{(n)}-2$. With such a $j$, we have $$Hom_{\Ver_{p^{n}}^{\zeta}}(\mathbb{T}_{\zeta}(j),\mathbb{T}_{\zeta}(i)\otimes \mathbbm{q}\mathbb{FL}(\mathrm{L}))\cong Hom_{\Ver_{p^{n}}^{\zeta}}(\mathbb{T}_{\zeta}(i)\otimes \mathbb{T}_{\zeta}(j),\mathbbm{q}\mathbb{FL}(\mathrm{L})).$$ In particular, it follows from lemma \ref{lem:smallsimpleandcovers} and corollary \ref{cor:coverqFLsimple} above that this $Hom$-space is non-zero if and only if $\mathbb{T}_{\zeta}(2\ell-2+\ell r)$ is a direct summand of $\mathbb{T}_{\zeta}(i)\otimes \mathbb{T}_{\zeta}(j)$. But, by lemma \ref{lem:precisesummand}, this holds if and only if $j=2\ell-2+\ell r -i$, in which case the summand $\mathbb{T}_{\zeta}(2\ell-2+\ell r)$ appears exactly once. This establishes that $\mathbb{T}_{\zeta}(2\ell -2-i+\ell r)$ is the projective cover of $\mathbb{T}_{\zeta}(i)\otimes \mathbbm{q}\mathbb{FL}(\mathrm{L})$ in $\Ver_{p^{n}}^{\zeta}$ as desired. This construction produces $\ell (p^{n-1}-p^{n-2})=p^{(n)}-p^{(n-1)}$ distinct simple objects of $\Ver_{p^{n}}^{\zeta}$, and is therefore exhaustive.
\end{proof}

Our next theorem is an analogue of Steinberg's tensor product formula for mixed Verlinde categories. This result has many variants. Let us, for instance, mention the version given in \cite[Theorem 1.10]{AK} for integrable representations of quantum groups in the mixed case. Most directly relevant to our purposes is the positive characteristic version of this result obtained for symmetric Verlinde categories in \cite[Theorem 4.42]{BEO}. In particular, following the convention given in \cite[Section 4.11]{BEO}, for any integer $m\geq 1$, we write $\mathrm{L}(b)$ with $0\leq b\leq p^{m}-p^{m-1}-1$ to denote the simple objects of $\Ver^{\sigma}_{p^{m}}$. Combining \cite[Theorem 4.42]{BEO} with proposition \ref{prop:preVerlindeSteinbergtensor}, we obtain the result below.

\begin{Theorem}\label{thm:VerlindeSteinbergtensor}
Let $0\leq a\leq p^{(n)}-p^{(n-1)}-1$, and write $a=a_0+\ell b$ with $0\leq a_0\leq \ell-1$ and $0\leq b\leq p^{n-1}-p^{n-2}-1$. Then, $\mathrm{L}_{\zeta}(a):=\mathbb{T}_{\zeta}(a_0)\otimes \mathbbm{q}\mathbb{FL}(\mathrm{L}(b))$ is a simple object of $\Ver_{p^{(n)}}^{\zeta}$, and any simple object can be uniquely expressed in this form. Furthermore, if $a=[a_{n-1},...,a_0]_{p,\ell}$, then the projective cover of $\mathrm{L}_{\zeta}(a)$ is $\mathbb{P}_{\zeta}(s(a))=\mathbb{T}_{\zeta}(s(a))$ with $$s(a)=p^{(n-1)}-1+[a_{n-1}, p-1-a_{n-2},...,p-1-a_{2},\ell -1 -a_0]_{p,\ell}.$$
\end{Theorem}

\begin{Remark}\label{rem:maximalSteinberg}
The simple objects of the symmetric Verlinde category $\Ver^{\sigma}_{p^{n-1}}$ were explicitly described in \cite[Theorem 4.42]{BEO}. More precisely, recall from \cite[Theorem 4.34]{BEO}, that there is a sequence of symmetric embeddings $\Ver^{\sigma}_{p}\subset...\subset\Ver^{\sigma}_{p^{n-1}}$ given by iterating the Frobenius functor $\mathbb{F}$. Given any integers $1\leq r\leq n-1$ and $0\leq i\leq p-1$ with $i\neq p-1$ if $r=1$, we write $\mathbb{T}^{[r]}(i)$ for the image of the simple object $\mathbb{T}(i)$ of $\Ver_{p^{r}}$ in $\Ver_{p^{n-1}}$. They show that the object $\mathbb{T}^{[r]}(i)$ is simple. Moreover, given any integer $0\leq b\leq p^{n-1}-p^{n-2}-1$ with $p$-adic expansion $b=[b_{n-2},...,b_0]$, they proved that the object $$\mathrm{L}(b):=\mathbb{T}^{[1]}(b_{n-2})\otimes ...\otimes \mathbb{T}^{[n-1]}(b_0)$$ of $\Ver^{\sigma}_{p^{n-1}}$ is simple, and that every simple object can be uniquely written in this fashion. In particular, for any integer $0\leq a\leq p^{(n)}-p^{(n-1)}-1$ with $p\ell$-adic expansion $a=[a_{n-1},...,a_0]_{p,\ell}$, we immediately find that $$\mathrm{L}_{\zeta}(a)=\mathbb{T}_{\zeta}(a_0)\otimes \mathbbm{q}\mathbb{FL}\big(\mathbb{T}^{[1]}(a_{n-1})\otimes ...\otimes \mathbb{T}^{[n-1]}(a_1)\big),$$ and every simple object of $\Ver_{p^{(n)}}^{\zeta}$ can be uniquely written in this way. As an example, the simple object $\mathbb{U}$ of $\Ver_{p^{(n)}}^{\zeta}$ constructed in lemma \ref{lem:simpleobjectU} is given by $\mathbb{U}=\mathbbm{q}\mathbb{FL}(\mathbb{T}^{[n-1]}(1))$.
\end{Remark}

\begin{Corollary}\label{cor:FPdimsimples}
The quantum dimension of the simple object $\mathrm{L}_{\zeta}(a)$ in the spherical tensor category $\Ver_{p^{(n)}}^{\zeta}$ is given by $$\dim(\mathrm{L}_{\zeta}(a))=(-1)^{a_0}[a_0+1]_{\zeta} \prod_{i=1}^{n-1}\sigma^{a_i}(a_i+1)\in\mathbbm{k}.$$ The Frobenius-Perron dimension of the simple object $\mathrm{L}_{\zeta}(a)$ is given by $$\FPdim(\mathrm{L}_{\zeta}(a))=\prod_{i=0}^{n-1}[a_i+1]_{q^{p^{(i)}}}\in\mathbb{R}$$ with $q=e^{\pi i/p^{(n)}}\in\mathbb{C}^{\times}$.
\end{Corollary}
\begin{proof}
The first part follows from \cite[Proposition 3.23]{STWZ} and theorem \ref{thm:quantumFrobeniusLusztig}. The second part follows from proposition \ref{prop:FPdim} as the Frobenius-Perron dimension is preserved by tensor functors.
\end{proof}

\begin{Remark}
The astute reader may have noticed that there is a difference of sign between the first part of our last corollary above and the positive characteristic version given in \cite[Corollary 4.44]{BEO}. This discrepancy is explained by our choice of spherical structure on $\Tilt^{+1}$ and therefore $\Ver^{+1}_{p^{n}}$. More precisely, we follow the conventions of \cite{STWZ}, and therefore have $\dim(\mathrm{T}(1))=-2$, whereas \cite[Corollary 4.44]{BEO} uses the spherical structure of $\Ver_{p^{n}}^{-1}$.
\end{Remark}

\begin{Example}\label{ex:compositionfactors}
Let $\ell\leq a\leq 2\ell -2$ and $n>2$ if $p=2$. By lemma \ref{lem:precisesummand}, we find that $\mathbb{T}_{\zeta}(2\ell-2-a)\otimes \mathbb{T}_{\zeta}(a)$ contains $\mathbb{T}_{\zeta}(2\ell-2)$ as a summand exactly once. In particular, by lemma \ref{lem:3.6}, we have non-zero maps $\mathbb{T}_{\zeta}(2\ell-2-a)\otimes \mathbb{T}_{\zeta}(a)\rightarrow \mathbbm{1}$ and $\mathbbm{1}\rightarrow\mathbb{T}_{\zeta}(2\ell-2-a)\otimes \mathbb{T}_{\zeta}(a)$, and therefore also non-zero maps $\mathbb{T}_{\zeta}(2\ell-2-a)\rightarrow \mathbb{T}_{\zeta}(a)$ and $\mathbb{T}_{\zeta}(2\ell-2-a)\rightarrow \mathbb{T}_{\zeta}(a)$. Given that $\mathbb{T}_{\zeta}(2\ell-2-a)$ is simple, we find that $\mathbb{T}_{\zeta}(2\ell-2-a)$ is both a quotient and subobject of $\mathbb{T}_{\zeta}(a)$. Moreover, by comparing Frobenius-Perron dimensions, we find that $\mathbb{T}_{\zeta}(a)$ has composition series $[\mathbb{T}_{\zeta}(2\ell-2-a),\mathbb{T}_{\zeta}(a-\ell)\otimes \mathbbm{q}\mathbb{FL}(\mathbb{T}(1)),\mathbb{T}_{\zeta}(2\ell-2-a)]$. This generalizes lemma \ref{lem:simpleobjectU}.
\end{Example}

\begin{Corollary}
Up to monoidal natural isomorphism, the identity functor is the only braided, and, a fortiori also ribbon, tensor functor $\Ver_{p^{(n)}}^{\zeta^{1/2}}\rightarrow \Ver_{p^{(n)}}^{\zeta^{1/2}}$.
\end{Corollary}
\begin{proof}
This follows from corollary \ref{cor:mixedVerlindeuniversalproperty} by adapting that argument used to prove \cite[Corollary 4.48]{BEO} in the obvious way.
\end{proof}

\subsection{Tensor Product of Simple Objects and Grothendieck Ring}

We compute the tensor product of the simple objects of $\Ver_{p^{(n)}}^{\zeta}$ in its Grothendieck ring, which we subsequently identify. More precisely, an iterative description for the tensor product of the simple objects $\mathrm{L}(c)\mathrm{L}(d)$ of $\Ver^{\sigma}_{p^{n-1}}$ in its Grothendieck ring was obtained in \cite[Corollary 4.50]{BEO}, which we extend to $\Ver_{p^{(n)}}^{\zeta}$.

\begin{Proposition}\label{prop:simpletensorproduct}
Let $n\geq 2$, and $0\leq a,b\leq p^{(n)}-p^{(n-1)}-1$, which we also write as $a= c\ell + r$ and $b=d\ell + s$ for some integers $0\leq r,s\leq \ell -1$, and $0\leq c,d$. Then, in the Grothendieck ring $Gr(\Ver_{p^{(n)}}^{\zeta})$, we have $$\mathrm{L}_{\zeta}(a)\mathrm{L}_{\zeta}(b) = \Big(\sum_{i=0}^{\min(r,s)}\mathrm{L}_{\zeta}(r+s-2i) \Big) \mathbbm{q}\mathbb{FL}(\mathrm{L}(c)\mathrm{L}(d))$$ if $r+s<\ell$, and $$\mathrm{L}_{\zeta}(a)\mathrm{L}_{\zeta}(b) =$$ $$\Big(\sum_{i=0}^{\ell - \max(r,s)}\mathrm{L}_{\zeta}(|r-s|+2i) + \sum_{i=0}^{\lceil (r+s-\ell)/2\rceil} (2-\delta_{(r+s-2i),\ell-1}) \mathrm{L}_{\zeta}(2\ell + 2i - r - s -2) \Big) \mathbbm{q}\mathbb{FL}(\mathrm{L}(c)\mathrm{L}(d))+$$ $$+\Big(\sum_{i=0}^{\lceil (r+s-\ell)/2\rceil}(1-\delta_{(r+s-2i),\ell-1})\mathrm{L}_{\zeta}(r+s-2i-\ell)\Big) \mathbbm{q}\mathbb{FL}(\mathrm{V} \mathrm{L}(c)\mathrm{L}(d))$$ if $r+s\geq \ell$, where $\mathrm{V}=\mathbb{T}(1)$ in $\Ver^{\sigma}_{p^{n-1}}$.
\end{Proposition}
\begin{proof}
The first equation follows directly by combining theorem \ref{thm:VerlindeSteinbergtensor} together with the tensor product formula for tilting modules given in \cite[Lemma 4.1]{STWZ}. The second equation makes, in addition, use of examples \ref{ex:compositionfactors}.
\end{proof}

\begin{Remark}
As in \cite[Corollary 4.51]{BEO}, the proposition above can straightforwardly be recast as expressing the $\mathbb{Z}_+$-ring $Gr(\Ver_{p^{(n)}}^{\zeta})$ as an extension of $Gr(\Ver^{\sigma}_{p^{n-1}})$.
\end{Remark}

Observe that the quotient $Q_{p^{(n)}}/Q_{p^{(n-1)}}(x)$ of the Chebyshev polynomials of the second kind is an integer polynomial. This follows easily by comparing their complex roots.

\begin{Proposition}
The map $\mathbb{Z}[x]/\big(Q_{p^{(n)}}/Q_{p^{(n-1)}}(x)\big)\rightarrow Gr(\Ver_{p^{(n)}}^{\zeta})$ given by $x\mapsto \mathrm{L}_{\zeta}(1)$ defines an isomorphism of rings.
\end{Proposition}
\begin{proof}
The assignment $x\mapsto \mathrm{L}_{\zeta}(1)$ induces a ring homomorphism by proposition \ref{prop:nTLobjects}. Further, the ring $\mathbb{Z}[x]/\big(Q_{p^{(n)}}/Q_{p^{(n-1)}}(x)\big)$ has a basis given by $x^k$, $0\leq k\leq p^{(n)}-p^{(n-1)}-1$, and the Grothendieck ring $Gr(\Ver_{p^{(n)}}^{\zeta})$ has a basis given by $\mathrm{L}_{\zeta}(a)$, $0\leq a\leq p^{(n)}-p^{(n-1)}-1$. It follows from proposition \ref{prop:simpletensorproduct} that, in these bases, the map of $\mathbb{Z}$-modules $$\mathbb{Z}[x]/\big(Q_{p^{(n)}}/Q_{p^{(n-1)}}(x)\big)\rightarrow Gr(\Ver_{p^{(n)}}^{\zeta})$$ is upper triangular with $1$'s on the diagonal. It is therefore an isomorphism, which concludes the proof.
\end{proof}

\begin{Remark}
It follows from the last proposition that the Grothendieck ring $Gr(\Ver_{p^{(n)}}^{\zeta,+})$ is identified with the image of the ring homomorphism $\mathbb{Z}[x^2]\rightarrow \mathbb{Z}[x]/\big(Q_{p^{(n)}}/Q_{p^{(n-1)}}(x)\big)$. But, upon inspection of the definitions, we find that, except in the case $(p,n)=(2,2)$ and $\ell$ odd, the integer polynomial $Q_{p^{(n)}}/Q_{p^{(n-1)}}(x)$ can be written as $P^+_{p^{(n)}}(x^2-2)$, with $P^+_{p^{(n)}}(x)$ an integer polynomial, so that $$Gr(\Ver_{p^{(n)}}^{\zeta,+})\cong \mathbb{Z}[x]/P^+_{p^{(n)}}(x).$$ More precisely, the roots of the polynomial $Q_{p^{(n)}}/Q_{p^{(n-1)}}(x)$ are $2\cos(k\pi/p^{(n)})$ with $1\leq k\leq p^{(n)}-1$ such that $p\nmid k$. It therefore follows that the complex roots of the polynomial $P^+_{p^{(n)}}(x)$ are $2\cos(2k\pi/p^{(n)})$ with $1\leq k < p^{(n)}/2$ such that $p\nmid k$. With $(p,n)=(2,2)$ and $\ell$ odd, there is also a polynomial $P^+_{p^{(n)}}$ such that $Q_{p^{(n)}}/Q_{p^{(n-1)}}(x)=xP^+_{p^{(n)}}(x^2-2)$, so that $Gr(\Ver_{p^{(n)}}^{\zeta,+})\cong \mathbb{Z}[x]/\big(xP^+_{p^{(n)}}(x)\big)$.

Now, if $\ell$ and $p$ are both odd, we have that $$Gr(\Ver_{p^{(n)}}^{\zeta})\cong \Big(\mathbb{Z}[x]/\big(P^+_{p^{(n)}}(x)\big)\Big)[g]/(g^2-1)$$ from corollary \ref{cor:fermion}. Further, when $p=\ell$ is odd, so that $\zeta =\pm 1$, we find that there are isomorphisms $Gr(\Ver_{p^{n}}^{+})\cong \mathbb{Z}[2\cos(2\pi/p^{n})]=\mathcal{O}_{p^n}$ and $Gr(\Ver_{p^{n}})\cong \mathcal{O}_{p^n}[g]/(g^2-1)$ as $P^+_{p^{n}}(x)$ is irreducible. This agrees with the result of \cite{BEO}. When $p=\ell = 2$, we find $Gr(\Ver_{2^{n}}^{+})\cong \mathbb{Z}[2\cos(2\pi/2^{n})]$ and $Gr(\Ver_{2^{n}})\cong \mathbb{Z}[2\cos(\pi/2^{n})]$ as both $Q_{2^{n}}/Q_{2^{n-1}}(x)$ and $P^+_{2^{n}}(x)$ are irreducible, which agrees with the results of \cite{BE}.
\end{Remark}

\subsection{The Cartan Matrix and the Blocks \texorpdfstring{of $\Ver_{p^{(n)}}^{\zeta}$}{}}

For completeness, we recall \cite[Proposition 5.4]{STWZ}. Before doing so, we need to review two notions. Firstly, the Cartan matrix of $\Ver_{p^{(n)}}^{\zeta}$ has entries $c_{ij}$ defined by $$c_{ij}=\dim Hom_{\Ver_{p^{(n)}}^{\zeta}}(\mathbb{T}_{\zeta}(i),\mathbb{T}_{\zeta}(j))$$ for $p^{(n-1)}-1\leq i,j\leq p^{(n)}-2$. This matrix records the simple factors in the composition series for the indecomposable projective objects of $\Ver_{p^{(n)}}^{\zeta}$. Secondly, given any integer $p^{(n-1)}\leq a\leq p^{(n)}-1$ with $p\ell$-adic expansion $a=[a_{n-1},a_{n-2},...,a_0]_{p\ell}$, we say that $b$ is a descendant of $a$ if $b=[a_{n-1},\pm a_{n-2},...,\pm a_0]_{p\ell}$ for some choice of signs.

\begin{Proposition}[\cite{STWZ}]\label{prop:Cartanmatrix}
The entry $c_{ij}$ of the Cartan matrix for $\Ver_{p^{(n)}}^{\zeta}$ is given by the number of common descendants between $i+1$ and $j+1$.
\end{Proposition}

\begin{Remark}\label{Rem:ExtendedCartan}
Slightly more generally, it follows from \cite[Theorem 3.25 (basis)]{STWZ} that the integer $$\dim Hom_{\Ver_{p^{(n)}}^{\zeta}}(\mathbb{T}_{\zeta}(i),\mathbb{T}_{\zeta}(j))$$ for any $0\leq i,j\leq p^{(n)}-2$ corresponds to the number of common descendants between $i+1$ and $j+1$.
\end{Remark}

\begin{Example}\label{ex:technicalcompositionseries}
Let $n\geq 2$ and $p>2$. For later use, we now compute the composition series of 
$\mathbb{T}_{\zeta}(3\ell -2) = \mathbb{T}_{\zeta}([2,\ell-1]_{p,\ell}-1)$. Firstly, recall from theorem \ref{thm:VerlindeSteinbergtensor} that the projective cover of $\mathbb{T}_{\zeta}(\ell-2)$ is $\mathbb{T}_{\zeta}([1,p-1,...,p-1,1]_{p,\ell}-1)$ and the projective cover of $\mathbb{U}$ is $\mathbb{T}_{\zeta}([1,p-1,...,p-1,p-2,\ell-1]_{p,\ell}-1)$. It follows from the preceding remark that $$\dim Hom_{\Ver_{p^{(n)}}^{\zeta}}(\mathbb{T}_{\zeta}([1,p-1,...,p-1,1]_{p,\ell}-1),\mathbb{T}_{\zeta}([2,\ell-1]_{p,\ell}-1)) = 1,$$ $$\dim Hom_{\Ver_{p^{(n)}}^{\zeta}}(\mathbb{T}_{\zeta}([1,p-1,...,p-2,\ell-1]_{p,\ell}-1),\mathbb{T}_{\zeta}([2,\ell-1]_{p,\ell}-1)) = 2.$$ But, $\mathbb{T}_{\zeta}(3\ell -2)$ is self-dual and indecomposable, so that its composition series must be given by $[\mathbb{U}, \mathbb{T}_{\zeta}(\ell-2),\mathbb{U}]$. This last fact also holds if $p=2$ under the provision that $n\geq 3$.
\end{Example}

Let $\mathcal{C}$ be an arbitrary finite category. Consider the weakest equivalence relation on the set of isomorphisms classes of indecomposable objects of $\mathcal{C}$ generated by declaring that two indecomposable objects are equivalent if there exists a non-zero morphism between them. An equivalence class for this relation is a block of $\mathcal{C}$.

\begin{Proposition}\label{prop:tiltinginblocks}
The objects $\mathbb{T}_{\zeta}(i)$ and $\mathbb{T}_{\zeta}(j)$ of $\Ver_{p^{(n)}}^{\zeta}$ lie in the same block if and only if the $p\ell$-adic expansions of $i+1 = [a_{n-1},...,a_0]_{p,\ell}$ and $j+1 = [b_{n-1},...,b_0]_{p,\ell}$ have the same number, say $k$, of zeroes at the end and either one of the following holds:
\begin{itemize}
    \item the last non-zero digits $a_k$ and $b_k$ agree and $i-j$ is divisible by $2p^{(k)}$;
    \item the last non-zero digits $a_k$ and $b_k$ sum to $p$ if $k\neq 0$, or $\ell$ if $k=0$, and $i+j+2$ is divisible by $2p^{(k)}$.
\end{itemize}
\end{Proposition}
\begin{proof}
We have recalled above that $Hom_{\Ver_{p^{(n)}}^{\zeta}}(\mathbb{T}_{\zeta}(i),\mathbb{T}_{\zeta}(j))$ is non-zero if and only if $i+1$ and $j+1$ have a common descendant. This implies that $\mathbb{T}_{\zeta}(i)$ and $\mathbb{T}_{\zeta}(j)$ lies in the same block if $i+1$ and $j+1$ share an iterated descendant, i.e.\ a descendant of a descendant ...\ of a descendant. As a consequence, it is straightforward to check that the conditions of the proposition are necessary. Furthermore, it is easy to see that every block of $\Ver_{p^{(n)}}^{\zeta}$ contains an object of the form $\mathbb{T}_{\zeta}(\alpha_kp^{(k)}-1)$ for some integers $0\leq k\leq n-1$ and $0\leq \alpha_k < p$ if $k>0$ and $0\leq \alpha_k < \ell$ if $k=0$. We claim that this object is unique. Namely, using the conditions of the proposition, one checks readily that if, for some integers $\alpha_k$ and $\beta_k$ as above, both $\mathbb{T}_{\zeta}(\alpha_kp^{(k)}-1)$ and $\mathbb{T}_{\zeta}(\beta_kp^{(k)}-1)$ are iterated descendants of the same $\mathbb{T}_{\zeta}(i)$, then we must have $\alpha_k=\beta_k$. This concludes the proof.
\end{proof}

As an immediate consequence of the above result, we obtain the following proposition.

\begin{Proposition}\label{prop:generalblocksVerlinde}
The category $\Ver_{p^{(n)}}^{\zeta}$ has $(n-1)(p-1)+(\ell-1)$ blocks. More specifically, it has $(p-1)$ blocks of size $1$ formed by the projective simple objects $\mathbb{T}_{\zeta}(i)$ with $i+1$ divisible by $p^{(n-1)}$, $(p-1)$ blocks of size $(p-1)$ formed by the projective objects $\mathbb{T}_{\zeta}(i)$ with $i+1$ divisible by $p^{(n-2)}$ but not $p^{(n-1)}$, ..., and $(\ell -1)$ blocks of maximal size $p^{n-1}-p^{n-2}$ formed by the projective objects $\mathbb{T}_{\zeta}(i)$ with $i+1$ not divisible by $p^{(1)}=\ell$.
\end{Proposition}

\begin{Remark}
It was shown in \cite[Section 4.14]{BEO} using the translation principle of \cite[II.7.9 and II.E.11]{Jan} that all the blocks of $\Ver^{+1}_{p^n}$ of the same size are equivalent. We suspect that the translation principle also holds in the mixed case. This would give a conceptual proof that the blocks of the category $\Ver_{p^{(n)}}^{\zeta}$ that have the same size are equivalent. In the next section, we give a direct argument that yields a stronger result.
\end{Remark}

\subsection{The Principal Blocks}

A principal block of $\Ver_{p^{(n)}}^{\zeta}$ is, by definition, one of the blocks of maximal size. It was shown in \cite[Proposition 4.54]{BEO} that all the blocks of $\Ver^{\sigma}_{p^{n}}$ that have the same size are equivalent. In fact, they argue that this holds even if we let $n$ vary. We now prove an analogous version of this result for the categories $\Ver^{\zeta}_{p^{(n)}}$, where not only $n$ but also $\zeta$ is allowed to vary. Before proving this result, we need to recall a few combinatorial definitions from \cite[Definition 2.13]{STWZ}.

Let us fix a (non-negative) integer $v$ with $p\ell$-adic expansion $v=[a_j,...,a_0]_{p\ell}$. Let also $S\subseteq \{0,...,j-1\}$ be a subset. The subset $S$ is called \textit{down-admissible} for $v$ if, for every $s\in S$ with $s-1\not\in S$, $a_s\neq 0$, and, whenever $s\in S$ with $a_{s+1}=0$, then $s+1\in S$. Provided that $S$ is down-admissible for $v$, the downward reflection of $v$ along $S$ is defined by $$v[S]:=[a_j,\epsilon_{j-1}a_{j-1},...,\epsilon_0a_0]_{p\ell},\ \epsilon_i:=\begin{cases}-1, & s\in S,\\ +1, & s\not\in S.\end{cases}$$ Any down-admissible subset $S$ for $v$ admits a finest partition $S=\sqcup_iS_i$ into down-admissible sets of consecutive integers, which we refer to as \textit{minimal down-admissible stretches}. Likewise, a subset $S'\subseteq \{0,...,j\}$ is called \textit{up-admissible} for $v$ if, for every $s\in S'$ with $s-1\not\in S'$, $a_s\neq 0$, and, whenever $s\in S'$ with $a_{s+1}=p-1$, then $s+1\in S'$. Provided that $S'$ is up-admissible for $v$, the upward reflection of $v$ along $S'$ is defined by $$v(S'):=[a'_{j+1},a'_{j},...,a'_0]_{p\ell},\ a'_k:=\begin{cases}a_k, & k,k-1\not\in S',\\ a_k+2, & k\not\in S',k-1\in S',\\ -a_k, & k\in S',\end{cases}$$ where, by convention, $a_{j+1}=0$. Any up-admissible subset $S'$ for $v$ admits a finest partition $S'=\sqcup_i S'_i$ into up-admissible sets of consecutive integers, which we refer to as \textit{minimal up-admissible stretches}. Finally, for any up-admissible set $S$ for $v$, its down-admissible hull $\overline{S'}$ is the smallest down-admissible set containing $S'$ if it exists.

\begin{Proposition}\label{prop:alltheblocks}
If $p>2$, all the blocks of $\Ver_{p^{(n)}}^{\zeta}$ and of $\Ver_{p^{n}}^{\sigma}$ that have the same size are equivalent. If $p=2$, this holds for all the blocks of size strictly greater than $1$.
\end{Proposition}
\begin{proof}
For definitiveness, we will show that the principal blocks of $\Ver_{p^{(n)}}^{\zeta}$ are equivalent to the principal block of $\Ver_{p^{n}}^{\sigma}$. The blocks of other size can be dealt with using similar techniques.

It follows from propositions \ref{prop:tiltinginblocks} and \ref{prop:generalblocksVerlinde} that for any integer $1\leq a_0\leq \ell-1$ there is a corresponding maximal block of $\Ver_{p^{(n)}}^{\zeta}$ containing exactly the indecomposable projective objects $\mathbb{T}_{\zeta}(i)$, $i\in B_{\zeta}(a_0)$ with \begin{align*}B_{\zeta}(a_0):= &\{p^{(n-1)}-1\leq i\leq p^{(n)}-2|i+1\equiv a_0(\mathrm{mod}\ \ell) \ \mathrm{and}\ 2\ell\mid (i+1-a_0)\}\\ & \cup \{p^{(n-1)}-1\leq i\leq p^{(n)}-2|i+1\equiv -a_0(\mathrm{mod}\ \ell) \ \mathrm{and}\ 2\ell\mid (i+1+a_0)\}.\end{align*} We will compare such a block with the principal block of $\Ver_{p^{n}}^{\sigma}$ containing precisely the indecomposable projective objects $\mathbb{T}(i)$, $i\in B(1)$ with \begin{align*}B(1):=&\{p^{n-1}-1\leq i\leq p^{n}-2|i+1\equiv 1(\mathrm{mod}\ p) \ \mathrm{and}\ 2p\mid i\} \\ &\cup\{p^{n-1}-1\leq i\leq p^{n}-2|i+1\equiv -1(\mathrm{mod}\ p) \ \mathrm{and}\ 2p\mid (i+2)\}.\end{align*} For simplicity, we will assume that $a_0$ is odd. Observe that there is a bijection $f:B_{\zeta}(a_0)\cong B(1)$ given by $$[a_{n-1},...a_1,a_0-1]_{p\ell}\mapsto [a_{n-1},...a_1,0]_{p}.$$

In order to check that the block corresponding to $B_{\zeta}(a_0)$ with $a_0$ odd is equivalent to the block corresponding to $B(1)$, it is enough to construct an isomorphism of algebras $$A_{\zeta}(a_0):=End_{\Ver^{\zeta}_{p^{(n)}}}\Big(\bigoplus_{i\in B_{\zeta}(a_0)}\mathbb{T}_{\zeta}(i)\Big)\xrightarrow{\sim} A(1):=End_{\Ver^{\sigma}_{p^{n}}}\Big(\bigoplus_{i\in B(1)}\mathbb{T}(i)\Big).$$ But, the canonical functors $\Tilt^{\zeta}_{p^{(n)}}\rightarrow\Ver^{\zeta}_{p^{(n)}}$ and $\Tilt^{\sigma}_{p^{n}}\rightarrow\Ver^{\sigma}_{p^{n}}$ are full on the tilting modules, so the above endomorphisms may be taken within the categories $\Tilt^{\zeta}$ and $\Tilt^{\sigma}$. Then, thanks to \cite[Theorem 3.25]{STWZ}, the algebras $A_{\zeta}(a_0)$ and $A(1)$ admit explicit descriptions. In particular, they are generated by elements $\mathrm{E}_{v-1}$, $\mathrm{D}_S\mathrm{E}_{v-1}$, and $\mathrm{U}_{S'}\mathrm{E}_{v-1}$ where $v\in B_{\zeta}(a_0)$, resp. $v\in B(1)$, $S$ is a minimal down-admissible stretch for $v$, and $S'$ is a minimal up-admissible stretch for $v$. A complete set of relations for these algebras is supplied in \cite[Theorem 3.25]{STWZ}.

We claim that the assignment $\phi$ given by $\mathrm{E}_{v-1}\mapsto \mathrm{E}_{f(v)-1}$, $\mathrm{D}_S\mathrm{E}_{v-1}\mapsto \mathrm{D}_S\mathrm{E}_{f(v)-1}$, $\mathrm{U}_{S'}\mathrm{E}_{v-1}\mapsto \mathrm{U}_{S'}\mathrm{E}_{f(v)-1}$ induced by the bijection $f$ defined above is an isomorphism of algebras $A_{\zeta}(a_0)\cong A(1)$. This follows straightforwardly from inspecting the relations given in \cite[Theorem 3.25]{STWZ} and \cite[Equation (3-10)]{STWZ} together with the fact that the linear coefficients that appear do agree. This is a consequence of the equality $$(-\zeta)^{\ell} = -\sigma = (-\sigma)^{p},$$ as $\zeta^{\ell} = - (-1)^N$ with $N$ the order of $\zeta$.
\end{proof}

The next result follows from the fact that $\Ver^{+1}_{p^{(n)}}$ and $\Ver^{-1}_{p^{(n)}}$ are equivalent as tensor categories, which follows from the corresponding fact for the monoidal categories of tilting modules discussed in section \ref{sub:tilting}, together with \cite[Proposition 4.54]{BEO}.

\begin{Corollary}
The blocks of the same size in the categories $\Ver_{p^{(n)}}^{\zeta}$ for any $n$ and $\zeta$ are equivalent.
\end{Corollary}

\begin{Remark}
If follows from the corollary above that all the results of \cite[Section 4.16]{BEO} concerning the $\mathrm{Ext}^1$ groups between the simple objects of $\Ver^{\sigma}_{p^n}$ carry over to the mixed case. Furthermore, so do the many results and conjectures of \cite{BE2} on the cohomology of the tensor categories $\Ver^{\sigma}_{p^n}$.
\end{Remark}

\subsection{Serre Subcategories}

Recall that a tensor subcategory is a full subcategory that is closed under subquotients, tensor products, and taking duals. A Serre subcategory is subcategory that is closed under subquotients and extensions.

\begin{Proposition}
Any tensor subcategory $\mathcal{C}$ of $\Ver_{p^{(n+1)}}^{\zeta}$ that contains $\Ver^{\sigma}_{p^{n}}$ is either $\Ver^{\sigma}_{p^{n}}$, $\Ver_{p^{(n+1)}}^{\zeta}$, $\Ver_{p^{(n+1)}}^{\zeta, +}$ if $\ell$ is even, or $\Ver_{4}^{\sigma, +}$ provided that $(p,n+1)=(2,2)$. In particular, $\Ver^{\sigma}_{p^{n}}$ if $\ell\neq 2$ and $(p,n+1)\neq (2,2)$, $\Ver_{4}^{\sigma, +}$ if $(p,n+1)=(2,2)$, and $\Ver_{p^{(n+1)}}^{\zeta, +}$ are Serre subcategories of $\Ver_{p^{(n+1)}}^{\zeta}$.
\end{Proposition}
\begin{proof}
We will consider the simple projective object $\mathbb{T}(p^{n-1}-1)$ of $\Ver^{\sigma}_{p^{n}}$. The corresponding simple object $\mathbbm{q}\mathbb{FL}(\mathbb{T}(p^{n-1}-1))$ of $\Ver_{p^{(n+1)}}^{\zeta}$ is not projective. By corollary \ref{cor:coverqFLsimple}, its projective cover is given by $\mathbb{T}_{\zeta}(p^{(n)}+\ell-2)$. Further, it follows from proposition \ref{prop:Cartanmatrix} that $\mathbbm{q}\mathbb{FL}(\mathbb{T}(p^{n-1}-1))$ appears in the composition series of $\mathbb{T}_{\zeta}(p^{(n)}+\ell-2)$ exactly twice.

Now, let $Q$ be the projective cover of $\mathbbm{q}\mathbb{FL}(\mathbb{T}(p^{n-1}-1))$ in $\mathcal{C}$. It exists because any tensor subcategory of a finite category is finite \cite{EGNO}. Then, there are surjective homomorphisms in $\Ver^{\zeta}_{p^{(n+1)}}$ $$\mathbb{T}_{\zeta}(p^{(n)}+\ell-2)\twoheadrightarrow Q\twoheadrightarrow\mathbbm{q}\mathbb{FL}(\mathbb{T}(p^{n-1}-1)).$$ We claim that it follows from \cite[Lemma 6.4.2]{EGNO} that $Q$ is a self-dual object. Namely, if the tensor category $\Ver_{p^{(n+1)}}^{\zeta}$ has a non-trivial invertible object, it has exactly one and it does not lie in the same block as $\mathbbm{1}$. Thus, the distinguished invertible object of $\Ver_{p^{(n+1)}}^{\zeta}$ is the monoidal unit $\mathbbm{1}$ and the claim is proven. As a consequence, both the socle and the head of $\mathbb{T}_{\zeta}(p^{(n)}+\ell-2)$ are given by the simple object $\mathbbm{q}\mathbb{FL}(\mathbb{T}(p^{n-1}-1))$, which implies that $Q$ is isomorphic to either $\mathbb{T}_{\zeta}(p^{(n)}+\ell-2)$ or $\mathbbm{q}\mathbb{FL}(\mathbb{T}(p^{n-1}-1))$.

Let us begin by considering the case $Q=\mathbb{T}_{\zeta}(p^{(n)}+\ell-2)$. If $p=2$ and $n+1=2$, then the composition series of $\mathbb{T}_{\zeta}(p^{(n)}+\ell-2)=\mathbb{T}_{\zeta}(2\ell-2)$ is given by $[\mathbbm{1},\mathbbm{1}]$, which generates a tensor subcategory of $\Ver_{p^{(n)}}^{\zeta}$ that is equivalent to $\Ver_4^{\sigma,+}$ and contains $\Ver_2^{\sigma}=\mathrm{Vec}$. One checks directly that the only other tensor subcategory of $\Ver_{p^{(n)}}^{\zeta}$ is $\Ver_{p^{(n)}}^{\zeta,+}$ when $p=2$ and $n+1=2$. We will therefore assume that $(p,n+1)\neq (2,2)$. Provided that $(p,n+1)\neq (2,2)$, the composition series of $\mathbb{T}_{\zeta}(p^{(n)}+\ell-2)$ must contain $\mathrm{L}_{\zeta}(p^{(n)}-\ell -2)=\mathrm{L}_{\zeta}([p-1,...,p-1,p-2,\ell-2]_{p\ell})$. This follows from proposition \ref{prop:Cartanmatrix} and theorem \ref{thm:VerlindeSteinbergtensor} as $p^{(n)}-\ell-1$ and $p^{(n)} + \ell +1$ share a descendant. But, as $\mathcal{C}$ is closed under subquotients and contains $\Ver^{\sigma}_{p^n}$, we find from proposition \ref{prop:simpletensorproduct} that, if $\ell > 2$, the simple object $\mathrm{L}_{\zeta}(\ell-2)$ of $\Ver_{p^{(n+1)}}^{\zeta}$ must lie in $\mathcal{C}$. It then follows that $\mathcal{C}$ must contain all the simple objects of $\Ver_{p^{(n+1)}}^{\zeta,+}$ if $\ell$ is even, and $\Ver_{p^{(n+1)}}^{\zeta}$ if $\ell$ is odd. (Note that when $\ell=2$, this conditions is automatic.) Finally, note that the projective object $\mathbb{T}_{\zeta}(2p^{(n)}-2)$ of $\Ver_{p^{(n+1)}}^{\zeta}$ must be a projective object in $\mathcal{C}$ as it is a direct summand of $Q\otimes Q^*$. Therefore, the projective cover of any simple object $S$ in $\mathcal{C}$ is a direct summand of $S\otimes \mathbb{T}_{\zeta}(2p^{(n)}-2)$, and therefore agrees with the projective cover of $S$ in $\Ver_{p^{(n+1)}}^{\zeta}$. This finishes the proof in this case.

In the remaining case, that is, when $Q=\mathbbm{q}\mathbb{FL}(\mathbb{T}(p^{n-1}-1))$, then all the simple objects of $\mathcal{C}$ must be contained in $\Ver^{\sigma}_{p^{n}}$. Namely, otherwise, we would obtain a contradiction using proposition \ref{prop:simpletensorproduct}.
Furthermore, as $Q=\mathbbm{q}\mathbb{FL}(\mathbb{T}(p^{n-1}-1))$ is projective in $\mathcal{C}$, so is $\mathbb{T}_{\zeta}(2p^{(n-1)}-2)$. This implies that for every simple object $S$ of $\mathcal{C}$, $S\otimes \mathbb{T}_{\zeta}(2p^{(n-1)}-2)$ is projective, which establishes that $\mathcal{C}=\Ver^{\sigma}_{p^{n}}$ in this case, and concludes the first part of the proof.

The second part follows from the first. The only part that is not obvious is that $\Ver^{\sigma}_{p^{n}}$ is a Serre subcategory if $\ell\neq 2$ and $(p,n+1)\neq (2,2)$. But, if $C$ is any object of $\Ver_{p^{(n+1)}}^{\zeta}$ that is not in $\Ver^{\sigma}_{p^n}$, then the tensor subcategory of $\Ver_{p^{(n+1)}}^{\zeta}$ generated by $C$ and $\Ver^{\sigma}_{p^n}$ is either $\Ver_{p^{(n+1)}}^{\zeta, +}$ or $\Ver_{p^{(n+1)}}^{\zeta}$ thanks to our hypotheses. In particular, the composition series of $C$ must contain a simple object that is in $\Ver_{p^{(n+1)}}^{\zeta}$ but not $\Ver^{\sigma}_{p^n}$.
\end{proof}

Combining the above proposition with \cite[Proposition 4.60]{BEO} when $p>2$ and \cite[Proposition 2.4]{BE} when $p=2$, we obtain the next result.

\begin{Corollary}
If $p>2$ and $\ell>2$, then $\Ver^{\sigma}_{p^{n}}$ is a Serre subcategory of $\Ver_{p^{(n+k)}}^{\zeta}$ for any $k\geq 1$. If $p=2$ and $\ell>2$, then $\Ver^{\sigma,+}_{2^{n}}$ is a Serre subcategory of $\Ver_{2^{(n+k)}}^{\zeta}$ for any $k\geq 1$.
\end{Corollary}

Combining the proposition above with \cite[Corollary 4.61]{BEO} (see also \cite[Corollary 2.6]{BE} for the case $p=2$), we obtain the following result.

\begin{Corollary}\label{cor:Serresubcats}
The tensor subcategories of $\Ver_{p^{(n+1)}}^{\zeta}$ are given by $\Ver_{p^{(n+1)}}^{\zeta}$, $\Ver_{p^{(n+1)}}^{\zeta,+}$, $\Ver^{\sigma}_{p^{m}}$, $\Ver_{p^{m}}^{\sigma,+}$ for $1\leq m\leq n$, $\Ver_{4}^{\sigma, +}$ if $p=2$ and $n=2$, $\mathrm{Vec}^{\sigma}(\mathbb{Z}/2)$ if $p>2$ or $n=1$, and $\mathrm{Vec}$.
\end{Corollary}

\subsection{The Symmetric Center}

Recall from remarks \ref{rem:standardtilting} and \ref{rem:standardVerlinde} that if $\sigma= +1$, so that $\sigma^{1/2}=\pm 1$, then it follows from Kauffman's skein relation that $\beta^2_{\mathbb{T}(1),\mathbb{T}(1)}= + Id_{\mathbb{T}(1)\otimes \mathbb{T}(1)}$, and $\Ver_{p^{n}}^{\sigma^{1/2}}$ is symmetric. On the other hand, if $p>2$ and $\sigma=-1$, then $\sigma^{1/2}$ is a primitive fourth root of unity, and it follows from Kauffman's skein relation that $\beta^2_{\mathbb{T}(1),\mathbb{T}(1)}= - Id_{\mathbb{T}(1)\otimes \mathbb{T}(1)}$, so that $\Ver_{p^{n}}^{\sigma^{1/2}}$ is not symmetric, its symmetric center is $\Ver_{p^{n}}^{\sigma^{1/2},+}$.

We presently investigate the symmetric center of $\Ver_{p^{(n+1)}}^{\zeta^{1/2}}$ when $\zeta\neq\pm 1$. When $n=0$, we are in the semisimple case, and it is well-known that $\Ver_{p^{(1)}}^{\zeta^{1/2}}$ has symmetric center $\mathrm{Vec}$ if $\zeta$ has even order and $\mathrm{sVec}$ otherwise. We will therefore assume that $n\geq 1$ from now on. Firstly, we give a conceptual proof that covers the case $\ell > 2$. Secondly, we give a computational argument to cover the remaining case $\ell = 2$.

\begin{Theorem}
With $\ell > 2$ and $n\geq 1$, the symmetric center of $\Ver_{p^{(n+1)}}^{\zeta^{1/2}}$ is $\Ver_{p^{n}}^{\sigma^{1/2}}$ if $N$ is odd and $\Ver_{p^{n}}^{\sigma^{1/2}, +}$ otherwise.
\end{Theorem}
\begin{proof}
We begin by observing that it is easy to check by hand that $\beta^2_{\mathbb{T}_{\zeta}(1),\mathbb{T}_{\zeta}(1)}$ and $\beta^2_{\mathbb{T}_{\zeta}(1),\mathbb{T}_{\zeta}(2)}$ are non-trivial. The symmetric center of $\Ver_{p^{(n+1)}}^{\zeta^{1/2}}$ is therefore a non-trivial proper tensor subcategory different from $\Ver_{p^{(n+1)}}^{\zeta^{1/2},+}$. But, in the case $(p,n+1)=(2,2)$, the only other non-trivial tensor subcategory is $\Ver_{4}^{\sigma, +}$ generated by $\mathbb{T}_{\zeta}(2\ell-2)$. It follows from \cite[Theorem 5.8]{STWZ} that the double braiding $\beta_{\mathbb{T}_{\zeta}(1),\mathbb{T}_{\zeta}(2\ell-2)}$ is non-trivial, so that the symmetric center of $\Ver_{2^{(2)}}^{\zeta^{1/2}}$ is $\Ver^{\sigma}_{2^{(1)}}=\mathrm{Vec}$. This agrees with the statement of the theorem because $\ell$ must be odd if $p=2$ and $\ell >2$. We will therefore always assume that $(p,n+1)\neq (2,2)$ for the remainder of this proof.

As $(p,n+1)\neq (2,2)$, we can consider the simple object $\mathbb{U}$ of lemma \ref{lem:simpleobjectU}. It will be enough to show that the double braiding $\beta^2_{\mathbb{T}_{\zeta}(1),\mathbb{U}}$ between the simple objects $\mathbb{T}_{\zeta}(1)$ and $\mathbb{U}$ satisfies $$\beta^2_{\mathbb{T}_{\zeta}(1),\mathbb{U}} = \begin{cases} +Id_{\mathbb{T}_{\zeta}(1)\otimes\mathbb{U}}, & \mathrm{if}\ N\ \mathrm{is\ odd},\\ -Id_{\mathbb{T}_{\zeta}(1)\otimes\mathbb{U}}, & \mathrm{if}\ N\ \mathrm{is\ even}.\end{cases}$$ Namely, $\mathbb{T}_{\zeta}(1)$ generates $\Ver_{p^{(n+1)}}^{\zeta^{1/2}}$ and $\mathbb{U}$ generates the tensor subcategory $\Ver_{p^{n}}^{\sigma^{1/2}}$. In particular, the collection $\mathbb{U}^{\otimes 2k}$ for arbitrary $k$ generates $\Ver_{p^{n}}^{\sigma^{1/2},+}$.

Computing $\beta^2_{\mathbb{T}_{\zeta}(1),\mathbb{U}}$ directly is difficult, we therefore use the first two lemmas below to reduce this problem to a manageable computation, which is carried out in the third lemma.

Given any object $X$ of $\Ver_{p^{(n+1)}}^{\zeta^{1/2}}$ and any endomorphism $h:\mathbb{T}_{\zeta}(1)\otimes X\rightarrow \mathbb{T}_{\zeta}(1)\otimes X$, we can consider its left partial trace $\mathrm{Tr}^L(h):X\rightarrow X$ defined by $$\mathrm{Tr}^L(h):X\xrightarrow{\mathrm{coev}\otimes Id} \mathbb{T}_{\zeta}(1)\otimes\mathbb{T}_{\zeta}(1)\otimes X\xrightarrow{Id\otimes h}\mathbb{T}_{\zeta}(1)\otimes\mathbb{T}_{\zeta}(1)\otimes X\xrightarrow{\mathrm{ev}\otimes Id} X.$$ In particular, observe that $\mathrm{Tr}^L(Id_{\mathbb{T}_{\zeta}(1)\otimes X})= \mathrm{dim}(\mathbb{T}_{\zeta}(1))\cdot Id_{X}$.

\begin{Lemma}
We have $\beta^2_{\mathbb{T}_{\zeta}(1),\mathbb{U}}=\pm Id_{\mathbb{T}_{\zeta}(1)\otimes\mathbb{U}}$ if and only if $$\mathrm{Tr}^L(\beta^2_{\mathbb{T}_{\zeta}(1),\mathbb{U}}) = \pm \mathrm{dim}(\mathbb{T}_{\zeta}(1))\cdot Id_{\mathbb{U}}.$$
\end{Lemma}
\begin{proof}
Note that $\mathbb{T}_{\zeta}(1)\otimes\mathbb{U}$ is a simple object by theorem \ref{thm:VerlindeSteinbergtensor}, so that we may treat $\beta^2_{\mathbb{T}_{\zeta}(1),\mathbb{U}}$ as a scalar. As $\ell > 2$, we know that the quantum dimension $\mathrm{dim}(\mathbb{T}_{\zeta}(1))=-[2]_{\zeta}\neq 0$, so that the linear map $$\mathrm{Tr}^L: End_{\Ver_{p^{(n+1)}}^{\zeta^{1/2}}}(\mathbb{T}_{\zeta}(1)\otimes\mathbb{U})\rightarrow End_{\Ver_{p^{(n+1)}}^{\zeta^{1/2}}}(\mathbb{U})$$ is in fact an isomorphism. This establishes the claim.
\end{proof}

We have seen in example \ref{ex:technicalcompositionseries} that the object $\mathbb{T}_{\zeta}(3\ell -2)$ has composition series given by $[\mathbb{U}, \mathbb{T}_{\zeta}(\ell-2),\mathbb{U}]$ (recall that we have assumed $(p,n+1)\neq (2,2)$). In particular, there exists non-zero morphisms $f:\mathbb{T}_{\zeta}(3\ell -2)\twoheadrightarrow\mathbb{U}$ and $g:\mathbb{U}\hookrightarrow\mathbb{T}_{\zeta}(3\ell -2)$.

\begin{Lemma}
We have that $\mathrm{Tr}^L(\beta^2_{\mathbb{T}_{\zeta}(1),\mathbb{U}})=\pm \mathrm{dim}(\mathbb{T}_{\zeta}(1))\cdot Id_{\mathbb{U}}$ if and only if $$\mathrm{Tr}^L\big((Id_{\mathbb{T}_{\zeta}(1)}\otimes g)\circ \beta^2_{\mathbb{T}_{\zeta}(1),\mathbb{U}}\circ (Id_{\mathbb{T}_{\zeta}(1)}\otimes f)\big)=\pm \mathrm{dim}(\mathbb{T}_{\zeta}(1))\cdot(g\circ f).$$
\end{Lemma}
\begin{proof}
The forward direction is immediate. As for the backward direction, observe that the linear map $$End_{\Ver_{p^{(n+1)}}^{\zeta^{1/2}}}(\mathbb{U})\rightarrow End_{\Ver_{p^{(n+1)}}^{\zeta^{1/2}}}(\mathbb{T}_{\zeta}(3\ell -2))$$ given by $h\mapsto g\circ h\circ f$ is a monomorphism because $f$ is an epimorphism and $g$ is a monomorphism.
\end{proof}

As we have seen above in lemma \ref{lem:simpleobjectU}, the composition series of $\mathbb{T}_{\zeta}(\ell)$ is given by $[\mathbb{T}_{\zeta}(\ell-2),\mathbb{U},\mathbb{T}_{\zeta}(\ell-2)]$. This implies that there are non-zero morphisms $\overline{f}:\mathbb{T}_{\zeta}(3\ell -2)\rightarrow\mathbb{T}_{\zeta}(\ell)$ and $\overline{g}:\mathbb{T}_{\zeta}(\ell)\rightarrow\mathbb{T}_{\zeta}(3\ell -2)$ such that $g\circ f = \overline{g}\circ \overline{f}$. It is then evident that $$\mathrm{Tr}^L\big((Id_{\mathbb{T}_{\zeta}(1)}\otimes g)\circ \beta^2_{\mathbb{T}_{\zeta}(1),\mathbb{U}}\circ (Id_{\mathbb{T}_{\zeta}(1)}\otimes f)\big)=\pm \mathrm{dim}(\mathbb{T}_{\zeta}(1))\cdot(g\circ f)$$ if and only if $$\mathrm{Tr}^L\big((Id_{\mathbb{T}_{\zeta}(1)}\otimes \overline{g})\circ \beta^2_{\mathbb{T}_{\zeta}(1),\mathbb{T}_{\zeta}(\ell)}\circ (Id_{\mathbb{T}_{\zeta}(1)}\otimes \overline{f})\big)=\pm \mathrm{dim}(\mathbb{T}_{\zeta}(1))\cdot(\overline{g}\circ \overline{f}).$$ Crucially, note that this last equality only involves objects (and consequently also morphisms) in the image of the canonical functor $F:\Tilt^{\zeta^{1/2}}\rightarrow\Ver_{p^{(n)}}^{\zeta^{1/2}}$. We can use this to perform an explicit computation.

\begin{Lemma}
We have $$\mathrm{Tr}^L\big((Id_{\mathbb{T}_{\zeta}(1)}\otimes \overline{g})\circ \beta^2_{\mathbb{T}_{\zeta}(1),\mathbb{T}_{\zeta}(\ell)}\circ (Id_{\mathbb{T}_{\zeta}(1)}\otimes \overline{f})\big)=\begin{cases} +\mathrm{dim}(\mathbb{T}_{\zeta}(1))\cdot(\overline{g}\circ \overline{f}), & \mathrm{if}\ N\ \mathrm{is\ odd},\\ -\mathrm{dim}(\mathbb{T}_{\zeta}(1))\cdot(\overline{g}\circ \overline{f}), & \mathrm{if}\ N\ \mathrm{is\ even}.\end{cases}$$
\end{Lemma}
\begin{proof}
The morphism $\mathrm{Tr}^L( \beta^2_{T_{\zeta}(1),T_{\zeta}(\ell)})$ is computed in \cite[Lemma 5.10]{STWZ}. More precisely, upon specializing $\mathbbm{v}\mapsto\zeta$ and using $\mathrm{T}_{\zeta}(\ell)=\mathrm{T}_{\zeta}([1,1]_{p,\ell}-1)$, we find $$\mathrm{Tr}^L( \beta^2_{\mathrm{T}_{\zeta}(1),\mathrm{T}_{\zeta}(\ell)}) =-[2]_{\zeta^{\ell+1}}\cdot Id_{\mathrm{T}_{\zeta}(\ell)}+s_{(\ell+1)}(\zeta)\cdot (\widehat{g}\circ\widehat{f}) +Rest ,$$ where $s_{(\ell+1)}(\zeta)$ is a scalar, and $\widehat{f}:\mathrm{T}_{\zeta}(\ell)\rightarrow \mathrm{T}_{\zeta}(\ell-2)$, $\widehat{g}:\mathrm{T}_{\zeta}(\ell-2)\rightarrow \mathrm{T}_{\zeta}(\ell)$ are morphisms. But, the set of descendants of $[1,1]_{p,\ell}$ is given by $\nabla\mathrm{supp}([1,1]_{p,\ell})=\{[1,1]_{p,\ell},[1,-1]_{p,\ell}\}$, so that it follows from the proof of \cite[Lemma 5.10]{STWZ} that there are no additional terms, i.e.\ $Rest = 0$.

Thus, upon taking the image under $F$, precomposing with $\overline{f}$ and postcomposing with $\overline{g}$, we obtain $$\mathrm{Tr}^L\big((Id_{\mathbb{T}_{\zeta}(1)}\otimes \overline{g})\circ \beta^2_{\mathbb{T}_{\zeta}(1),\mathbb{T}_{\zeta}(\ell)}\circ (Id_{\mathbb{T}_{\zeta}(1)}\otimes \overline{f})\big) =-[2]_{\zeta^{\ell+1}}\cdot (\overline{g}\circ \overline{f})+s_{(\ell+1)}(\zeta)\cdot (\overline{g}\circ \widehat{g}\circ\widehat{f}\circ \overline{f}).$$ Now, the morphism $$\widehat{f}\circ \overline{f}:\mathbb{T}_{\zeta}(3\ell-2)\rightarrow \mathbb{T}_{\zeta}(\ell-2)$$ must be zero as $\mathbb{T}_{\zeta}(\ell-2)$ is simple and the top of $\mathbb{T}_{\zeta}(3\ell-2)$ is $\mathbb{U}$. (In fact, this can also be shown directly in $\Tilt^{\zeta^{1/2}}$ using \cite[Theorem 3.25]{STWZ}.) Thus, we get $$\mathrm{Tr}^L\big((Id_{\mathbb{T}_{\zeta}(1)}\otimes \overline{g})\circ \beta^2_{\mathbb{T}_{\zeta}(1),\mathbb{T}_{\zeta}(\ell)}\circ (Id_{\mathbb{T}_{\zeta}(1)}\otimes \overline{f})\big) =-[2]_{\zeta^{\ell+1}}\cdot (\overline{g}\circ \overline{f}),$$ and the result follows from the equalities $$-[2]_{\zeta^{\ell+1}} = \begin{cases} +\mathrm{dim}(\mathbb{T}_{\zeta}(1)), & \mathrm{if}\ N\ \mathrm{is\ odd},\\ -\mathrm{dim}(\mathbb{T}_{\zeta}(1)), & \mathrm{if}\ N\ \mathrm{is\ even}.\end{cases}$$
\end{proof}

This completes the proof of the theorem.
\end{proof}

We now deal with the remaining case $\ell = 2$, that is $\zeta$ is a primitive fourth root of unity, and $n\geq 1$.

\begin{Proposition}
With $\ell = 2$, $p>2$, and $n\geq 1$, the symmetric center of $\Ver_{p^{(n+1)}}^{\zeta^{1/2}}$ is $\Ver_{p^{n}}^{\sigma^{1/2}, +}$.
\end{Proposition}
\begin{proof}
As in the proof above, it will suffice to compute that $$\beta^2_{\mathbb{T}_{\zeta}(1),\mathbb{U}} = -Id_{\mathbb{T}_{\zeta}(1)\otimes\mathbb{U}}.$$ We continue to write $f:\mathbb{T}_{\zeta}(2\ell-2)=\mathbb{T}_{\zeta}(4)\twoheadrightarrow\mathbb{U}$ and $g:\mathbb{U}\hookrightarrow\mathbb{T}_{\zeta}(4)= \mathbb{T}_{\zeta}(2\ell-2)$ for the non-zero morphisms. As above, it is enough to show that $$(Id_{\mathbb{T}_{\zeta}(1)}\otimes g)\circ \beta^2_{\mathbb{T}_{\zeta}(1),\mathbb{U}}\circ (Id_{\mathbb{T}_{\zeta}(1)}\otimes f) = -(Id_{\mathbb{T}_{\zeta}(1)}\otimes g)\circ (Id_{\mathbb{T}_{\zeta}(1)}\otimes f).$$ In particular, if we write $e=g\circ f$, a nilpotent endomorphism of $\mathbb{T}_{\zeta}(4)$, then the last equation may be rewritten as \begin{equation}\label{eq:braidingreduction}\beta^2_{\mathbb{T}_{\zeta}(1),\mathbb{T}_{\zeta}(4)}\circ (Id_{\mathbb{T}_{\zeta}(1)}\otimes e) = - (Id_{\mathbb{T}_{\zeta}(1)}\otimes e),\end{equation} which is in the image of the canonical functor $F:\Tilt^{\zeta^{1/2}}\rightarrow\Ver_{p^{(n)}}^{\zeta^{1/2}}$. Now, we have $\mathbb{T}_{\zeta}(1)^{\otimes 4} = \mathbb{T}_{\zeta}(4)\oplus \mathbb{T}_{\zeta}(2)\oplus \mathbb{T}_{\zeta}(2)$ and $\mathbb{T}_{\zeta}(4) = \mathbb{T}_{\zeta}(3)\otimes \mathbb{T}_{\zeta}(1)$ by \cite[Proposition 4.7]{STWZ}. In particular, as the morphism $I_3:= \mathbb{T}_{\zeta}(1)^{\otimes 3}\rightarrow \mathbb{T}_{\zeta}(1)^{\otimes 3}$ given by $$I_3:=Id_{\mathbb{T}_{\zeta}(1)^{\otimes 3}} - (\mathrm{coev}\otimes Id_{\mathbb{T}_{\zeta}(1)})\circ (Id_{\mathbb{T}_{\zeta}(1)}\otimes \mathrm{ev})- (Id_{\mathbb{T}_{\zeta}(1)}\otimes\mathrm{coev})\circ (\mathrm{ev}\otimes Id_{\mathbb{T}_{\zeta}(1)})$$ is an idempotent corresponding to the summand $\mathbb{T}_{\zeta}(3)$ of $\mathbb{T}_{\zeta}(1)^{\otimes 3}$, it follows that $$I_4:=I_3\otimes Id_{\mathbb{T}_{\zeta}(1)}:\mathbb{T}_{\zeta}(1)^{\otimes 4}\rightarrow \mathbb{T}_{\zeta}(1)^{\otimes 4}$$ is an idempotent corresponding to the summand $\mathbb{T}_{\zeta}(4)$ of $\mathbb{T}_{\zeta}(1)^{\otimes 4}$. Then, given that we have $End_{\Tilt^{\zeta}}(\mathbb{T}_{\zeta}(4))\cong \mathbbm{k}^{\oplus 2}$, it follows that the morphism $e$ corresponds to the nilpotent endomorphism $$E=I_4 \circ (Id_{\mathbb{T}_{\zeta}(1)}^{\otimes 2}\otimes\mathrm{coev})\circ (Id_{\mathbb{T}_{\zeta}(1)}^{\otimes 2}\otimes \mathrm{ev}) \circ I_4:\mathbb{T}_{\zeta}(1)^{\otimes 4}\rightarrow \mathbb{T}_{\zeta}(1)^{\otimes 4}.$$ Using these definitions, if follows that equation \eqref{eq:braidingreduction} above is equivalent to $$\beta^2_{\mathbb{T}_{\zeta}(1),\mathbb{T}_{\zeta}(1)^{\otimes 4}}\circ (Id_{\mathbb{T}_{\zeta}(1)}\otimes E) = -E.$$ We have explicitly checked that this last equality holds using a computer.
\end{proof}

\begin{Corollary}
With $n\geq 1$, the symmetric center of $\Ver_{p^{(n+1)}}^{\zeta^{1/2},+}$ is $\Ver_{p^{n}}^{\sigma^{1/2},+}$ if $N$ is odd and $\Ver_{p^{n}}^{\sigma^{1/2}}$ otherwise.
\end{Corollary}

\subsection{The Grothendieck Ring of the Stable Category}

We now compute the Grothendieck ring of the stable categories of $\Ver_{p^{(n)}}^{\zeta}$ and $\Ver_{p^{(n)}}^{\zeta,+}$. When $\zeta = \pm 1$, so that $p=\ell$, this was carried out in \cite[Proposition 4.61]{BEO}. Namley, it was shown that $GrStab(\Ver_{p^{n}}^{\pm 1, +})\cong \mathbb{F}_p[z]/z^{\frac{p^{n-1}-1}{2}}$ and $GrStab(\Ver_{p^{n}}^{\pm 1})\cong \mathbb{F}_p[z,g]/(z^{\frac{p^{n-1}-1}{2}},g^2-1)$ when $p\neq 2$. Further, when $p=2$, they proved that $GrStab(\Ver_{2^{n}}^{+ 1, +})\cong \mathbb{F}_2[z]/z^{2^{n-2}}$ and $GrStab(\Ver_{2^{n}}^{+1})\cong \mathbb{F}_p[z]/z^{2^{n-1}-1}$. We extend the above results to the case when $\zeta\neq \pm 1$.

\begin{Proposition}
With $\ell\neq p$, $p> 2$ and $n\geq 2$, we have $$GrStab(\Ver_{p^{(n)}}^{\zeta})\cong\mathbb{F}_p[z]/Q_{\ell}(z)^{p^{n-2}} \oplus\mathbb{F}_p[z,g]/(z^{\frac{p^{n-2}-1}{2}},g^2-1),$$ where the last term is omitted entirely if $n=2$. If $\ell$ is odd, we have $$GrStab(\Ver_{p^{(n)}}^{\zeta, +})\cong \mathbb{F}_p[z]/P^+_{\ell}(z)^{p^{n-2}} \oplus\mathbb{F}_p[z]/z^{\frac{p^{n-2}-1}{2}},$$ where the last term is omitted if $n=2$.
\end{Proposition}
\begin{proof}
It follows from proposition \ref{prop:alltheblocks} and the fact that $GrStab(\Ver_{p^{n}}^{\pm 1})$ is a vector space of $\mathbb{F}_p$ that the underlying abelian group of $GrStab(\Ver_{p^{(n)}}^{\zeta})$ is a vector space over $\mathbb{F}_p$. Moreover, its dimension is $p^{(n-1)}-1$. Namely, an easy inductive argument establishes that the dimension of the Grothendieck group of the maximal block of $\Ver_{p^{(n)}}^{\pm 1}$ is $p^{n-2}$. Now, the reduction modulo $p$ of this Grothendieck ring is $Gr(\Ver_{p^{(n)}}^{\zeta})\otimes \mathbb{F}_p\cong \mathbb{F}_{p}[x]/\big(Q_{p^{(n)}}/Q_{p^{(n-1)}}(x)\big)$, and the ideal of projective objects is generated by $\mathbb{T}_{\zeta}(p^{(n-1)}-1)$. But, in the Grothendieck ring of $\Ver_{p^{(n)}}^{\zeta}$, we have $[\mathbb{T}_{\zeta}(p^{(n-1)}-1)] = Q_{p^{(n-1)}}([\mathbb{T}_{\zeta}(1)])=Q_{p^{(n-1)}}(x)$. We therefore find $$GrStab(\Ver_{p^{(n)}}^{\zeta})\cong \mathbb{F}_{p}[x]/\big(Q_{p^{(n)}}/Q_{p^{(n-1)}}(x), Q_{p^{(n-1)}}(x)\big).$$ The expression above follows from the factorization of the polynomials $Q_{p^{(n-1)}}(x)$ and $Q_{p^{(n)}}(x)$ over $\mathbb{F}_{p}[x]$ given in \cite[Lemma 2.6]{HMP}.
\end{proof}

\begin{Remark}
With $p> 2$, $n\geq 2$, and $\ell$ even, it is still possible to identify the Grothendieck ring of the stable category of $\Ver_{p^{(n)}}^{\zeta, +}$ as the image of the canonical map $$\mathbb{F}_{p}[x^2]\rightarrow GrStab(\Ver_{p^{(n)}}^{\zeta})\cong \mathbb{F}_{p}[x]/\big(Q_{p^{(n)}}/Q_{p^{(n-1)}}(x), Q_{p^{(n-1)}}(x)\big).$$ However, there does not seem to be a convenient close formula for this image.
\end{Remark}

A similar argument yields the following result in the case $p=2$.

\begin{Proposition}
With $\ell\neq p$, $p=2$ and $n\geq 2$, we have $$GrStab(\Ver_{2^{(n)}}^{\zeta})\cong\mathbb{F}_2[z]/Q_{\ell}(z)^{2^{n-2}} \oplus\mathbb{F}_2[z]/z^{2^{n-2}-1}$$ and $$GrStab(\Ver_{2^{(n)}}^{\zeta, +})\cong \mathbb{F}_2[z]/P^+_{\ell}(z)^{2^{n-2}} \oplus\mathbb{F}_2[z]/z^{2^{n-3}},$$ where the last term is omitted if $n=2$.
\end{Proposition}

\subsection{Incompressibility}

Recall from \cite{BE} that a tensor category $\mathcal{C}$ is called incompressible if every tensor functor $F:\mathcal{C}\rightarrow \mathcal{D}$, where $\mathcal{D}$ is a tensor category, is an embedding, i.e.\ injective. Likewise, a braided tensor category is incompressible if every braided tensor functor out of it is an embedding. It was shown in \cite[Theorem 4.71]{BEO} that, with $p>2$, the tensor categories $\Ver^{+}_{p^{n}}$ are incompressible.

\begin{Proposition}\label{prop:incompressibility}
Let $p$ be an odd prime, and $\zeta$ a root of unity of odd order. The finite tensor category $\Ver^{\zeta,+}_{p^{(n)}}$ is incompressible.
\end{Proposition}
\begin{proof}
This follows exactly as in \cite[Theorem 4.71]{BEO} by tweaking their argument using our previous results.
\end{proof}

\begin{Remark}
If $\ell$ is even, then $\Ver^{\zeta,+}_{p^{(n)}}$ is not incompressible. Namely, in those cases, the tensor category $\mathrm{Vec}(\mathbb{Z}/2)$ is a tensor subcategory of $\Ver^{\zeta,+}_{p^{(n)}}$.
\end{Remark}

In the case $p=2$, the tensor categories $\Ver^{\sigma}_{2^n}$ are incompressible by \cite[Theorem 4.4]{BE}. We wonder whether a similar result holds in the mixed case. In this case, the proofs of neither \cite[Theorem 4.71]{BEO} nor \cite[Theorem 4.4]{BE} can be straightforwardly modified.

\begin{Question}\label{q:incompressibility2}
Let $p=2$, $n\geq 2$, and $\zeta$ be a root of unity. Is the finite tensor category $\Ver^{\zeta}_{p^{(n)}}$ incompressible?
\end{Question}

\begin{Remark}
When $p=2$ and $n=1$, the tensor category $\Ver^{\zeta}_{p^{(n)}}$ is not incompressible because it contains $\mathrm{Vec}(\mathbb{Z}/2)$ as a tensor subcategory. In this case, one can show that the tensor category $\Ver^{\zeta,+}_{p^{(n)}}$ is incompressible.
\end{Remark}

On the other hand, if we only wish to study the incompressibility of $\Ver^{\zeta^{1/2}}_{p^{(n)}}$ as a braided tensor categories, we can directly appeal to the incompressibility results of \cite{BEO,EO} instead of needing to reproduce their proofs. The connection is achieved through the next result, which generalizes \cite[Corollary 3.26]{DMNO}.

\begin{Proposition}
Let $F:\mathcal{A}\rightarrow\mathcal{B}$ be a braided tensor functor between finite braided tensor categories. If $\mathcal{Z}_{(2)}(\mathcal{A})$ is incompressible as a symmetric tensor category, then $F$ is an embedding.
\end{Proposition}
\begin{proof}
By the results of \cite[Section 6.3]{EGNO}, the tensor functor $F$ may be factored as a surjective tensor functor followed by an injective one. It is therefore enough to consider the case when $F$ is in addition surjective. Because the functor $F$ is exact, it has a right adjoint $F^*$. Moreover, $F^*$ inherits a lax braided monoidal structure. We write $A$ for the connected commutative algebra in $\mathcal{A}$ given by $A:=F^*(\mathbbm{1})$. We have $\mathbf{Mod}_{\mathcal{A}}(A)\simeq \mathcal{B}$ as braided tensor categories. This follows from \cite[Proposition 8.8.8]{EGNO} (see also \cite[Lemma 3.5]{DMNO}). Moreover, as in \cite[Proposition 3.22]{DMNO}, the commutative algebra $A$ lies in $\mathcal{Z}_{(2)}(\mathcal{A})$. In particular, we can consider the surjective symmetric tensor functor $(-)\otimes A:\mathcal{Z}_{(2)}(\mathcal{A})\rightarrow \mathbf{Mod}_{\mathcal{Z}_{(2)}(\mathcal{A})}(A)$. As $A$ is connected, $\mathbf{Mod}_{\mathcal{Z}_{(2)}(\mathcal{A})}(A)$ has simple monoidal unit, and it therefore follows from the incompressibility of $\mathcal{Z}_{(2)}(\mathcal{A})$ that $(-)\otimes A:\mathcal{Z}_{(2)}(\mathcal{A})\rightarrow \mathbf{Mod}_{\mathcal{Z}_{(2)}(\mathcal{A})}(A)$ is an equivalence. This implies in particular that for every objects $X,Y$ of $\mathcal{Z}_{(2)}(\mathcal{A})$, we have $$Hom_{\mathcal{Z}_{(2)}(\mathcal{A})}(X,Y)\cong Hom_A(X\otimes A,Y\otimes A)\cong Hom_{\mathcal{Z}_{(2)}(\mathcal{A})}(X,Y\otimes A).$$ Taking $Y=\mathbbm{1}$, and letting $X$ vary amongst the indecomposable projective objects of $\mathcal{Z}_{(2)}(\mathcal{A})$, it follows that $A=\mathbbm{1}$, so that $F$ is an equivalence.
\end{proof}

\begin{Corollary}
The braided tensor category $\Ver^{\zeta^{1/2}}_{p^{(n)}}$ is incompressible.
\end{Corollary}

\begin{Corollary}
The braided tensor category $\Ver^{\zeta^{1/2},+}_{p^{(n)}}$ is incompressible precisely if $\ell\not\equiv 2\mod 4$.
\end{Corollary}
\begin{proof}
On the one hand, when $N$ is odd, the symmetric center of $\Ver^{\zeta^{1/2},+}_{p^{(n)}}$ is $\Ver^{\sigma,+}_{p^{n-1}}$, which is incompressible. On the other hand, when $N$ is even, the symmetric center of $\Ver^{\zeta^{1/2},+}_{p^{(n)}}$ is $\Ver^{\sigma^{1/2}}_{p^{n-1}}$. If $p>2$, we have $\Ver^{\sigma^{1/2}}_{p^{n-1}}\cong \Ver^{\sigma^{1/2},+}_{p^{n-1}}\boxtimes \mathrm{Vec}^{\sigma^{1/2}}(\mathbb{Z}/2)$. The braided tensor category $\mathrm{Vec}^{\sigma^{1/2}}(\mathbb{Z}/2)$ is incompressible if and only if $\sigma^{1/2} = \zeta^{(\ell-2)\ell/2}\neq (-1)^{\ell}$. But, the equality $\zeta^{(\ell-2)\ell/2}=(-1)^{\ell}$ holds exactly when $\ell\equiv 2\mod 4$.
\end{proof}

\subsection{Higher Algebraic Properties}

In characteristic zero, one of the many applications of the semisimple Verlinde categories is that they give interesting classes in the (quantum) Witt group $\mathcal{W}itt$ of non-degenerate braided fusion categories \cite{DMNO}. More precisely, the Witt group $\mathcal{W}itt$ is the quotient of the monoid of equivalence classes of non-degenerate braided fusion categories by the Drinfeld centers, and the Verlinde categories are not Drinfeld centers as can be seen for instance from their central charges. In fact, as explained in \cite{SY}, the simplicity hypothesis may be removed. The above construction can also be generalized in a different direction:\ To any symmetric fusion category $\mathcal{E}$, one may associate a Witt group $\mathcal{W}itt(\mathcal{E})$ of braided fusion categories with symmetric center identified with $\mathcal{E}$ as in \cite{DNO}. Removing the characteristic zero and semisimplicity hypotheses, we find that, to any finite symmetric tensor category $\mathcal{E}$, one may associate a Witt group $\mathcal{W}itt(\mathcal{E})$ of finite braided tensor categories with symmetric center identified with $\mathcal{E}$. More precisely, given a finite tensor category equipped with a braided embedding $\mathcal{E}\rightarrow \mathcal{Z}(\mathcal{C})$, its relative Drinfeld center is $\mathcal{Z}(\mathcal{C},\mathcal{E})$, the centralizer of $\mathcal{E}$ in $\mathcal{Z}(\mathcal{C})$. Then, a finite braided tensor category $\mathcal{B}$ with symmetric center $\mathcal{E}$ gives the trivial class in $\mathcal{W}itt(\mathcal{E})$ if and only if it is a relative Drinfeld center. The above discussion provides motivation for the next result.

\begin{Proposition}
Let $\zeta \neq \pm 1$ be a fixed root of unity of odd order with square root $\zeta^{1/2}$, and assume $n\geq 2$. If $p$ is odd, then $\Ver^{\zeta^{1/2},+}_{p^{(n)}}$ is not a relative center.
\end{Proposition}
\begin{proof}
This follows from a Frobenius-Perron argument. Let $\mathcal{C}$ be a finite tensor category equipped with a braided tensor functor $$\Ver^{\zeta^{1/2},+}_{p^{(n)}}\rightarrow \mathcal{Z}(\mathcal{C}).$$ By incompressibility, this functor is necessarily an embedding. Moreover, given that we have $\mathcal{Z}_{(2)}(\Ver^{\zeta^{1/2},+}_{p^{(n)}})\simeq\Ver^{\sigma^{1/2},+}_{p^{n-1}}$, we also get a braided embedding $\Ver^{\sigma^{1/2},+}_{p^{n-1}}\rightarrow \mathcal{Z}(\mathcal{C})$. By \cite[Theorem 4.9]{Shi} and incompressibility of $\Ver^{\zeta^{1/2}, +}_{p^{(n)}}$ as a plain tensor category, we find $$\FPdim(\mathcal{Z}(\mathcal{C}))=\FPdim(\mathcal{C})^2\geq \FPdim(\Ver^{\zeta^{1/2},+}_{p^{(n)}})^2.$$ On the other hand, it follows from \cite[Theorem 4.9]{Shi} and \cite[Proposition 6.3.3]{EGNO} that $$\FPdim(\mathcal{Z}(\mathcal{C}, \Ver^{\sigma^{1/2},+}_{p^{n-1}}))=\frac{\FPdim(\mathcal{Z}(\mathcal{C}))}{\FPdim(\Ver^{\sigma^{1/2},+}_{p^{n-1}}))}\geq \frac{\FPdim(\Ver^{\zeta^{1/2},+}_{p^{(n)}})^2}{\FPdim(\Ver^{\sigma^{1/2},+}_{p^{n-1}}))}.$$ But, for instance by appealing to proposition \ref{prop:FPdim}, we find that $$\FPdim(\mathcal{Z}(\mathcal{C}, \Ver^{\sigma^{1/2},+}_{p^{n-1}}))\geq\frac{\FPdim(\Ver^{\zeta^{1/2},+}_{p^{(n)}})^2}{\FPdim(\Ver^{\sigma^{1/2},+}_{p^{n-1}}))}> \FPdim(\Ver^{\zeta^{1/2},+}_{p^{(n)}}),$$ so that the result follows.
\end{proof}

\begin{Remark}
Provided that question \ref{q:incompressibility2} has a positive answer, the above argument implies that, in the case $p=2$, $\Ver^{\zeta^{1/2}}_{p^{(n)}}$ is not a relative center.
\end{Remark}

\begin{Remark}
In characteristic zero, the semisimple Verlinde categories provide many interesting classes in the Witt groups $\mathcal{W}itt$ and $\mathcal{W}itt(\mathrm{sVec})$ (see \cite[Section 6.4]{DMNO} and \cite[Section 5.5]{DNO} respectively). Moreover, the relations between them are completely understood \cite[Theorem 5.21]{DNO}. By analogy, in positive characteristic, we expect that the mixed Verlinde categories $\Ver^{\zeta^{1/2}}_{p^{(n+1)}}$ provide interesting classes in the Witt groups $\mathcal{W}itt(\Ver^{\sigma^{1/2}}_{p^{n}})$ and $\mathcal{W}itt(\Ver^{\sigma^{1/2},+}_{p^{n}})$, and one may try to understand explicitly the relations between them. We will come back to this question in future work.
\end{Remark}

\bibliography{bibliography.bib}

\end{document}